\documentclass[11pt,twoside,a4paper]{amsart}

\usepackage{graphicx}
\usepackage{enumerate}
\usepackage{wrapfig}
\usepackage{pgfplots}
\usepackage{amsfonts}
\usepackage{amsmath}
\usepackage{amsthm}
\usepackage[utf8]{inputenc}
\usepackage{xy} \xyoption{all}
\usepackage{setspace}
\usepackage[]{mathtools}
\usepackage{amsmath}
\usepackage{amsthm}
\usepackage{amssymb}
\usepackage{mathabx}
\usepackage{multicol}
\usepackage{calc}
\usepackage{lipsum}
\usepackage{colortbl}
\usepackage{hyperref}
\usepackage{geometry}
\usepackage{tikz-cd}

\theoremstyle{plain}
\newtheorem{theorem}{Theorem}

\newtheorem{corollary}[theorem]{Corollary}

\newtheorem{proposition}[theorem]{Proposition}
\newtheorem{lemma}[theorem]{Lemma}

\theoremstyle{definition}
\newtheorem{definition}[theorem]{Definition}
\newtheorem{example}[theorem]{Example}

\newtheorem{notation}[theorem]{Notation}

\newtheorem{construction}[theorem]{Construction}

\theoremstyle{remark}
\newtheorem{remark}[theorem]{Remark}

\numberwithin{theorem}{section}

\usetikzlibrary{arrows,automata,positioning}
\usetikzlibrary{decorations.markings}
\usetikzlibrary{shapes,arrows}
\usetikzlibrary{positioning}
\usetikzlibrary{arrows}
\usetikzlibrary{calc,shapes}
\usetikzlibrary{matrix, fit}
\usetikzlibrary{backgrounds}

\tikzstyle{vertex}=[circle, draw, inner sep=0pt, minimum size=6pt]

\tikzset{style green/.style={
		set fill color=green!50!lime!60,
		set border color=white,
	},
	style cyan/.style={
		set fill color=cyan!90!blue!60,
		set border color=white,
	},
	style orange/.style={
		set fill color=orange!30,
		set border color=white,
	},
	style blue/.style={
	set fill color=blue!20,
	set border color=blue!20,
	},
	hor/.style={
		above left offset={-0.100,0.31},
		below right offset={0.15,-0.125},
		#1
	},
	ver/.style={
		
	above left offset={-0.1,0.3},
		below right offset={0.15,-0.15},
		#1
	}
}

\title{Cohn-Leavitt path algebras of bi-separated graphs}
\author{Mohan. R and B. N. Suhas}

\address{Statistics and Mathematics Unit, Indian Statistical Institute Bangalore, India}
\email{rmohan689@gmail.com}
\address{Department of Mathematics, Amrita School of Engineering, Bangalore, Amrita Vishwa Vidyapeetham, India}
\email{chuchabn@gmail.com}

\begin{document}
\begin{abstract}
The purpose of this paper is to provide a common framework for studying various generalizations of Leavitt algebras and Leavitt path algebras. This paper consists of two parts. In part I we define Cohn-Leavitt path algebras of a new class of graphs with an additional structure called bi-separated graphs, which generalize the constructions of Leavitt path algebras of various types of graphs. We define and study the category \textbf{BSG} of bi-separated graphs with appropriate morphisms so that the functor which associates a bi-separated graph to its Cohn-Leavitt path algebra is continuous. We also characterize a full subcategory of \textbf{BSG} whose objects are direct limits of finite complete subobjects. We compute normal forms of these algebras and apply them to study some algebraic theoretic properties in terms of bi-separated graph-theoretic properties. 

In part II we specialize our attention to Cohn-Leavitt path algebras of a special class of bi-separated graphs called B-hypergraphs. We investigate their non-stable K-theory and show that the lattice of order-ideals of V-monoids of these algebras is determined by bi-separated graph-theoretic data. Using this information we study representations of Leavitt path algebras of regular hypergraphs and also find a matrix criterion for Leavitt path algebras of finite hypergraphs to have IBN property.
\end{abstract}
\maketitle

\section{Introduction}\label{Introduction}

A unital ring $R$ is said to have \textit{Invariant Basis Number}  (IBN) property if whenever the free left modules $R^m$ and  $R^n$ are isomorphic for some natural numbers $m$ and $n$, then $m=n$. In a series of papers \cite{MR0077520, MR0083986, MR0132764, MR0178018}, W.G. Leavitt studied rings which do not satisfy IBN property. For natural numbers $m,n$ with $m<n$, a non-IBN ring $R$ is said to have \textit{module type} $(m,n)$ if  $m$ and $n$ are the least positive integers such that $R^m\cong R^n$. For any arbitrary field $K$ and natural numbers $m,n$ with $m<n$, Leavitt constructed `universal' $K$-algebras $L_K(m,n)$ of module type $(m,n)$ (now called \textit{Leavitt algebras}).

$L_K(m,n)$ is presented as a unital $K$-algebra with generators $x_{ij},x_{ij}^\ast$ where $1\leq i\leq m, 1\leq j\leq n$, and relations 
$$\sum\limits_{j=1}^nx_{lj}x_{jk}^\ast=\delta_{lk}\quad\text{and}\quad\sum\limits_{i=1}^mx_{li}^\ast x_{ik}=\delta_{lk}.$$
Leavitt also proved that $L_K(m,n)$ is simple if and only if $m=1$, and that $L_K(m,n)$ is a domain for all $m>1$.

In a seemingly unrelated development, J. Cuntz constructed and studied in \cite{MR0467330} and \cite{MR604046}, a class of $C^\ast $-algebras, now known as \textit{Cuntz algebras} $\mathcal{O}_n$, which are generated by $n$ isometries $S_1,S_2,\dots,S_n$ such that $\sum_{i=1}^{n}S_iS_i^\ast =1$. It turns out that the Leavitt algebra $L_\mathbb{C}(1,n)$ of type $(1,n)$ is a dense $\ast$-subalgebra of the Cuntz algebra $\mathcal{O}_n$. Brown \cite{MR637007}, and McClanahan \cite{MR1172034,MR1245464}, studied $C^\ast$-algebras $U_{(m,n)}^{nc}$ which are the $C^\ast$-analogs of $L_K(m,n)$.

Cuntz and Krieger \cite{MR561974} generalized the construction of $\mathcal{O}_n$ by considering a class of $C^\ast$-algebras, known as \textit{Cuntz-Krieger algebras} associated to an $n\times n$ matrix with entries in $\{0,1\}$. In a subsequent development, in \cite{MR1626528}, Cuntz-Krieger algebras were realized as special cases of a broader class of $C^\ast$-algebras arising from directed graphs called \textit{graph $C^\ast$-algebras}. The interested reader is referred to \cite{MR2135030} for further information on this important class of $C^\ast$-algebras. 

Algebraic analogs of graph $C^\ast$-algebras were defined and studied for row-finite graphs independently in \cite{MR2172342} and \cite{MR2310414}, and for arbitrary graphs in \cite{MR2417402} and  \cite{MR2363133}, under the name \textit{Leavitt path algebras} $L_K(E)$. These algebras generalize Leavitt algebras $L_K(1,n)$ in a similar way as the graph $C^\ast$-algebras generalize Cuntz algebras $\mathcal{O}_n$. The interested reader may consult the book \cite{MR3729290} and the references provided there.  

In \cite{MR2826405} and \cite{MR2980456}, Ara and Goodearl initiated the study of a much larger class of algebras called \textit{Cohn-Leavitt path algebras} and their  $C^\ast$-analogs based on the concept of \textit{separated graphs} $(E,C)$. The Cohn-Leavitt path algebras extend the existing version of Leavitt path algebras and graph $C^\ast$-algebras for a particular choice of $C$. It was shown that any free product of algebras $L_K(1,n)$ appears as $L_K(E,C)$ for a suitable separated graph $(E,C)$. Also, for any $n\geq m\geq 1$, there is a separated graph $(E,C)$ such that $L_K(E,C)\cong M_{m+1}(L_K(m,n))\cong M_{n+1}(L_K(m,n))$ and the corner of $L_K(E,C)$ corresponding to one vertex of $E$ is isomorphic to $L_K(m,n)$ itself. Similarly, free products of Cuntz algebra $\mathcal{O}_n$ and matrix $C^\ast$-algebras $M_{m+1}(U_{(m,n)}^{\text{nc}})$ appear as $C^\ast(E,C)$ for suitable separated graphs $(E,C)$. However, their motivation to study Cohn-Leavitt path algebras of separated graphs was of $K$-theoretic nature and towards developing techniques to answer realization problem for von Neumann regular rings (cf. \cite{MR2513205}). Using Bergman's machinery \cite{MR0357503}, they describe the $\mathcal{V}$-monoid of $L_K(E,C)$ and study the lattice of trace ideals of Cohn-Leavitt path algebras in terms of separated graph theoretic data.

Independently, in \cite{MR3096577}, Hazrat defined the concept of \textit{weighted Leavitt path algebras} $L_K(E,w)$ of weighted graphs as a graph theoretic generalization of Leavitt algebras of type $(m,n)$ for any natural numbers $m,n$. The Gr\"obner–Shirshov bases of $L_K(E,w)$ (called normal forms) were found, and as an  application, the characterization of $L_K(E,w)$  which are domains was studied in \cite{MR3707905}. As another application of normal forms, Gelfand-Kirillov dimensions of weighted Leavitt path algebras were studied in \cite{preusser2018weighted}. However, Cohn-Leavitt path algebras and weighted Leavitt path algebras are not special cases of each other.

In \cite{preusser2019leavitt}, Raimund defined and studied Leavitt path algebras associated to \textit{hypergraphs}. He showed that Leavitt path algebras of separated graphs and vertex weighted Leavitt path algebras of row-finite vertex-weighted graphs are examples of hypergraphs. 

In this article, we initiate the study of  bi-separated graphs $\dot E=(E,(C,S),(D,T))$ and the associated Cohn-Leavitt path algebras $\mathcal{A}_K(\dot E)$. With a particular choice of $(C,S)$ and/or $(D,T)$, we show that the Leavitt path algebras, Cohn-Leavitt path algebras of separated graphs, weighted Leavitt path algebras and Leavitt path algebras of hypergraphs are special cases of $\mathcal{A}_K(\dot E)$. 

In section \ref{section: preliminaries} we recall some preliminaries on graphs and their Leavitt path algebras. We also recall the definitions of various generalizations of Leavitt path algebras and some of their properties. The part I of this paper begins with section \ref{The Algebras}, where we define the Cohn-Leavitt path algebras of bi-separated graphs and state some very basic results
that follow from the definitions. We also show how the various generalizations of Leavitt path algebras introduced in the previous section are special cases of $\mathcal{A}_K(\dot E)$. In section \ref{Category}, we define the category \textbf{BSG} of bi-separated graphs and show that every object in this category is a direct limit of countable complete sub-objects (see Proposition \ref{prop: direct limit of countable}). However, this statement does not hold if we replace countable by finite. We then define a new sub-category of \textbf{BSG}, which we call ``tame category \textbf{tBSG}" and show that this category characterizes all objects of \textbf{BSG} which are direct limits of finite complete sub-objects. Section \ref{normal forms} deals with computation of normal forms of $\mathcal{A}_K(\dot E)$ using Bergman's diamond lemma and some of their applications. In particular we find bi-separated graph theoretic conditions to study algebraic properties of Cohn-Leavitt path algebras such as simplicity,  semiprimitivity, von Neumann regularity, finiteness etc and also characterize the algebras which are domains.

In part II of this paper we focus our attention to the study of B-hypergraphs. In section \ref{section: B-hypergraphs} we define B-hypergraphs $(\dot{E},\Lambda)$ and their $H$-monoids and we show that $H$-monoids are isomorphic to the $\mathcal{V}$-monoids of the corresponding Cohn-Leavitt path algebras. In section \ref{section: ideal lattice}, we introduce the partially ordered set of admissible triples AT$(\Dot{E},\Lambda)$ for each B-hypergraph $(\Dot{E},\Lambda)$ and show that this poset is a lattice. We further show that the lattice of order-ideals of $H$-monoid of $(\Dot{E},\Lambda)$ is isomorphic to the lattice AT$(\Dot{E},\Lambda)$, which establishes that AT$(\Dot{E},\Lambda)$ is isomorphic to the complete lattice of trace ideals of Cohn-Leavitt path algebra of $(\Dot{E},\Lambda)$. In section \ref{sec: reps of LPA of hypergraphs} we study the representations of Leavitt path algebras of regular hypergraphs and show that the category of unital right modules of these algebras is a full subcategory and a retract of quiver representations of underlying graphs of the hypergraphs. Also, we give a characterization of Leavitt path algebras of regular hypergraphs having a finite dimensional representation in terms of their $H$-monoids. Finally in section \ref{sec: lpa of hypergraphs and IBN} we provide a matrix criteria for a Leavitt path algebra of a finite hypergraph having invariant basis number. 

\begin{notation}\label{notation_2.1}
	Throughout this paper, $K$ denotes a fixed field; $\mathbb{Z}$ denotes the set of integers; $\mathbb{Z}^+$ denotes the set of non-negative integers; $\mathbb{N}$ denotes the set of positive integers. $\delta$ is Kronecker delta (i.e $\delta_{ij}=0$ if $i\neq j$ and $\delta_{ij}=1$ if $i=j$). By a ring (resp. $K$-algebra) we mean an associative (not necessarily commutative or unital) ring (resp. $K$-algebra).
\end{notation}

\section{Preliminaries}\label{section: preliminaries}

In this section we recall some preliminary definitions and  propositions and fix some conventions which will be used throughout the article.

A \textit{graph} (or \textit{ quiver})  $E=(E^0,E^1,r,s)$ consists of two sets $E^0, E^1$  called the \textit{set of vertices} and the \textit{set of edges} respectively, and two functions $r,s:E^1\rightarrow E^0$ called the \textit{range map} and the \textit{source map} respectively. We place no restriction on the cardinalities of $E^0$ and $E^1$ or on the properties of the functions $r$ and $s$. We say a graph is \textit{finite} if both $E^0$ and $E^1$ are finite. A vertex $v\in E^0$ is called a \textit{source} (resp. \textit{sink}) if $r^{-1}(v)=\emptyset$ (resp. $s^{-1}(v)=\emptyset$). 

A \textit{subgraph} $F=(F^0,F^1,r_F,s_F)$ of $E=(E^0,E^1,r_E,s_E)$ is defined by $F^0\subseteq E^0$, $F^1\subseteq E^1$ and $r_F$ is the restriction of $r_E$ on $F^1$ and $s_F$ is the restriction of $s_E$ on $F^1$. Let $V$ be a subset of $E^0$. The \textit{induced subgraph} on $V$ is the subgraph $E_V=(V,E_V^1,r_V,s_V)$ such that $E_V^1:=s^{-1}(V)\cap r^{-1}(V)$, $r_V$ and $s_V$ are restrictions of $r_E$ and $s_E$ on $E_V^1$ respectively. A subgraph is \textit{full} if it is induced on its set of vertices.



A \textit{graph morphism} $$\phi:F=(F^0,F^1,r_F,s_F)\rightarrow E=(E^0,E^1,r_E,s_E)$$ is a pair of maps $\phi^0:F^0\rightarrow E^0$ and $\phi^1:F^1\rightarrow E^1$ such that $r_E(\phi^1(e))=\phi^0(r_F(e))$ and $s_E(\phi^1(e))=\phi^0(s_F(e))$, for every $e\in F^1$.

We denote the category of graphs along with graph morphisms by \textbf{Gra}. Given a family of graphs $\{E_i\}_{i\in I}$ in \textbf{Gra}, we define their disjoint union $\bigsqcup\limits_{i\in I}E_i$ to be the graph whose vertex set is $\bigsqcup\limits_{i\in I}E_i^0$, edge set is $\bigsqcup\limits_{i\in I}E_i^1$, and the source and range maps are trivial extensions of $s_i$ and $r_i$ respectively for all $i\in I$.

A \textit{path} $\mu$ in a graph $E$ is either a vertex $v\in E^0$ or a finite sequence of edges $\mu=e_1e_2\dots e_n$ such that $r(e_i)=s(e_{i+1})$, for $i=1,\dots,n-1$. The set of all paths in $E$ is denoted by $E^\star$. We define the length function $l:E^\star\rightarrow\mathbb{Z}^+$ by
\begin{center}
	\[
    l(\mu) =
	\begin{cases}
	0, &\quad\text{if}\quad \mu=v\in E^0,\\
	n, &\quad\text{if}\quad \mu=e_1e_2\dots e_n.\\
	\end{cases}
	\]
\end{center}
We denote the set of all paths in $E$ of length $n$ by $E^n$, and hence $E^\star=\bigcup_{n\geq 0}E^n$. The source and range functions $s,r$ can be extended to $E^\star$ as follows:  
$$\text{if}~v\in E^0,~ s(v):=v~\text{and}~r(v):=v,$$
$$\text{if}~\mu=e_1e_2\dots e_n, ~s(\mu):=s(e_1)~\text{and}~r(\mu):=r(e_n).$$

The \textit{(free) path category} $\mathcal{C}_E$ of a graph $E$ is the small category with Ob($\mathcal{C}_E):=E^0$ and for $v,w\in E^0$,  Mor$(v,w):=\{\mu\in E^\star\mid s(\mu)=v,r(\mu)=w\}$. In other words, the elements of $\mathcal{C}_E$ are paths in $E$ and the partial multiplication is defined by path concatenation. The \textit{path $0$-semigroup} $S^0(E)$ of a graph $E$ is the set $E^\star\sqcup\{0\}$ along with multiplication defined by extension of partial multiplication of $\mathcal{C}_E$ by $0$. That is 
\[
    \mu\cdot\nu :=
	\begin{cases}
	0, &\quad\text{if}\quad\text{either}~\mu=0~\text{or}~\nu=0,\\
	0, &\quad\text{if}\quad r(\mu)\neq s(\nu),\\
	\mu\nu, &\quad\text{if}\quad r(\mu)=s(\nu).
	\end{cases}
	\]

\begin{definition}\label{path_algebra}
	Let $E$ be a graph. The \textit{Path $K$-algebra of $E$}, denoted by $K(E)$, is defined to be the quotient of the free associative $K$-algebra generated by $E^0\cup E^1$ modulo the following relations:
	\begin{enumerate}
		\item $vw=\delta_{vw}v$, for all $v,w\in E^0$,
		\item $s(e)e=e r(e)=e$, for all $e\in E^1$.
	\end{enumerate}
\end{definition}

In other words, the path $K$-algebra of $E$ is obtained as the contracted $K$-algebra of the graph $0$-semigroup $S^0(E)$ (i.e, the zero of $K(E)$ and $S^0(E)$ are identified). The following proposition follows from an application of Bergman's diamond lemma \cite{MR506890}.

\begin{proposition}\label{prop: basis of a path algebra}
Let $E$ be a graph. Then $E^\star$ is a linear $K$-basis for $K(E)$.
\end{proposition}

We recall that an associative ring $R$ is said to have a \textit{set of local units} $U$ if $U$ is a set of idempotents in $R$ having the property that, for each finite subset $r_1,\ldots,r_n$ of $R$, there exists a $u \in U$ for which $ur_iu = r_i$, for $1 \leq i \leq n$. Also, an associative ring $R$ is said to have \textit{enough idempotents} if there exists a set of nonzero orthogonal idempotents $I$ in $R$ for which the set $F$ of finite sums of distinct elements of $I$ is a set of local units for $R$. We denote a ring with ring with enough idempotents by $(R,I)$. For any graph $E$ note that $(K(E),E^0)$ is a $K$-algebra with enough idempotents. Moreover, $K(E)$ is unital if and only if $E^0$ is finite in which case $\sum_{v\in E^0}v$ is the unit.

By \textbf{K-Alg} we mean the category whose objects are $K$-algebras with enough idempotents and whose morphisms are $K$-algebra morphisms which map local units to local units. We note that $K(\_)$ is \textit{not} a functor from \textbf{Gra} to \textbf{K-Alg}. This is because a graph morphism $\phi:F\rightarrow E$ can map two distinct vertices $v,w\in F^0$ to a same vertex in $E$, in which case $K(\phi)(vw)\neq 0$ in $K(E)$, but $vw=0$ in $K(F)$. However, if \textbf{Gr} denotes the category whose objects are graphs and morphisms are graph morphisms $\phi=(\phi^0,\phi^1)$  such that $\phi^0$ is injective. Then it is easy to verify that $K(\_)$ is a \textit{continuous functor} from \textbf{Gr} to \textbf{K-Alg} (i.e, $K(\_)$ maps direct limits to direct limits).

\begin{definition}\label{double_graph}
	Given a graph $E$, the \textit{double of $E$}, denoted by $\widehat{E}$, is defined to be the graph $(E^0,E^1\sqcup\overline{E^1} ,\widehat{r},\widehat{s})$, where $\overline{E^1} =\{e^\ast \mid e\in E^1\}$, and the functions $\widehat{r}$ and $\widehat{s}$ are such that
	\begin{center}
		$\widehat{r}|_{E^1}=r, \widehat{s}|_{E^1}=s$, $\widehat{r}(e^\ast )=s(e)$ and $\widehat{s}(e^\ast )=r(e)$, for all $e\in E^1$.
	\end{center}
\end{definition}

A path $\mu\in\widehat{E}^\star$ is called a \textit{generalized path}. We say a graph $E$ is connected if its double $\widehat{E}$ is connected. That is, for any $v,w\in E^0$, there is a generalized path $\mu\in \widehat{E}^\star$ such that $s(\mu)=v$ and $r(\mu)=w$. The \textit{connected components} of $E$ are the graphs $\{E_j\}_{j \in J}$ such that $E = \bigsqcup \limits_{j \in J} E_j$, where every $E_j$ is connected. 

A $K$-algebra $R$ is called a \textit{$G$-graded algebra} if $R=\bigoplus\limits_{g\in G}R_g$, where $G$ is a group and each $R_g$ is a $K$-subspace of $R$ and $R_{g}R_{h}\subseteq R_{gh}$ for all $g,h\in G$. The set $R^H=\bigcup_{g\in G}R_g$ is called the set of \textit{homogeneous elements} of $R$. $R_g$ is called the $g$-component of $R$ and the nonzero elements of $R_g$ are called \textit{homogeneous of degree $g$}. We write deg$(r)=g$ if $r\in A_g-\{0\}$. Note that for any graph $E$ we can define $\mathbb{Z}$-grading on $K(\widehat{E})$ by defining deg($v)=0$ for every $v\in E^0$, deg($e)=1$ and deg($e^\ast)=-1$.

A $\ast$-ring is an associative unital ring $R$ with an antiautomorphism $\ast:R\rightarrow R$ that is also an involution. That is, $\ast$ satisfies the following properties: for every $x,y\in R$, $(x+y)^\ast=x^\ast+y^\ast$, $(xy)^\ast=y^\ast x^\ast$, $(x^\ast)^\ast=x$ and $1^\ast=1$. Let $K$ be a $\ast$-field with involution $\overline{\phantom{A}}:K\rightarrow K$. A $\ast$-algebra is a $K$-algebra $R$ that is also a $\ast$-algebra such that $(kr)^\ast=\overline{k}r^\ast$ for every $k\in K$ and $r\in R$. Note that for any graph $E$ and a $\ast$-field $K$ with involution $\overline{\phantom{A}}:K\rightarrow K$, $K(\widehat{E})$ is a $\ast$-algebra with respect to the involution defined (on the generators) by $(v)\mapsto v$ for every $v\in E^0$, $e\mapsto e^\ast$ for every $e\in E^1$ and $e^\ast\mapsto e$ for each $e^\ast\in\overline{E^1}$.

\subsection{Leavitt path algebras and their generalizations}\hfil\\

In this subsection we recall the definitions of Leavitt path algebras and their various generalizations.

Recall that a vertex $v\in E^0$ is called \textit{regular} if $0<|s^{-1}(v)|<\infty$. The set of all row regular vertices in $E$ is denoted by Reg$(E)$.

\begin{definition}[\textbf{Cohn-Leavitt path algebras of graphs}]\label{def: CLPA of a graph}
	Let $E$ be a graph and $S$ be any subset of row regular vertices Reg$(E)$ of $E$. The \textit{Cohn-Leavitt path algebra} $CL_K^S(E)$ of $E$ \textit{relative to} $S$ is obtained from $P_K(\widehat{E})$ by imposing the following (Cuntz-Krieger) relations:

	\begin{enumerate}
		\item[CK1:] $\quad e^\ast f=\delta_{e,f}r(e)$ for every $e,f\in E^1$,
		\item[CK2:] $\quad v=\sum\limits_{e\in s^{-1}(v)}ee^\ast $ for every $v\in S$.
	\end{enumerate}

In particular, $CL_K^{\emptyset}(E)$ is called \textit{Cohn path algebra} of $E$ and denoted by $C_K(E)$. The algebra $CL_K^{\text{Reg}(E)}(E)$ is called \textit{Leavitt path algebra} of $E$ and denoted by $L_K(E)$.
\end{definition}

\begin{definition}
[\textbf{Cohn-Leavitt path algebras of separated graphs}]\label{defn: Cohn-Leavitt path algebras of separated graphs}
	A \textit{separated graph} is a pair $(E,C)$, where $E$ is a graph and $C=\bigsqcup_{v\in E^0}C_v$, $C_v$ being a partition of $s^{-1}(v)$, for each vertex $v\in E^0$. Let $C_{\text{fin}} = \{X \in C \mid |X| < \infty\}$ and $S\subset C_{\text{fin}}$. Then the \textit{Cohn-Leavitt path algebra} $CL_K(E,C,S)$ of $(E,C)$ \textit{relative to} $S$ is defined as the quotient of $P_K(\widehat{E})$ obtained by imposing the following relations:
	
		\begin{enumerate}
			\item[SCK1:] $\quad e^\ast f=\delta_{e,f}r(e)$, for all $e,f\in X, X\in C$,
			\item[SCK2:] $\quad v=\sum_{e\in X}ee^\ast$, for every $X\in S, v\in E^0$.
		\end{enumerate}
		
	If $S=C_{\text{fin}}$, the Cohn-Leavitt path algebra is simply called Leavitt path algebra of $(E,C)$ and denoted by $L_K(E,C)$.
\end{definition}
\begin{definition}[\textbf{Weighted Leavitt path algebras}]\label{defn: WLPA of wgraph}
Let $E$ be a row-finite graph and $w:E^1\rightarrow \mathbb{N}$ be a (weight) function. We define the (associated) \textit{weighted graph} $E_{w}$ to be a graph $(E_{w}^0,E_{w}^1,r_{w},s_{w})$, where $E_{w}^0=E^0$, $E_{w}^1=\{e_i\mid e\in E^1, 1\leq i\leq w(e)\}$, $r_{w}(e_i)=r(e)$ and $s_{w}(e_i)=s(e)$, for all $e\in E^1, 1\leq i\leq w(e)$. Set $w(v) = \max\{w(e)\mid e\in s^{-1}(v)\}$. (Note that $w(v)$ is well-defined since $E$ is row-finite). The \textit{weighted Leavitt path algebra} $WL_K(E_{w})$ of the weighted graph $E_{w}$ is the quotient of $P_K(\widehat{E}_{w})$ obtained by going modulo the following relations:
	\begin{enumerate}		
		\item[wCK1:] $\quad\quad \sum\limits_{1\leq i\leq w(v)}e_i^\ast f_i=\delta_{e,f}r(e)$, for every $e,f\in E^1$ and $s(e)=s(f)=v\in E^0$,
		\item[wCK2:] $\quad\quad \sum\limits_{e\in s^{-1}(v)}e_ie_j^\ast =\delta_{i,j}v$, for each $v\in E^0, 1\leq i,j\leq w(v)$,
	\end{enumerate}
	where we set $e_i$ and $e_i^\ast $ to be zero whenever $i>w(e)$.
\end{definition}

\begin{definition}[\textbf{Leavitt path algebras of Hypergraphs}]\label{defn: LPA of hypergraph}
Let $I$ and $X$ be sets. Recall that a function $x:I\rightarrow X$, given by $i\mapsto x_i=x(i)$ is called a family of elements in $X$ indexed by $I$. We denote a family $x$ of elements in $X$ indexed by $I$ by $(x_i)_{i\in I}$. 

A \textit{hypergraph} is a quadruple $\mathcal{H}=(\mathcal{H}^0,\mathcal{H}^1,s,r)$ where $\mathcal{H}^0$ and $\mathcal{H}^1$ are sets called the set of vertices and the set of hyperedges respectively. For each $h\in\mathcal{H}^1$ there exists a pair of non-empty indexing sets $I_h,J_h$ such that  $s(h):I_h\rightarrow\mathcal{H}^0$, and $r(h):J_h\rightarrow\mathcal{H}^0$ are families of vertices. 

Let $\mathcal{H}$ be a hypergraph. A hyperedge $h\in\mathcal{H}^1$ is called \textit{source regular} (resp. \textit{range regular}) if $I_h$ is finite (resp. $J_h$ is finite). The set of all source regular hyperedges of $\mathcal{H}$ is denoted by $\mathcal{H}_{\text{sreg}}^1$ and the set of all range regular hyperedges of $\mathcal{H}$ is denoted by $\mathcal{H}_{\text{rreg}}^1$. The hypergraph $\mathcal{H}$ is said to be \textit{regular} if $\mathcal{H}^1=\mathcal{H}_\text{sreg}^1=\mathcal{H}_\text{sreg}^1$.

The \textit{Leavitt path algebra $L_K(\mathcal{H})$ of the hypergraph $\mathcal{H}$} is the $K$-algebra presented by the generating set $\{v,h_{ij},h_{ij}^\ast \mid v\in \mathcal{H}^0, h\in \mathcal{H}^1, i\in I_h, j\in J_h\}$ and the relations
	\begin{enumerate}
		\item $uv=\delta_{u,v}u$, for every $u,v\in \mathcal{H}^0$,
		\item $s(h)_ih_{ij}=h_{ij}=h_{ij}r(h)_j$ and $ r(h)_jh_{ij}^\ast =h_{ij}^\ast =h_{ij}^\ast s(h)_i$, for every $h\in \mathcal{H}^1$, $i\in I_h$, and $j\in J_h$
		\item $\sum\limits_{j\in J_h}h_{ij}h_{kj}^\ast =\delta_{ik}s(h)_i$, for every $h\in \mathcal{H}_\text{sreg}^1$ and $i,k\in I_h$,
		\item $\sum\limits_{i\in I_h}h_{ij}^\ast h_{ik}=\delta_{jk}r(h)_j$, for every $h\in \mathcal{H}_\text{rreg}^1$ and $1\leq j,k\in J_h$.
	\end{enumerate}
\end{definition}

\begin{remark}
Let $L$ denote any one of the $K$-algebras appearing in the definitions \ref{def: CLPA of a graph}-\ref{defn: LPA of hypergraph}. Then note that $L$ satisty the following properties.
\begin{enumerate}
     \item The algebra $L$ is unital if and only if the set of vertices $V$ in the underlying graph is finite. In this case, the unit is the sum of vertices. In general $(L,V)$ is a $K$-algebra with enough idempotents.
     \item If $\overline{\phantom{A}} : K \rightarrow K$ is an involution on the field $K$, then $L$ is a $\ast$-algebra (with respect to the involution $\ast : L \rightarrow L$).
     \item $L$ is a graded quotient algebra of $K(\widehat{E})$ with respect to standard $\mathbb{Z}$-grading given by length of paths.
 \end{enumerate}
\end{remark}

\begin{center}
	\textbf{PART I}
\end{center}
\section{Cohn-Leavitt path algebras of bi-separated graphs}\label{The Algebras}

\begin{definition}\label{biseparated graph}
	A \textit{bi-separated graph} is a triple $\dot E=(E,C,D)$ such that
	\begin{enumerate}
		\item $E=(E^0,E^1,r,s)$ is a graph,
		\item $C=\bigsqcup_{v\in E^0}C_v$, where $C_v$ is a partition of $s^{-1}(v)$ for every non-sink $v\in E^0$,
		\item $D=\bigsqcup_{v\in E^0}D_{v}$, where $D_{v}$ is a partition of $r^{-1}(v)$ for every non-source $v\in E^0$,
		\item $|X\cap Y|\leq 1$, for every $X\in C$ and $Y\in D$,
	\end{enumerate}
\end{definition}
In the above definition, $C$ is called row-separation of $E$, $D$ is called column-separation of $E$ and $(C,D)$ is called bi-separation of $E$. The elements of $C$ are called rows and the elements of $D$ are called columns. Let $C_{\text{fin}}:=\{X\in C\mid |X|<\infty\}$ and $D_{\text{fin}}:=\{Y\in D\mid |Y|<\infty\}$. A bi-separated graph $\dot{E}$ is called \textit{finitely row-separated} (resp. \textit{finitely column-separated}) if $C=C_{\text{fin}}$ (resp. if $D=D_{\text{fin}}$) and is called \textit{finitely bi-separated} if both $C = C_{\text{fin}}$ and $D = D_{\text{fin}}$.  

In the above definition we follow the convention that if $S$ is a set, by a partition $P$ of $S$ we mean a family of pairwise disjoint nonempty subsets of $S$, whose union is $S$. For any non-empty set $S$, there always exist two trivial partitions: the partition $P_1$ on $S$, called the \textit{discrete} partition, where every element of $P_1$ is a singleton and the partition $P_S$, called the \textit{full} partition, where $S$ is the only element of $P_S$. 

\begin{example}(\textbf{Standard bi-separation of a simple graph})
Let $E$ be a \textit{simple} graph (that is, $e,f\in E^1$ such that $s(e)=s(f)$ and $r(e)=r(f)$ implies $e=f$. In other words, there are no multiedges between any two vertices). We can obtain a canonical  bi-separation on $E$ by considering both $C_v$ and $D_v$ to be full partitions. In other words, by defining $C_v=\{s^{-1}(v)\}$ for every non-sink $v\in E^0$ and $D_v=\{r^{-1}(v)\}$ for every non-source $v\in E^0$. This bi-separation is called standard.

\end{example}

In the following examples, $E$ denotes an arbitrary graph.
\begin{example}(\textbf{Trivial bi-separation of a graph})
By  trivial bi-separation on a graph $E$, we mean both $C_v$ and $D_v$ are discrete partitions. 
\end{example}

\begin{example}\textbf{(Cuntz-Krieger bi-separation of a graph)}\label{Cuntz-Krieger bi-separation}
We can obtain another canonical bi-separation on $E$ by combining full row-separation and discrete column-separation on $E$ as follows: Consider $C_v=\{s^{-1}(v)\}$ for every non-sink $v\in E^0$ and $D_v=\{\{e\}\mid e\in r^{-1}(v)\}$ for every non-source $v\in E^0$.
\end{example}

\begin{example}\label{separated graphs}\textbf{(Separated graphs)}
A bi-separated graph $(E,C,D)$ in which the column-separation is discrete is called a \textit{row-separated graph} or simply \textit{separated graph} (cf. definition \ref{defn: Cohn-Leavitt path algebras of separated graphs}). Separated graphs are denoted by $(E,C)$.
\end{example}

\begin{example}\textbf{(Weighted graphs)} \label{Weighted bi-separation}
Let $E$ be a row-finite graph and $w:E^1\rightarrow\mathbb{N}$ be a weight map on $E$. Consider the weighted graph $E_w=(E_w^0,E_w^1,r_w,s_w)$. We associate a  bi-separation on $E_w$ as follows: For every non-sink $v\in E^0$ and $1\leq i\leq w(v)$ define $X_v^i:=\{e_i\mid e\in s^{-1}(v)\}$. For every $v\in E^0$ and $e\in s^{-1}(v)$ define $Y_v^e:=\{e_i\mid 1\leq i\leq w(e)\}$. Now consider $C_v:=\{X_v^i\mid 1\leq i\leq w(v)\}$ and $D_v:=\bigsqcup_{w\in E^0}D_{vw}$ where $D_{vw}=\{Y_v^e\mid e\in s^{-1}(v)\cap r^{-1}(w)\}$. Here $C=C_{\text{fin}}$, since $E$ is row-finite and $D=D_{\text{fin}}$, since $w$  takes natural numbers as values.
\end{example}

\begin{example}\label{hyper bi-separation}(\textbf{Hypergraphs})
We show that any hypergraph $\mathcal{H}$ can be associated to bi-separated graph $\dot{E}=(E,C,D)$ as follows: Define $E = (\mathcal{H}^0,E^1,s^\prime,r^\prime)$, where $E^1=\{h_{ij}\mid h\in\mathcal{H}^1, i\in I_h, j\in J_h\}$, $s^\prime(h_{ij})=s(h)_i$ and $r^\prime(h_{ij})=r(h)_j$. For an arbitrary $h\in\mathcal{H}^1$, if $i\in I_h$ then $X_h^i := \{h_{ij}\mid j\in J_h\}$ and if $j\in J_h$ then $Y_h^j := \{h_{ij}\mid i\in I_h\}$. For $v\in E^0$, define $C_v = \{X_h^i\mid h\in \mathcal{H}^1, v\in s(h)\}$ and $D_v = \{Y_h^j\mid h\in \mathcal{H}^1, v\in r(h)\}$. By construction, $C=C_{\text{fin}}$ and $D=D_{\text{fin}}$.
\end{example}

\begin{notation}
Given a bi-separated graph $\dot E=(E,C,D)$, the maps $s$ and $r$ can be extended to $C$ and $D$ respectively in well-defined manner as follows: For $X\in C$, define $s(X):=s(e)$ where $e\in X$ and for $Y\in D$, define $r(Y):=r(e)$ where $e\in Y$. 

Also, for each $X\in C$ and $Y\in D$ we set
\begin{center}
	\[
	XY = YX =
	\begin{cases}
	e, &\quad\text{if}\quad X\cap Y=\{e\},\\
	0, &\quad\text{otherwise}.\\
	\end{cases}
	\]
\end{center}
We interchangeably use $XY$ and $X \cap Y$, wherever there is no cause for confusion.
\end{notation}

\begin{definition}\label{LPA of BSG}
	Let $\dot E=(E,C,D)$ be a bi-separated graph. The \textit{Leavitt path algebra of $\dot E$} with coefficients over $K$, denoted by $L_K(\dot E)$, is the quotient of $K(\widehat{E})$ obtained by imposing the following relations:
	\begin{enumerate}
		\item[$L1$:] for every $X,X^\prime\in C_{\text{fin}}$,
		$$\sum\limits_{Y\in D}(XY)(YX^\prime)^\ast =\delta_{XX^\prime}s(X),$$
		
		\item[$L2$:] for every $Y,Y^\prime\in D_{\text{fin}}$,
		$$\sum\limits_{X\in C}(YX)^\ast(XY^\prime) =\delta_{YY^\prime}r(Y).$$
		
	\end{enumerate}
\end{definition}

\begin{example}(\textbf{Leavitt path algebra of a standard bi-separated simple graph})

Let $E$ be a simple graph and consider the standard bi-separation $(C,D)$ on $E$. Let the set of all non-sinks of $E$ be denoted by $E_\alpha$ and the set of all non-sources be denoted by $E_\omega$. Then $|C|=|E_\alpha|$ and $|D|=|E_\omega|$. Recall that a vertex $v\in E^0$ is called row-regular if $0<|s^{-1}(v)|<\infty$ and $w\in E^0$ is called column-regular if $0<|r^{-1}(w)|<\infty$. The set of all row-regular vertices is denoted by RReg($E$) and the set of all column regular vertices is denoted by CReg($E$). Note that $|C_\text{fin}|=|\text{RReg}(E)|$ and $|D_\text{fin}|=|\text{CReg}(E)|$. 

Let $A$ be a $|C_\text{fin}|\times|D|$ matrix over $K(\widehat{E})$ with entries 
\[	A(v,w) =
	\begin{cases}
	e, &\quad\text{if}\quad e\in E^1~\text{such that}~s(e)=v,r(e)=w,\\
	0, &\quad\text{otherwise}.\\
	\end{cases}
	\]
where $v\in \text{RReg}(E)$ and $w\in E_\omega$. Let $A^\ast$ denote the `adjoint transpose' of $A$. Similiarly, let $B$ be a $|C|\times|D_\text{fin}|$ matrix over $K(\widehat{E})$ with entries 
\[	B(v,w) =
	\begin{cases}
	e, &\quad\text{if}\quad e\in E^1~\text{such that}~s(e)=v,r(e)=w,\\
	0, &\quad\text{otherwise}.\\
	\end{cases}
	\]
where $v\in E_\alpha$ and $w\in \text{CReg}(E)$. Let $B^\ast$ denote the `adjoint transpose' of $B$. 

Then the defining relations L1 and L2 of Leavitt path algebras are obtained by imposing the following matrix relations:
\begin{eqnarray*}
(L1):\quad AA^\ast &=& V,\\
(L2):\quad B^\ast B &=& U.
\end{eqnarray*}

where $V$ is the $|\text{RReg}(E)|\times|\text{RReg}(E)|$ diagonal matrix with diagonal entries $V(v,v)=s(v)$ and and $U$ is the $|\text{CReg}(E)|\times|\text{CReg}(E)|$ diagonal matrix with diagonal entries $U(w,w)=r(w)$.

In particular, if $E$ is finite simple graph, then the matrix $A$ (resp. the matrix $B$) is obtained from the adjacency matrix of $E$ by removing the zero rows (resp. zero columns) and replacing $1$'s with corresponding edges. We illustrate this with a few examples below.

\begin{enumerate}
	    \item[(1)] \label{line graph} For $n \geq 1$, let $\Sigma_n$ be the following line graph with $n$ vertices and $n-1$ edges:
	\begin{center}
			\begin{tikzpicture}
			[->,>=stealth',shorten >=1pt,thick,scale=0.6]

			\draw [fill=black] (-2,0)		circle [radius=0.1];		
			\draw [fill=black] (-4,0)		circle [radius=0.1];		
			\draw [fill=black] (-6,0)		circle [radius=0.1];	
			\draw [fill=black] (2,0)		circle [radius=0.1];	
			\draw [fill=black] (4,0)		circle [radius=0.1];

			\path[->]	(-1.8,0)	edge[] 	node	{}		(-0.2,0);
			\path[->]	(-3.8,0)	edge[] 	node	{}		(-2.2,0);
			\path[->]	(-5.8,0)	edge[] 	node	{}		(-4.2,0);
			\path[->]	(2.2,0)	edge[] 	node	{}		(3.8,0);

			\node at (0.8,0) {$\dots$};
			\node at (-6,-0.5) {$v_1$};
			\node at (-4,-0.5) {$v_2$};
			\node at (-2,-0.5) {$v_3$};
			\node at (2,-0.5) {$v_{n-1}$};
			\node at (4,-0.5) {$v_n$};
			
			\node at (-5,0.4) {$e_1$};
			\node at (-3,0.4) {$e_2$};
			\node at (-1,0.4) {$e_3$};
			\node at (3,0.4) {$e_{n-1}$};

			\node at (-8,0) {$\Sigma_n=$};	
			
			\end{tikzpicture}
		\end{center}
	Then \[A=\begin{pmatrix}
	e_1 & 0 & \dots & 0 \\
	0 & e_2 & \dots & 0\\
	\vdots & \vdots & \dots & \vdots \\
	0 & 0 & \dots & e_{n-1} \\
	\end{pmatrix},\]
	
	and the relations obtained are
	$$e_ie_i^\ast =v_i\quad\text{and}\quad e_i^\ast e_i=v_{i+1},$$\
	where $1\leq i\leq n-1$. In this case, it is easy to see that	
	 $$L_K(\dot \Sigma_n)\cong L_K(\Sigma_n)\cong M_n(K).$$ 
	 \item[(2)] \label{tri-diagonal} For $n \geq 3$, let $\Gamma_n$ denote the following graph with $n$ vertices:
	\begin{center}
		\begin{tikzpicture}
	[->,>=stealth',shorten >=1pt,thick, scale=0.6]

	\draw [fill=black] (-2,0)		circle [radius=0.1];		
	\draw [fill=black] (-4,0)		circle [radius=0.1];		
	\draw [fill=black] (-6,0)		circle [radius=0.1];	
	\draw [fill=black] (2,0)		circle [radius=0.1];	
	\draw [fill=black] (4,0)		circle [radius=0.1];

	\path[->]	(2,0.2)	edge[out=130, in=50, looseness=4,  distance=2cm, ->] 	node	{}		(2,0.2);
	\path[->]	(-2,0.2)	edge[out=130, in=50, looseness=4,  distance=2cm, ->] 	node	{}		(-2,0.2);
	\path[->]	(-4,0.2)	edge[out=130, in=50, looseness=4,  distance=2cm, ->] 	node	{}		(-4,0.2);
	\path[->]	(4.2,0)	edge[out=40, in=320, looseness=4,  distance=2cm, ->] 	node	{}		(4.2,0);
	\path[->]	(-6.2,0)	edge[out=220, in=140, looseness=4,  distance=2cm, ->] 	node	{}		(-6.2,0);

	\path[->]	(-5.8,0)	edge[bend left=40] 	node	{}		(-4.2,0);
	\path[->]	(-3.8,0)	edge[bend left=40] 	node	{}		(-2.2,0);
	\path[->]	(-1.8,0)	edge[bend left=40] 	node	{}		(-0.2,0);
	\path[->]	(2.2,0)	edge[bend left=40] 	node	{}		(3.8,0);
	
	\path[->]	(-4.2,0)	edge[bend left=40] 	node	{}		(-5.8,0);
	\path[->]	(-2.2,0)	edge[bend left=40] 	node	{}		(-3.8,0);
	\path[->]	(-0.2,0)	edge[bend left=40] 	node	{}		(-1.8,0);
	\path[->]	(3.8,0)	edge[bend left=40] 	node	{}		(2.2,0);

	\node at (0.8,0) {$\dots$};
	\node at (-6,-0.5) {$v_1$};
	\node at (-4,-0.5) {$v_2$};
	\node at (-2,-0.5) {$v_3$};
	\node at (2,-0.5) {$v_{n-1}$};
	\node at (4,-0.5) {$v_n$};

	\node at (-9,0) {$\Gamma_n=$};	
	
	\end{tikzpicture}
\end{center}
Then the adjacency matrix of $\Gamma_n$ has at least two entries and at most three entries in both rows and columns, i.e.,

\[A=\begin{pmatrix}
\ast & \ast & 0 & 0 & 0 & \dots & 0 & 0\\
\ast & \ast & \ast & 0 & 0 & \dots & 0 & 0\\
0 & \ast & \ast & \ast & 0 & \dots & 0 & 0\\
\vdots & \vdots & \vdots & \vdots & \vdots & \ddots & \vdots & \vdots\\
0 & 0 & 0 & 0 & 0 &\dots & \ast & \ast\\
\end{pmatrix},\]
where $\ast$'s are nonzero entries filled by the corresponding edges. In this case, though, the explicit description of the Leavitt path algebra $L_K(\dot{\Gamma_n})$ is not known.
	\end{enumerate}
\end{example}

\begin{example}(\textbf{Groupoid algebra of a free groupoid})
Let $E$ be a graph and consider the trivial bi-separation $(C,D)$ on $E$. The defining relations of Leavitt path algebra of $(E,C,D)$ turns the free path category of $\widehat{E}$ into a free groupoid and hence $L_K(E,C,D)$ is the groupoid algebra of this free groupoid. 
In particular if $E$ has only one vertex then with respect to trivial bi-separation, the Leavitt path algebra is the group algebra of the free group with generators $E^1$ (Here we identified the vertex with the group identity).
\end{example}

\begin{example}(\textbf{Leavitt path algebra of a graph})
Let $E$ be any graph and $\dot{E}$ be the associated bi-separated graph with respect to Cuntz-Kreiger bi-separation on $E$. Then we have $L_K(\dot{E})\cong L_K(E)$.
\end{example}

\begin{example}(\textbf{Leavitt path algebra of a separated graph})
Let $\dot{E}=(E,C)$ be a (row) separated graph. Then it is direct that $L_K(\dot{E})\cong L_K(E,C)$.
\end{example}

\begin{example}(\textbf{Weighted Leavitt path algebra of a weighted graph})
Let $E$ be a row-finite graph and $w:E^1\rightarrow \mathbb{N}$ be a weight map. Consider $\dot{E}=(E,C,D)$ where $(C,D)$ is the weighted bi-separation on $E$ as in example \ref{Weighted bi-separation}. Then it is immediate that $L_K(\dot{E})\cong WL_K(E_w)$.
\end{example}

It has been noted in \cite[page 171]{MR2980456} that neither weighted Leavitt path algebras nor Leavitt path algebras of separated graphs are particular cases of the each other. One can mix the above two examples and construct new algebras as follows:
\begin{example}[\textbf{Weighted Cohn-Leavitt path algebras of finitely separated graphs}]
	Let $(E,C)$ be a finitely row-separated graph (i.e. a separated graph in which $C=C_{\text{fin}}$). Let $w:E^1\rightarrow\mathbb{N}$ be a function and $E_{w}$ be the associated weighted graph. For $X\in C$, set $w(X) = \max\{w(e)\mid e\in X\}$. The weighted Cohn-Leavitt path algebra $CL_K(E_{w},C)$ of $(E_{w},C)$ can be defined as the quotient of $P_K(\widehat{E}_{w})$ by factoring out the following relations:
	\begin{enumerate}
		\item[wSCK1:] $\quad\sum\limits_{1\leq i\leq w(X)}e_i^\ast f_i=\delta_{e,f}r(e),$ for every $e,f\in X$ and $X\in C$, 
		\item[wSCK2:] $\quad\sum\limits_{e\in X}e_ie_j^\ast =\delta_{i,j}s(e),$ for each $X\in C, 1\leq i,j\leq w(X)$,
	\end{enumerate}
where we set $e_i$ and $e_i^\ast $ to be zero whenever $i>w(e)$.

Given a weighted finitely separated graph $(E_w,C)$, we get a canonical bi-separated graph as follows: For $X\in C$ and $1\leq i\leq w(X)$, define $\widetilde{X}^i = \{e_i\mid e\in X\}$ and set $\widetilde{C}_v=\{\widetilde{X}^i\mid X \in C_v ~\text{and}~ 1\leq i\leq w(X)\}$. Here $\widetilde{C}=\widetilde{C}_{\text{fin}}$, since $E$ is finitely separated. Now, for $e\in X$, define $\widetilde{Y}_X^e = \{e_i\mid 1\leq i\leq w(e)\}$ and  $\widetilde{D}_{vw}=\{\widetilde{Y}_X^e\mid e\in X\}$. Observe that $\widetilde{D}=\widetilde{D}_{\text{fin}}$, since $w$ is natural number valued. Now setting $\dot E=(E,\widetilde{C},\widetilde{D})$, we immediately get
$$L_K(E,\widetilde{C},\widetilde{D})\cong L_K(E_{w},C).$$
\end{example}

\begin{example}[\textbf{Leavitt path algebra of a hypergraph}]
Given any hypergraph $\mathcal{H}$, consider the associated bi-separated graph $\dot{E_H}$ as in example \ref{hyper bi-separation}. Then we have $L_K(\dot{E})\cong L_K(\mathcal{H})$.
\end{example}

\begin{definition}\label{CLPA of BSG}
	Let $\dot E=(E,C,D)$ be bi-separated graph. Let $S\subseteq C_{\text{fin}}$ and $T\subseteq D_{\text{fin}}$ be two distinguished sets. The \textit{Cohn-Leavitt path algebra of $\dot E$} with coefficients over $K$ relative to $(S,T)$, denoted by $\mathcal{A}_K(E,(C,S),(D,T))$, is the quotient of $K(\widehat{E})$ obtained by imposing the following relations:
	\begin{enumerate}
		\item[$\mathcal{A}1$:] for every $X,X^\prime\in S$,
		$$\sum\limits_{Y\in D}(XY)(YX^\prime)^\ast =\delta_{XX^\prime}s(X),$$
		
		\item[$\mathcal{A}2$:] for every $Y,Y^\prime\in T$,
		$$\sum\limits_{X\in C}(YX)^\ast(XY^\prime) =\delta_{YY^\prime}r(Y),$$
		
	\end{enumerate}
\end{definition}

For notational convenience we denote the bi-separated graph with given distinguished subsets as in the above definition as a 5-tuple $\dot{E}=(E,(C,S),(D,T))$ and again call it bi-separated graph if there is no confusion and denote the Cohn-Leavitt path algebra also as $\mathcal{A}_K(\dot{E})$. Whenever we want to distinguish the case that $S=C_{\text{fin}}$ and $T=D_{\text{fin}}$ we simply call the Cohn-Leavitt path algebra as Leavitt path algebra.

\begin{proposition}[\textbf{Universal property of $\mathcal{A}_K(\dot{E})$}]\label{Universal_property_for_CLPA}
 Let $\dot{E}=(E,(C,S),(D,T))$ be a bi-separated graph. Suppose $\mathfrak{A}$ is a $K$-algebra which contains a set of pairwise orthogonal idempotents $\{A_v \mid v \in E^0\}$, two sets $\{A_e \mid e \in E^1\}$, $\{B_e \mid e \in E^1\}$ for which the following hold.
 \begin{enumerate}
     \item $A_{s(e)}A_e=A_eA_{r(e)}=A_e$, and $A_{r(e)}B_e=B_eA_{s(e)} =B_e$ for all $e \in E^1$.
     \item for every $X,X^\prime \in S$, $\sum\limits_{Y \in D}A_{XY}B_{YX^{\prime}} = \delta_{XX^{\prime}}A_{s(X)}$,
     \item for every $Y,Y^\prime \in T$, $\sum\limits_{X \in C}B_{YX}A_{XY^{\prime}} = \delta_{YY^{\prime}}A_{r(X)}$
 \end{enumerate}

 then there exists a unique map $\psi: \mathcal{A}_K(\dot{E}) \rightarrow \mathfrak{A}$ such that $\psi(v)=A_v$, $\psi(e) = A_e$, and $\psi(e^\ast)=B_e$ for all $v \in E^0$ and $e \in E^1$.
\end{proposition}

\begin{example}[\textbf{Cohn-Leavitt path algebra of a graph}]
Let $E$ be a graph and let $S\subseteq\text{RReg}(E)$. Then the Cohn-Leavitt path algebra $CL_K^S(E)$ of $E$ can be realized as Cohn-Leavitt path algebra $\mathcal{A}_K(\dot{E})$ of the bi-separated graph $\dot{E}=(E,(C,S),(D,T))$, where $(C,D)$ is the Cuntz-Krieger bi-separation on $E$, RReg$(E)=C_{\text{fin}}$ and $T=D$. 
\end{example}
\begin{example}(\textbf{Cohn-Leavitt path algebra of a separated graph})
Let $(E,C)$ be a separated graph and $S\subseteq C_{\text{fin}}$. Set $\dot{E}=(E,(C,S),(D,T))$, where $T=D_{\text{fin}}=D$. Then from definition it is clear that $\mathcal{A}_K(\dot{E})\cong CL_K(E,C,S)$.
\end{example}

 We say a bi-separated graph $\dot{E}=(E,(C,S),(D,T))$ is \textit{connected} if the underlying graph $E$ is connected. Because of the following proposition, we assume henceforth that every bi-separated graph is connected.

\begin{proposition}
 Let $\dot{E}$ be a bi-separated graph. Suppose $\dot{E} = \bigsqcup \limits_{j \in J} E_j$ is a decomposition of $\dot{E}$ into its connected components. Then $\mathcal{A}_{K}(\dot{E}) \cong \bigoplus \limits_{j \in J}\mathcal{A}_{K}(\dot{E_j})$, where $\dot{E_j}$ is the bi-separated graph structure on $E_j$ induced by the bi-separated graph structure on $E$.
\end{proposition}
\begin{proof}
Follows from universal property of $\mathcal{A}_K(\dot{E})$.
\end{proof}

\begin{lemma}
Let $\dot{E}$ be a bi-separated graph.
 \begin{enumerate}
     \item The algebra $\mathcal{A}_{K}(\dot{E})$ is unital if and only if $E^0$ is finite. In this case, $$1_{\mathcal{A}_{K}(\dot{E})} = \sum \limits_{v \in E^0} v.$$
     \item For each $\alpha \in \mathcal{A}_{K}(\dot{E})$, there exists a finite set of distinct vertices $V(\alpha)$ for which $\alpha = f \alpha f$, where $f = \sum \limits_{v \in V(\alpha)}v$. Moreover, the algebra $(\mathcal{A}_{K}(\dot{E}), E^0)$ is a ring with enough idempotents.
     \item Let $\overline{\phantom{A}} : K \rightarrow K$ be an involution on the field $K$. Then with respect to the involution $\ast : \mathcal{A}_K(\Dot{E}) \rightarrow \mathcal{A}_K(\Dot{E})$, $\mathcal{A}_K(\dot{E})$ is a $\ast$-algebra.
     \item $\mathcal{A}_K(\dot{E})$ is a graded quotient algebra of $K(\widehat{E})$ with respect to standard $\mathbb{Z}$-grading given by length of paths.
 \end{enumerate}
 \begin{proof}
  The proof follows on similiar lines of \cite[Lemma 1.2.12]{MR3729290}.
 \end{proof}
\end{lemma}

Now, the linear extension of $\ast$ in paths induces a grade reversing involution. Hence for any $\dot{E}$,  $\mathcal{A}_K(\dot{E})$ are $\mathbb{Z}$-graded $\ast$-algebras and the (graded) categories of left modules and right modules for any of Cohn-Leavitt path algebras of bi-separated graphs are equivalent.

Let $\Gamma$ be any group with identity $\varepsilon$ and we can consider a $\Gamma$-grading on $\mathcal{A}_K(\dot{E})$ with $E^0\sqcup E^1$ being homogeneous. Since $v^2=v$ for each $v\in E^0$, we are forced to define deg$(v)=\varepsilon$. Since the relations $\mathcal{A}1$ and $\mathcal{A}2$ are homogeneous we can conclude that all $e^\ast$ in $\overline{E^1}$ are also homogeneous with deg$(e^\ast)$=(deg$(e))^{-1}$. As a result, any function from $E^1$ to $\Gamma$ defines a unique $\Gamma$-grading on $\mathcal{A}_K(\dot{E})$ with deg$(v)=\varepsilon$ for all $v\in E^0$ and deg$(e^\ast)=$(deg$(e))^{-1}$. A refinement or a morphism from a $\Gamma$-grading to $\Gamma^\prime$-grading on an algebra $A$ is given by a group homomorphism $\phi:\Gamma\rightarrow\Gamma^\prime$ such that for all $\gamma^\prime\in\Gamma^\prime$, $A_{\gamma^\prime}=\bigoplus\limits_{\phi(\gamma)=\gamma^\prime}A_\gamma$ where $A_\gamma:=\{a\in A\mid \text{deg}_\Gamma(a)=\gamma\}\cup\{0\}$. There is a universal $\Gamma$-grading on $\mathcal{A}_K(\dot{E})$ which is a refinement of all others:
\begin{proposition}
	Let $\Gamma:=FE^1$, the free group on $E^1$ with identity $\varepsilon$. The $\Gamma$-grading defined by deg$_\Gamma(v)=\varepsilon$ and deg$_\Gamma(e)=e$ is an initial universal object in the category of group gradings of $\mathcal{A}_K(\dot{E})$ with $E^0\sqcup E^1$ being homogeneous.
\end{proposition}
\begin{proof}
	For any $\Gamma^\prime$-grading let $\phi:\Gamma\rightarrow\Gamma^\prime$ be the group homomorphism given by $\phi(e)=\text{deg}_{\Gamma^\prime}(e)$.
\end{proof}

\section{The categories \textbf{BSG} and \textbf{tBSG}}\label{Category}
In this section we introduce two categories \textbf{BSG} of bi-separated graphs and the category \textbf{tBSG} of tame bi-separated graphs. We show that the functor $\mathcal{A}_K(\underline{\hspace{7pt}})$ from \textbf{BSG} to \textbf{K-Alg} is continuous. We also show that each object of \textbf{tBSG} is a direct limit of sub-objects based on finite graphs, from which we obtain every Cohn-Leavitt path algebra of tame bi-separated graph as a direct limit of  unital Cohn-Leavitt path algebras.
\begin{definition}
	We define a category \textbf{BSG} of bi-separated graphs as follows: The objects of \textbf{BSG} are bi-separated graphs (with distinguished subsets) $\dot{E}=(E,(C,S),(D,T))$. A morphism $\phi: \dot{E} \rightarrow \dot{\widetilde{E}}$ in \textbf{BSG} is a triple $\phi=(\phi_0,\phi_1,\phi_2)$ satisfying the following conditions:
	\begin{enumerate}
		\item  $\phi_0:E\rightarrow \widetilde{E}$ is graph morphism such that $\phi_0^0,\phi_0^1$ are injective.
		
		\item $\phi_1:C\rightarrow\widetilde{C}$ is a map such that $X\in C_v\implies\phi_1(X)\in C_{\phi_0^0(v)}$.
		
		\item $\phi_1(S)\subset\widetilde{S}$. Moreover $\left.\phi_0^1\right|_X:X\rightarrow\widetilde{X}$ is a bijection, for every $X\in S$.
		
		\item $\phi_2:D\rightarrow\widetilde{D}$ is a map such that $Y\in D_v\implies\phi_1(Y)\in D_{\phi_0^0(v)}$.
		
		\item $\phi_2(T)\subset\widetilde{T}$. Moreover $\left.\phi_0^1\right|_Y:Y\rightarrow\widetilde{Y}$ is a bijection, for every $Y\in T$.
		\item If $X\in S, Y\in T$ then $X\cap Y=\emptyset\implies\widetilde{X}\cap\widetilde{Y}=\emptyset$.
	\end{enumerate}
\end{definition}

\begin{proposition}
	The category \emph{\textbf{BSG}} admits arbitrary direct limits.
\end{proposition}
\begin{proof}
The proof is similar to \cite[Proposition 3.3]{MR2980456}. The only addition is that we have to define $D$ and $T$ analogous to the way we define $C$ and $S$.

\end{proof}

Recall that a functor is \textit{continuous} if it preserves direct limits.
\begin{proposition}\label{continuous_functor_proposition}
	The assignment $\dot{E} \rightsquigarrow \mathcal{A}_K(\dot{E})$ extends to a continuous covariant functor $\mathcal{A}_K(\underline{\hspace{7pt}}))$ from \emph{\textbf{BSG}} to \emph{\textbf{K-Alg}}.
\end{proposition}
\begin{proof}
	The proof is similar to \cite[Proposition 3.6]{MR2980456}.
\end{proof}

\begin{definition}\label{complete_morphism}
	We say a morphism $\phi:\dot{E} \rightarrow \dot{\widetilde{E}}$ in \textbf{BSG} is \textit{complete} if $\phi_1^{-1}(\widetilde{S})=S$ and $\phi_2^{-1}(\widetilde{T})=T$.
\end{definition}
\begin{definition}\label{sub-object}
 Let $\dot{E}$ be an object in \textbf{BSG}. A \textit{sub-object} of $\dot{E}$ is an object $\dot{E^\prime}=(E^\prime,(C^\prime,S^\prime),(D^\prime,T^\prime))$ such that $E^{\prime}$ is a sub-graph of $E$ and the following conditions hold:
 \begin{eqnarray}
     C^{\prime} &=& \{X \cap (E^{\prime})^{1}~ |~ X \in C\setminus S,~ X \cap (E^{\prime})^{1} \neq \emptyset\} ~\sqcup \nonumber \\
    & & \{X \in S ~|~ X \cap (E^{\prime})^{1} \neq \emptyset\}.\nonumber \\
     S^{\prime} &=& \{X \in S ~|~ X \cap (E^{\prime})^{1} \neq \emptyset\}.\nonumber \\
     D^{\prime} &=& \{Y \cap (E^{\prime})^{1}~ |~ Y \in D\setminus T,~ D \cap (E^{\prime})^{1} \neq \emptyset\}~\sqcup \nonumber \\
     & &\{D \in T ~|~ D \cap (E^{\prime})^{1} \neq \emptyset\}.\nonumber \\
     T^{\prime} &=& \{Y \in T ~|~ Y \cap (E^{\prime})^{1} \neq \emptyset\}.\nonumber
 \end{eqnarray}
\end{definition}
\begin{definition}\label{complete_subobject}
	Let $\dot{E}$ be an object in \textbf{BSG}. A \textit{complete sub-object} of $\dot{E}$ is a sub-object $\dot{E^\prime}$ such that the inclusion is a complete morphism.		
\end{definition}
\begin{proposition}\label{prop: direct limit of countable}
	Any object in \emph{\textbf{BSG}} is a direct limit of countable complete sub-objects.
\end{proposition}
\begin{proof}
	Let $\dot{E}$ be an object in $\textbf{BSG}$. For a finite subset $A\subset E^0\sqcup E^1$, let $E_A$ be the graph generated by $A$, i.e., \begin{center}
		$E_A^1 = A\cap E^1$ and $E_A^0 = (A\cap E^0)\cup s_E^{-1}(E_A^1)\cup r_E^{-1}(E_A^1)$.
	\end{center}
	
	Take $v\in E_A^0$ and set
	\begin{center}
		$\mathcal{E}_{0v}=s_E^{-1}(v)\cup\bigcup\limits_{\substack{X\in S\cap C_v\\X\cap A\neq\emptyset}}X\cup\bigcup\limits_{\substack{Y\in T\cap D_v\\Y\cap A\neq\emptyset}}Y$.
	\end{center}
	If the graph $\mathcal{E}_0$ generated by $E_A^0\sqcup \bigsqcup\limits_{v\in E_A^0}\mathcal{E}_{0v}$ is a complete sub-object of the given object, we are done. If not, define
	\begin{center}
		$\mathcal{E}_{1v}$ to be $\mathcal{E}_{0v}\cup\bigcup\limits_{\substack{X\in S\cap C_v\\X\cap \mathcal{E}_{0v}\neq\emptyset}}X\cup\bigcup\limits_{\substack{Y\in T\cap D_v\\Y\cap \mathcal{E}_{0v}\neq\emptyset}}Y$.
	\end{center}
	If the graph $\mathcal{E}_1$ generated by $\mathcal{E}_0^0\sqcup \bigsqcup\limits_{v\in\mathcal{E}_0^0}\mathcal{E}_{1v}$ is a complete sub-object of the given object we are done. If not define $\mathcal{E}_{2v}$ similarly and continue this process.
	
	Hence we have a chain $\mathcal{E}_0\rightarrow\mathcal{E}_1\rightarrow\mathcal{E}_2\rightarrow\dots$.
	
	Let $\mathcal{E}$ be the direct limit of this chain, i.e., let $\dot{\mathcal{E}} = (\mathcal{E},(\mathcal{C},\mathcal{S}),(\mathcal{D},\mathcal{T}))$ be the direct limit of the directed system $\{\dot{\mathcal{E}_i} = (\mathcal{E}_i,(\mathcal{C}_i,\mathcal{S}_i),(\mathcal{D}_i,\mathcal{T}_i))\}$, where $i\in\mathbb{N} \cup \{0\}$.

	Observe that for $i\geq 0, ~\mathcal{S}_i=\{X\in S\mid X\cap\mathcal{E}_{i-1}^1\neq\emptyset\}$ and $\mathcal{C}_i=\{X\in\mathcal{E}_i^1\mid X\in C \setminus S, X\cap\mathcal{E}_i^1\neq\emptyset\}\sqcup \mathcal{S}_i$, where $\mathcal{E}_{-1}$ means $A$. Similar statements hold for $\mathcal{D}_i$ and $\mathcal{T}_i$.
	
	
	Since the vertex set and edge set of $\mathcal{E}$ are  countable union of finite sets, it is a countable sub-graph of $E$.
	We claim that $\dot{\mathcal{E}}$ is a complete sub-object of $\dot{E}$. For, let $X\in S$ and $X\cap\mathcal{E}^1\neq\emptyset$. Then $X\cap\mathcal{E}_i^1\neq\emptyset$, for some $i\in \mathbb{N} \cup \{0\}$. This implies $X\in\mathcal{S}_{i+1}$, which in turn means  $X\in\mathcal{S}$.  Similarly one can start with $X \in C \setminus S$ and prove that whenever $X \cap \mathcal{E}^1\neq\emptyset$, the set $X \cap \mathcal{E}^1 \in \mathcal{C} \setminus \mathcal{S}$. Repeating the analogous arguments for $T$, we can conclude that for each finite subset $A\subset E^0\sqcup E^1$, there exists a countable complete sub-object $\dot{\mathcal{E}}$ of $\dot{E}$. Now by keeping the set of all finite subsets of $E^0\sqcup E^1$ as the indexing set, we get a directed system of countable complete sub-objects whose direct limit is $\dot{E}$.

\end{proof}

We emphasize that a general object in \textbf{BSG} cannot be written as a direct limit of finite complete sub-objects as the following example illustrates:
\begin{example}
	Consider the following simple graph $\Gamma_{\infty}$ on countably infinite vertices.
	
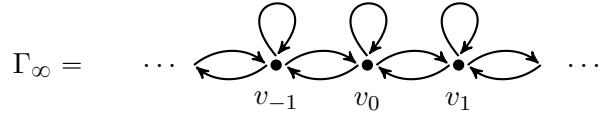
\begin{center}
		\begin{tikzpicture}
		[->,>=stealth',shorten >=1pt,thick, scale=0.6]

		\draw [fill=black] (-2,0)		circle [radius=0.1];		
		\draw [fill=black] (-4,0)		circle [radius=0.1];		
		\draw [fill=black] (0,0)		circle [radius=0.1];

		\path[->]	(0,0.2)	edge[out=130, in=50, looseness=4,  distance=2cm, ->] 	node	{}		(0,0.2);
		\path[->]	(-2,0.2)	edge[out=130, in=50, looseness=4,  distance=2cm, ->] 	node	{}		(-2,0.2);
		\path[->]	(-4,0.2)	edge[out=130, in=50, looseness=4,  distance=2cm, ->] 	node	{}		(-4,0.2);

		\path[->]	(-5.8,0)	edge[bend left=40] 	node	{}		(-4.2,0);
		\path[->]	(-3.8,0)	edge[bend left=40] 	node	{}		(-2.2,0);
		\path[->]	(-1.8,0)	edge[bend left=40] 	node	{}		(-0.2,0);
		\path[->]	(0.2,0)	edge[bend left=40] 	node	{}		(1.8,0);
		
		\path[->]	(-4.2,0)	edge[bend left=40] 	node	{}		(-5.8,0);
		\path[->]	(-2.2,0)	edge[bend left=40] 	node	{}		(-3.8,0);
		\path[->]	(-0.2,0)	edge[bend left=40] 	node	{}		(-1.8,0);
		\path[->]	(1.8,0)	edge[bend left=40] 	node	{}		(0.2,0);

		\node at (2.8,0) {$\dots$};
		\node at (-6.5,0) {$\dots$};
		\node at (-4,-0.8) {$v_{-1}$};
		\node at (-2,-0.8) {$v_0$};
		\node at (0,-0.8) {$v_1$};

		\node at (-9,0) {$\Gamma_\infty=$};	
		
		\end{tikzpicture}
	\end{center}
Observe that $\Gamma_\infty$ is a simple graph. Consider the standard bi-separation $(C,D)$ on $\Gamma_\infty$ and let $S=C_{\text{fin}}=C$ and $T=D_{\text{fin}}=D$. Then $C_\infty$ cannot be written as a direct limit of finite complete sub-objects. For, if there is a complete sub-object of $\Gamma_\infty$, then by definition we are forced to include all the edges and so, it will no more be finite.
	
\end{example}

\subsection{The category of tame bi-separated graphs}
\hfill\\

Let $\dot{E}$ be an object in \textbf{BSG}. Set \begin{eqnarray}\label{definition_of_S_1_and_T_1}
	 S_1 &=&\{X\in S\mid X\cap Y\neq\emptyset,\quad\text{for some}\quad Y\in T\},\\
	 S_2 &=& S-S_1, \\
	T_1 &=& \{Y\in T\mid X\cap Y\neq\emptyset,\quad\text{for some}\quad X\in S\}~ \text{and} \\
	T_2 &=& T-T_1.
\end{eqnarray}

We define a relation $\sim_T$ on $S_1$ as follows: For $X,X^{\prime} \in  S_1$, define $X\sim_T X^\prime$ if either $X = X^{\prime}$ or there exists a finite sequence $X_0,Y_1,X_1,Y_2,X_2,$  $\ldots, Y_{n-1},X_{n-1},Y_n,X_n$ such that for each $0\leq i\leq n$, $X_i \in S, Y_i \in T$, and $X_0=X$, $X_n=X^\prime$, $X_i\cap Y_{i+1}\neq\emptyset$ and $Y_i\cap X_i \neq \emptyset$. It is not hard to see that $\sim_T$ is an equivalence relation on $S_1$. Let $S_1 = \bigsqcup\limits_{\lambda \in \Lambda} \mathcal{X}_{\lambda}$ be the partition of $S_1$ induced by $\sim_T$.

Define $\sim_S$ on $T_1$ similarly and let $T_1=\bigsqcup\limits_{\lambda^\prime\in\Lambda^\prime}\mathcal{Y}_{\lambda^\prime}$ be the partition induced by $\sim_S$.

We claim that the indexing sets $\Lambda$ and $\Lambda^{\prime}$ are in bijection. To see this, start with $\lambda \in \Lambda$. Let $X \in \mathcal{X}_{\lambda}$ be an arbitrarily fixed element. This means, there exists a $\lambda^{\prime} \in \Lambda^{\prime}$ and $Y \in \mathcal{Y}_{\lambda^{\prime}}$ such that $X\cap Y \neq \emptyset$. If $X^{\prime} \neq X$ is another element of $\mathcal{X}_{\lambda}$, and if there is a $Y^{\prime} \in T$ such that $X^{\prime}\cap Y^{\prime} \neq \emptyset$, then $Y^{\prime} \sim_S Y$ because $X^{\prime} \sim_T X$. So $Y^{\prime}$ belongs to the same $\mathcal{Y}_{\lambda^{\prime}}$ as $Y$. Also if there is another element $Y_1 \in T$ such that $X\cap Y_1 \neq \emptyset$, then clearly $Y_1 \sim_S Y$ and so $Y_1$ also lies in same $\mathcal{Y}_{\lambda^{\prime}}$. This implies that the map $\Lambda \rightarrow \Lambda^{\prime}$ defined by $\lambda \mapsto \lambda^{\prime}$ is well-defined. Similarly one can define a map $\Lambda^{\prime} \rightarrow \Lambda$. It is not hard to see that these maps are inverses of each other which proves the claim. Therefore, we have the following proposition:

\begin{proposition}\label{Lambda=Lambda{prime}}
 Let $\dot{E}$ be an object in \emph{\textbf{BSG}} and let $S_1, T_1$ be as defined in equations \ref{definition_of_S_1_and_T_1} and 3 above. Then there exist canonical partitions $S_1 = \bigsqcup \limits_{\lambda \in \Lambda} \mathcal{X}_{\lambda}$ and $T_1 = \bigsqcup \limits_{\lambda^{\prime} \in \Lambda^{\prime}} \mathcal{Y}_{\lambda^{\prime}}$ of $S_1$ and $T_1$ respectively such that the indexing sets $\Lambda$ and $\Lambda^{\prime}$ are bijective.
\end{proposition}

\begin{remark}\label{Indexing set Lambda in BSG}
 Because of the above proposition, we will denote the indexing sets of the canonical partitions of both $S_1$ and $T_1$ by $\Lambda$.
\end{remark}
	\begin{definition}\label{definition_of_tame}
	A bi-separated graph $\dot{E}$ is called  \textit{tame} if $|\mathcal{X}_{\lambda}| < \infty$ and $|\mathcal{Y}_{\lambda}| < \infty$, for each $\lambda \in \Lambda$. The tame bi-separated graphs along with complete morphisms form a category which we call a \textit{tame (sub)category of bi-separated graphs}. It will be denoted by \textbf{tBSG}.
\end{definition}
Note that any finite bi-separated graph is tame. Also, the class of bi-separated graphs in examples \ref{Cuntz-Krieger bi-separation}, \ref{separated graphs}, \ref{Weighted bi-separation}, and \ref{hyper bi-separation} are all tame.


\begin{proposition}\label{free-product_and_tame}
	Let $\dot{E}$ be a tame bi-separated graph such that $S = C_{\emph{fin}} = C$, $T = D_{\emph{fin}} = D$ and $|E^0|=1$. For $\lambda\in\Lambda$, let $E_\lambda$ be the subgraph of $E$ with edge set $\bigcup\limits_{\substack{e\in X\\X\in\mathcal{X}_\lambda}}\{e\}$ and consider the bi-separation $C_\lambda=\{X\in\mathcal{X}_\lambda\}$, and $D_\lambda=\{Y\in\mathcal{Y}_\lambda\}$ . Then  
	$\mathcal{A}_K(\dot{E})$ is isomorphic to the free-product of algebras $\mathcal{A}_K(\dot{E_\lambda})$, where $\lambda$ varies over the indexing set $\Lambda$.
\end{proposition} 
\begin{corollary}[\cite{MR2980456} Proposition 2.10]
 Let $(E,C)$ be a separated graph with $|E^0|=1$. Then $L_K(E,C)$ is isomorphic to the free-product of algebras $L_K(1,|X|)$, where $L_K(1,|X|)$ is the Leavitt algebra of type $(1,|X|)$ and $X$ varies over $C$.
\end{corollary}

\begin{theorem}\label{direct-limit_of_finite_complete_sub-objects}
	Every object $\dot{E}$ in \emph{\textbf{tBSG}} is a direct limit of finite (complete) sub-objects. Conversely, if $\dot{E}$ in \emph{\textbf{BSG}} is a direct limit of finite complete sub-objects then it belongs to \emph{\textbf{tBSG}}.
\end{theorem}
\begin{proof}
	Let $\dot{E}$ be an object in \textbf{tBSG}. By exactly same arguments as in Proposition \ref{prop: direct limit of countable}, $\dot{E}$ is a direct limit of the directed system $$\{\left( \dot{\mathcal{E}_A} = (\mathcal{E},(\mathcal{C},\mathcal{S}),(\mathcal{D},\mathcal{T}))_A,\hookrightarrow\right) \mid A \quad\text{is a finite subset of} \quad E^0\sqcup E^1\}.$$ It follows from the definition of tame bi-separated graphs that $\dot{\mathcal{E}_A}$ is finite, for each finite subset $A$ of $E^0 \sqcup E^1$.
	
	Conversely, let $\dot{E}$ be a direct limit of the directed system
	$$\{\left( \dot{\mathcal{E}_i} = (\mathcal{E}_i,(\mathcal{C}_i,\mathcal{S}_i),(\mathcal{D}_i,\mathcal{T}_i)),\hookrightarrow\right) \mid i \in I\}$$ of finite complete sub-objects. We know that $S$ and $T$ can be partitioned as $S_1\sqcup S_2$ and $T_1\sqcup T_2$ respectively (see equations \ref{definition_of_S_1_and_T_1} to 4 before the proposition \ref{Lambda=Lambda{prime}}). If $S_1=\emptyset$, then $T_1=\emptyset$ and clearly $\dot{E}$ is tame.
	
	Suppose $S_1\neq\emptyset$. Then we have $S_1=\bigsqcup\limits_{\lambda \in \Lambda}\mathcal{X}_{\lambda}$ and $T_1=\bigsqcup\limits_{\lambda \in \Lambda}\mathcal{Y}_{\lambda}$. That each $|\mathcal{X}_{\lambda}|$ and $|\mathcal{Y}_{\lambda}|$ is finite follows from the fact that the directed system consists of finite objects and the morphisms involved are complete.
\end{proof}
\begin{corollary}\label{cor: CLPA of tame is direct limit of unital}
 Let $\Dot{E}$ be an object in \emph{\textbf{tBSG}}. Then the Cohn-Leavitt path algebra $\mathcal{A}_K(\Dot{E})$ is the direct limit of the directed system of unital algebras $\{\mathcal{A}_K(\Dot{E}_i)\}_{i \in I}$ such that whenever $j \geq i$, the map $\mathcal{A}_K(\Dot{E}_i)\rightarrow \mathcal{A}_K(\Dot{E}_j)$ is a monomorphism, where 
 $\{\Dot{E}_i\}_{i \in I}$ is a directed system of finite complete sub-objects of $\Dot{E}$ whose direct limit is $\Dot{E}$.
 \begin{proof}
     By the previous theorem, $\Dot{E}$ is a direct limit of a directed system $\{\Dot{E}_i\}_{i \in I}$ consisting of its finite complete sub-objects. Therefore, by proposition \ref{continuous_functor_proposition} and Theorem \ref{Theorem: Normal forms}, $\mathcal{A}_K(\Dot{E})$ is the direct limit of the directed system of algebras $\{\mathcal{A}_K(\Dot{E}_i)\}_{i \in I}$. 
 \end{proof}
\end{corollary}

By corollary \ref{cor: CLPA of tame is direct limit of unital} (Cohn-)Leavitt path algebras of the classes of bi-separated graphs in examples \ref{Cuntz-Krieger bi-separation}-\ref{hyper bi-separation} are direct limits of unital sub-(Cohn-)Leavitt path algebras of same type.
 
\section{Normal forms and their applications}\label{normal forms}

\begin{definition}\label{forbidden_word}
	For each pair $X,X^\prime \in S$, if there exists $Y\in D$ such that $X\cap Y\neq\emptyset$ and $X^\prime\cap Y\neq \emptyset$, then we choose and fix one such $Y$ (this may vary with $X,X^\prime$) and call $(XY)(YX^\prime)^\ast $ a \textit{forbidden word of type I}.
	
	Similarly for each pair $Y,Y^\prime \in T$, if there exists $X\in D$ such that $Y\cap X\neq\emptyset$ and $Y^\prime\cap X\neq \emptyset$, then we choose and fix one such $X$ and call $(YX)^\ast (XY^\prime)$ a \textit{forbidden word of type II}.
\end{definition}

\begin{definition}\label{normal_path}
	A generalized path $\mu\in\widehat{E}^\star$ is called \textit{normal} if it does not contain any forbidden subword of types I or II. An element of $K(\widehat{E})$ is called \textit{normal} if it lies in the $K$-linear span of generalized normal paths.
\end{definition}

We will show that any element of $\mathcal{A}_K(\Dot{E})$ has precisely one normal representative in $K(\widehat{E})$. For this, we need to use Bergman's diamond lemma. We refer the reader to \cite[pp. 180-182]{MR506890} for the statement of the lemma and basic terminologies.

\begin{theorem}\label{Theorem: Normal forms}
	Let $\Dot{E}$ be a bi-separated graph . Then $\mathcal{A}_K(\Dot{E})$ has a basis consisting of normal generalized paths.
\end{theorem}
\begin{proof}
	In order to apply Bergman's diamond lemma, we
	replace the defining relations by the following:
	\begin{center}
	    \begin{enumerate}
	    \item[$1^\prime$:] For any $v,w\in E^0$, $$vw=\delta_{v,w}v.$$
	    \item[$2^\prime$:] For any $v\in E^0, e\in E^1$,
	    $$ve=\delta_{v,s(e)}e,$$
	    $$ev=\delta_{v,r(e)}e,$$
	    $$ve^\ast=\delta_{v,r(e)}e^\ast,$$
	    $$e^\ast v=\delta_{v,s(e)}e^\ast.$$
	    \item[$3^\prime$:] For any $e,f\in E^1$,
	    $$ef=0,\quad\text{if}\quad r(e)\neq s(f),$$
	    $$e^\ast f=0,\quad\text{if}\quad s(e)\neq s(f),$$
	    $$ef^\ast=0,\quad\text{if}\quad r(e)\neq r(f),$$
	    $$e^\ast f^\ast=0,\quad\text{if}\quad s(e)\neq r(f).$$
	   
		\item[$\mathcal{A}^\prime 1:$] For each $X,X^\prime \in S$ for which there exists $Y\in D$ such that $X\cap Y\neq \emptyset$ and $X^\prime \cap Y\neq\emptyset$,
		$$e_Xe_{X^\prime}^\ast:=(XY)(YX^\prime)^\ast=\delta_{X,X^\prime}s(X)-\sum\limits_{\substack{Y_1\in D\\Y_1\neq Y}}(XY_1)(Y_1X^\prime)^\ast.$$
		
		\item[$\mathcal{A}^\prime 2:$] For each $Y,Y^\prime \in T$ for which there exists $X\in C$ such that $X\cap Y\neq \emptyset$ and $X\cap Y^\prime\neq\emptyset$,
		$$e_Ye_{Y^\prime}^\ast:=(YX)(XY^\prime)^\ast=\delta_{Y,Y^\prime}r(Y)-\sum\limits_{\substack{X_1\in C\\X_1\neq X}}(YX_1)^\ast(X_1Y^\prime).$$
	\end{enumerate}
	\end{center}
	(i.e. In $\mathcal{A}^\prime 1$ and $\mathcal{A}^\prime 2$, LHS contains forbidden words).
	
    Denote by $\Sigma$ the reduction system consisting of all pairs $\sigma=(w_\sigma,f_\sigma)$, where $w_\sigma$ equals the LHS of an equation above and $f_\sigma$ the corresponding RHS. Let $\langle \Bar{P}\rangle$ be the monoid consisting of all words formed by letters in $E^0\cup E^1\cup \overline{E^1}$ and $\langle P \rangle$ be the semi-group obtained by removing the identity element of $\langle \Bar{P}\rangle$. We define a partial order on $\langle P\rangle$ as follows:
    
    Let $A=x_1x_2\dots x_n\in \langle P\rangle$. Set $l(A)=n$ and $$m(A)=|\{i\in\{1,2,\dots,(n-1)\}\mid x_ix_{i+1} \text{is of type I or type II}\}|.$$ Define a partial order $\leq$ on $\langle P\rangle$ by $A\leq B$ if and only if one of the following holds:
    \begin{enumerate}
        \item $A=B$,
        \item $l(A)<l(B)$ or
        \item $l(A)=l(B),\quad\text{and for each}\quad G,H\in\langle\Bar{P}\rangle,\quad m(GAH)<m(GBH)$.
    \end{enumerate}
    
    Clearly $\leq$ is a semigroup partial order on $\langle P\rangle$ compatible with $\Sigma$ and also the descending chain condition is satisfied. It remains to show that all ambiguities of $\Sigma$ are resolvable. Recall from Proposition \ref{prop: basis of a path algebra} that $\widehat{E}^\star$ is a linear $K$-basis for $K(\widehat{E})$. Hence it is sufficient to show that the following ambiguities are resolvable:
    
    \begin{equation}
       e_Xe_{X^\prime}^\ast e_Y = (XY^\prime)(Y^\prime X^\prime)^\ast(X^\prime Y) \tag{$\mathcal{A}^\prime 1-\mathcal{A}^\prime 2$}
    \end{equation}
    \begin{equation}
        e_Y^\ast e_{Y^\prime}e_X^\ast = (YX^\prime)^\ast(X^\prime Y^\prime)(Y^\prime X^\prime)^\ast. \tag{$\mathcal{A}^\prime 2-\mathcal{A}^\prime 1$}
    \end{equation}
    
   We note that there are no inclusion ambiguities. We only show how to resolve ambiguity of type $\mathcal{A}^\prime 1-\mathcal{A}^\prime 2$ and the other case follows similarly.
    
	\begin{center}
		\begin{tikzpicture}
	[->,>=stealth',shorten >=1pt,thick, scale=0.6]

	\path[->]	(-3,24)	edge[bend right=20] 	node	{}		(-6,22);
	\path[->]	(3,24)	edge[bend left=40] 	node	{}		(7,19);
	
	\path[->]	(-7,20)	edge[] 	node	{}		(-7,16.5);
	\path[->]	(-7,14)	edge[] 	node	{}		(-7,10.5);
	\path[->]	(-7,8)	edge[] 	node	{}		(-7,4.5);
	
	\path[->]	(7,17)	edge[] 	node	{}		(7,13.5);
	\path[->]	(7,11)	edge[] 	node	{}		(7,7.5);
	\path[->]	(7,5)	edge[] 	node	{}		(7,1.5);
	
	\path[->]	(-7,2)	edge[bend right=40] 	node	{}		(-4,-2.5);
	\path[->]	(7,0)	edge[bend left=20] 	node	{}		(5,-2.5);
	
	\node at (-5.5,23.7) {$\mathcal{A}^\prime 1$};
	\node at (5.5,23.7) {$\mathcal{A}^\prime 2$};
	\node at (-8,18.5) {$=$};
	\node at (7.5,15.5) {$=$};
	\node at (-8,12.5) {$\mathcal{A}^\prime 2$};
	\node at (8,9.5) {$\mathcal{A}^\prime 1$};
	\node at (-8,6) {$=$};
	\node at (-7,-1) {$=$};
	\node at (7.5,3.5) {$=$};
	\node at (7,-1.5) {$=$};

	\node at (0,24) {$(XY_1)(Y_1X_1)^\ast(X_1Y)$};
	
	\node at (-5,21) {$\bigg[ \delta_{X,X_1}s(X_1)-\sum\limits_{\substack{Y_2\in D\\Y_2\neq Y_1}}(XY_2)(Y_2X_1)^\ast\bigg](X_1Y)$};
	
	\node at (5,18) {$(XY_1)\bigg[ \delta_{Y,Y_1}r(Y_1)-\sum\limits_{\substack{X_2\in C\\X_2\neq X_1}}(Y_1X_2)^\ast(X_2Y)\bigg]$};
	
	\node at (-5,15) {$\delta_{X,X_1}(X_1Y)-\sum\limits_{\substack{Y_2\in D\\Y_2\neq Y_1}}(XY_2)(Y_2X_1)^\ast(X_1Y)$};
	
	\node at (5,12) {$\delta_{Y,Y_1}(XY_1)-\sum\limits_{\substack{X_2\in C\\X_2\neq X_1}}(XY_1)(Y_1X_2)^\ast(X_2Y)$};
	
	\node at (-3,9) {$\delta_{X,X_1}(X_1Y)-\sum\limits_{\substack{Y_2\in D\\Y_2\neq Y_1}}(XY_2)\bigg[\delta_{Y,Y_2}r(Y)-\sum\limits_{\substack{X_2\in C\\X_2\neq X_1}}(Y_2X_2)^\ast(X_2Y)\bigg]$};
	
	\node at (3,6) {$\delta_{Y,Y_1}(XY_1)-\sum\limits_{\substack{X_2\in C\\X_2\neq X_1}}\bigg[\delta_{X,X_2}s(X)-\sum\limits_{\substack{Y_2\in D\\Y_2\neq Y_1}}(XY_2)(Y_2X_2)^\ast\bigg](X_2Y)$};
	
	\node at (-3,3) {$\delta_{X,X_1}(X_1Y)-\delta_{Y,Y_2}(XY_2)+\sum\limits_{\substack{Y_2\in D\\Y_2\neq Y_1}}\sum\limits_{\substack{X_2\in C\\X_2\neq X_1}}(XY_2)(Y_2X_2)^\ast(X_2Y)$};
	
	\node at (3,0) {$\delta_{Y,Y_1}(XY_1)-\delta_{X,X_2}(X_2Y)+\sum\limits_{\substack{X_2\in C\\X_2\neq X_1}}\sum\limits_{\substack{Y_2\in D\\Y_2\neq Y_1}}(XY_2)(Y_2X_2)^\ast(X_2Y)$};
	
	\node at (0,-3) {$\sum\limits_{\substack{X_2\in C\\X_2\neq X_1}}\sum\limits_{\substack{Y_2\in D\\Y_2\neq Y_1}}(XY_2)(Y_2X_2)^\ast(X_2Y)$};
	
	\end{tikzpicture}
\end{center}
    
    This proves the confluence condition. The final expression written above is a finite sum as $X\in S$ and $Y\in T$. If some terms in this expression contain forbidden words, we further reduce them as above. But this process has to terminate in finitely many stages (again since $X\in S$ and $Y\in T$). Also, it should be noted that in this sequence of reductions, the same forbidden word cannot appear more than once at different stages. For, that would mean that a single relation is contributing more than one forbidden words, which is a contradiction. This proves the reduction finiteness as well. 
\end{proof}

\begin{corollary}
 Let $\dot{E}$ be a bi-separated graph. Then the natural homomorphism from the path algebra $K(E)$ to the algebra $\mathcal{A}_K(\dot{E})$ is an inclusion.
 \begin{proof}
  From the theorem, it follows that each path $\mu\in E^\star$ is a part of a basis of $\mathcal{A}_K(\dot{E})$ as $\mu$ does not contain any forbidden word.
 \end{proof}
\end{corollary}

In the following subsections we give some applications of normal forms of Cohn-Leavitt path algebras. To be precise, we find bi-separated graph theoretic properties that correspond to algebraic properties. One needs to be careful in the sense that in the general setting of bi-separated graphs, finding conditions only on underlying graphs are not enough. We also need suitable conditions on $(S,T)$ which give the defining relations of Cohn-Leavitt path algebras. In the following subsections we recall some definitions and propositions from the theory of rings with enough idempotents. Then we give their applications to the case of Cohn-Leavitt path algebras. We note that the reasoning is very similar to that of \cite{preusser2019leavitt}. Wherever some care is required we provide complete proofs else the reader is refered to \cite{preusser2019leavitt} for proofs.

\subsection{Local valuations and their applications}\label{LV-Objects}\hfill\\

\begin{definition}
 Let $(R,I)$ be a ring with enough idempotents. A \textit{local valuation} on $(R,I)$ is a map $\nu:R\rightarrow\mathbb{Z}^+\cup\{-\infty\}$ such that 
 \begin{enumerate}
     \item $\nu(x)=-\infty~\text{if and only if}~x=0$,
     \item $\nu(x-y)\leq\max\{\nu(x),\nu(y)\}$ for any $x,y\in R$ and
     \item $\nu(xy)=\nu(x)+\nu(y)$ for any $e\in I$, $x\in Re$ and $y\in eR$.
 \end{enumerate}
 A local valuation $\nu$ on $(R,I)$ is called trivial if $\nu(x)=0$ for each $x\in R-\{0\}$. 
\end{definition}

Let $R$ be a ring. A left ideal $\mathfrak{a}$ of $R$ is called \textit{essential} if $\mathfrak{a}\cap\mathfrak{b}=0\Rightarrow\mathfrak{b}=0$ for any left ideal $\mathfrak{b}$ of $R$. For any $x\in R$ recall the left annihilator ideal of $x$ is Ann$(x):=\{y\in R\mid yx=0\}$. A ring $R$ is called \textit{left non-singular} if for any $x\in R$, Ann$(x)$ is essential $\Leftrightarrow x=0$. A right non-singular ring is defined similarly. A ring is \textit{non-singular} is if it is both left and right non-singular.
\begin{proposition}\cite[Proposition 37]{preusser2019leavitt}
Let $(R,I)$ be a ring with enough idempotents which admits a local valuation. Then $R$ is non-singular.
\end{proposition}

A non-zero ring $R$ is called a \textit{prime ring} if $\mathfrak{ab}=0\Rightarrow\mathfrak{a}=0$ or $\mathfrak{b}=0$ for any ideals $\mathfrak{a}$ and $\mathfrak{b}$ of $R$. A ring with enough idempotents $(R,I)$ is \textit{connected} if $eRf\neq 0$ for any $e,f\in I$.

\begin{proposition}\cite[Proposition 38]{preusser2019leavitt}
Let $(R,I)$ be a nonzero, connected ring with enough idempotents which admits a local valuation. Then $R$ is a prime ring.
\end{proposition}

A ring $R$ is said to be \textit{von Neumann regular} if for any $x\in R$ there exists $y\in R$ such that $xyx=x$.
\begin{proposition}\cite[Proposition 39]{preusser2019leavitt}
Let $(R,I)$ be a ring with enough idempotents that has a nontrivial local valuation. Then $R$ is not von Neumann regular.
\end{proposition}

Recall that the Jacobson radical of a ring $R$ is the ideal consisting of those elements in $R$ that annihilate all simple (right or left) $R$-modules. A ring is called \textit{semiprimitive} if its Jacobson radical is the zero ideal. 
\begin{proposition}\cite[Proposition 40]{preusser2019leavitt}
Let $(R,I)$ be a connected $K$-algebra with enough idempotents which admits a local valuation $\nu$ such that $\nu(x)=0$ if and only if $x$ is a nonzero $K$-linear combination of elements in $I$. Then $R$ is semiprimitive.
\end{proposition}

Now we find conditions on a bi-separated graph $\dot{E}$ for which the corresponding Cohn-Leavitt path algebra admits a local valuation.

\begin{definition}\label{def: LV_Condition}
A bi-separated graph $\dot{E}=(E,(C,S),(D,T))$ is said to satisfy \textit{Condition LV} if one of the following holds:
\begin{enumerate}
    \item[$(LV1)$ :] $|S|\leq 1$ , $|T|\leq 1$ and if $X \in S$ (resp. $Y \in T$), then $|X| > 1$ (resp. $Y > 1)$. 
    \item[$(LV2)$ :] $|S|>1$ or $|T|>1$, and the following two conditions are satisfied:
     \begin{enumerate}
    \item For any $X_1, X_2\in S$, either there is no $Y\in D$ such that $X_1\cap Y\neq\emptyset~\text{and}~X_2\cap Y\neq\emptyset$
     or $\exists~ Y_1,Y_2\in D, Y_1\neq Y_2$ such that $X_i \cap Y_j \neq \emptyset$ for each $1\leq i,j\leq 2$.
     \item For any $Y_1, Y_2\in T$, either there is no $X\in C$ such that $Y_1\cap X\neq\emptyset$ and $Y_2\cap X\neq\emptyset$
     or $\exists~ X_1,X_2\in C, X_1\neq X_2$ such that $X_i \cap Y_j \neq \emptyset$ for each $1\leq i,j\leq 2$.
     \end{enumerate}
\end{enumerate}

We say $\dot{E}$ satisfies \textit{Domain condition} if either $|S|\leq 1, |T|\leq 1$ or $(LV2)$ holds.
    
\end{definition}

\begin{proposition}\label{prop: local valuation existence}
Let $\dot{E}$ be a bi-separated graph and for any $a\in\mathcal{A}_K(\dot{E})$, let $\emph{supp}(a)$ denote the set of all normal generalized paths occuring in $\emph{NF}(a)$ with nonzero coefficients, where $\emph{NF}(a)$ is the unique normal representative of $a$. If $\Dot{E}$ satisfies condition LV, then the map
 $\nu:\mathcal{A}_K(\Dot{E}) \rightarrow \mathbb{Z}^+ \cup \{-\infty\}$ defined by 
 $$a \mapsto \max\{|p|\mid p \in \emph{supp}(a)\}$$
 is a local valuation on $\mathcal{A}_K(\Dot{E})$.
 
 \begin{proof}
  The first two conditions of a local valuation are obvious. It remains to show $\nu(ab) \, = \, \nu(a)+\nu(b),$ for any $v \in E^0,~a \in \mathcal{A}_K(\Dot{E})v$ and $b \in v\mathcal{A}_K(\Dot{E})$. If one of $\nu(a)$ and $\nu(b)$ is $0$ or $-\infty$, then the result is clear. Suppose now $\nu(a),\nu(b) \, \geq \, 1$. Since any reduction preserves or decreases the length of a generalized path, it follows that $\nu(ab) \, \leq \, \nu(a)+\nu(b)$. So it remains to show that $\nu(ab) \, \geq \, \nu(a)+\nu(b)$. Let
     $$p_k=x_1^k \dots x_{\nu(a)}^k ~~(1 \leq k \leq r)$$
 be the elements of supp($a$) with length $\nu(a)$ and
     $$q_l=y_1^l \dots y_{\nu(b)}^l ~~(1 \leq l \leq s)$$
 be the elements of supp($b$) with length $\nu(b)$. We assume that the $p_k$'s are pairwise distinct and so are $q_l$'s. Since NF is a linear map, we can make the following conclusions:
 
 \begin{enumerate}
     \item If $x_{\nu(a)}^ky_{1}^l$ is not of type I or II, then 
     \begin{center}
         NF$(p_kq_l)=p_kq_l$.
     \end{center}
     
     \item If $x_{\nu(a)}^ky_{1}^l$ is of type I, then there are $X, X^{\prime} \in S$ and $Y \in D$ such that $x_{\nu(a)}^ky_{1}^l \, = \, (XY)(X^{\prime}Y)^\ast $ and $(XY)(X^{\prime}Y)^\ast $ is forbidden. So
     \begin{eqnarray}
         \text{NF}(p_kq_l)&=& \left[ \delta_{XX^{\prime}}x_1^k \ldots x_{\nu(a)-1}^ky_2^l \ldots y_{\nu(b)}^l\right] \nonumber\\
         &-&
        \sum\limits_{\substack{Y \in D \\(XY)(X^{\prime}Y)^\ast \neq x_{\nu(a)}^ky_{1}^l}}x_1^k \ldots x_{\nu(a)-1}^k(XY)(X^{\prime}Y)^\ast y_2^l \ldots y_{\nu(b)}^l.\nonumber
     \end{eqnarray}
     
     \item If $x_{\nu(a)}^ky_{\nu(b)}^l$ is of type II, then there are $Y, Y^{\prime} \in T$ and $X \in C$ such that $x_{\nu(a)}^ky_{\nu(b)}^l \, = \, (XY)^\ast (XY^{\prime})$ and $(XY)^\ast (XY^{\prime})$ is forbidden. So
     \begin{eqnarray}
       \text{NF}(p_kq_l)&=& \left[ \delta_{YY^{\prime}}x_1^k \ldots x_{\nu(a)-1}^ky_2^l \ldots y_{\nu(b)}^l\right] \nonumber \\&-&
        \sum\limits_{\substack{X \in C \\(XY)^\ast(XY^{\prime}) \neq x_{\nu(a)}^ky_{1}^l}}x_1^k \ldots x_{\nu(a)-1}^k(XY)^\ast(XY^{\prime}) y_2^l \ldots y_{\nu(b)}^l.\nonumber  
     \end{eqnarray}
 \end{enumerate}

\noindent$\mathbf{Case~ 1:}$ Assume that $x_{\nu(a)}^ky_{1}^l$ is not of type I or II, for any $k$, $l$.\\ 
Then $p_kq_l \in \text{supp}($a$)$, for any $k$, $l$. So $\nu(ab) \, \geq \, |p_kq_l| \,=\, \nu(a)+\nu(b)$.\\

\noindent $\mathbf{Case~ 2:}$ Assume that there are $k$ and $l$ such that $x_{\nu(a)}^ky_{\nu(b)}^l$ is of type I. \\
Then there are $X, X^{\prime} \in S$ and $Y \in D$ such that $x_{\nu(a)}^ky_{1}^l \, = \, (XY)(X^{\prime}Y)^\ast $ and $(XY)(X^{\prime}Y)^\ast $ is forbidden. Since $\Dot{E}$ is an LV-object, there is at least one more element $\widetilde{Y} \in D$ other than $Y$ such that $X\widetilde{Y} \neq 0$ and $X^{\prime}\widetilde{Y} \neq 0$.\\
$\mathbf{Case~ 2.1:}$ Assume $p_{k^{\prime}}q_{l^{\prime}} \neq x_1^k \ldots x_{\nu(a)-1}^k(X\widetilde{Y})(X^{\prime}\widetilde{Y})^\ast y_2^l \ldots y_{\nu(b)}^l$, for any $k^{\prime}, l^{\prime}$. 

\noindent Then $x_1^k \ldots x_{\nu(a)-1}^k(X\widetilde{Y})(X^{\prime}\widetilde{Y})^\ast y_2^l \ldots y_{\nu(b)}^l \in \text{supp}(ab)$, since it does not cancel with any other term. So we are done.\\
$\mathbf{Case~ 2.2:}$ Assume $p_{k^{\prime}}q_{l^{\prime}} \, = \, x_1^k \ldots x_{\nu(a)-1}^k(X\widetilde{Y})(X^{\prime}\widetilde{Y})^\ast y_2^l \ldots y_{\nu(b)}^l$, for some $k^{\prime}, l^{\prime}$. In this case, $p_kq_{l^{\prime}} \, = \, x_1^k \ldots x_{\nu(a)-1}^k(XY)(X^{\prime}\widetilde{Y})^\ast y_2^l \ldots y_{\nu(b)}^l \in \text{supp}(ab)$ and so we are done.\\

\noindent$\mathbf{Case~ 3:}$ Assume that there are $k$ and $l$ such that $x_{\nu(a)}^ky_{1}^l$ is of type II. \\
Then there are $Y, Y^{\prime} \in T$ and $X \in C$ such that $x_{\nu(a)}^ky_{1}^l \, = \, (XY)^\ast (XY^{\prime})$ and $(XY)^\ast (XY^{\prime})$ is forbidden. Again since $\Dot{E}$ is an LV-object, there is at least one more element $\widetilde{X} \in C$ other than $X$ such that $\widetilde{X}Y \neq 0$ and $\widetilde{X}Y^{\prime} \neq 0$. The proof now follows in exactly the same way as in $\mathbf{Cases~ 2.1}$ and $\mathbf{2.2}$.
 \end{proof}
\end{proposition}

\begin{corollary}
 Let $\dot{E}$ be a bi-separated graph satifying condition LV. Then
 \begin{enumerate}
     \item $\mathcal{A}_K(\dot{E})$ is nonsingular.
     \item $\dot{E}$ is connected implies $\mathcal{A}_K(\dot{E})$ is  semiprimitive.
     \item $\dot{E}$ is connected and non-empty implies $\mathcal{A}_K(\dot{E})$ is prime.
     \item $|E^1|\geq 1$ implies $\mathcal{A}_K(\dot{E})$ is not von Neumann regular.
 \end{enumerate}
\end{corollary}

\begin{theorem}\label{theorem: characterisation of domains}
 Let $\Dot{E}$ be a bi-separated graph. Then $\mathcal{A}_K(\Dot{E})$ is a domain if and only if $\Dot{E}$ satisfies domain condition.
\end{theorem}
 \begin{proof}
  If $\Dot{E}$ satisfies domain condition, then we consider the two following cases:\\
 
    \noindent $\mathbf{Case~1:}$ Assume that $X \in S\Rightarrow|X|>1$ and $Y\in T\Rightarrow|Y|>1$. If both $S$ and $T$ are empty, then $\mathcal{A}_K(\Dot{E})$ is a free unital $K$-algebra and hence a domain (since $K$ is a field). Otherwise, by the proposition \ref{prop: local valuation existence}, there is a local valuation on $\mathcal{A}_K(\Dot{E})$. So if $ab \, = \, 0$ in $\mathcal{A}_K(\Dot{E})$, then $\nu(ab)=-\infty$, which implies $\nu(a)+\nu(b)=-\infty$. This means that $\nu(a)=-\infty$ or $\nu(b)=-\infty$. Hence $a=0$ or $b=0$. Therefore $\mathcal{A}_K(\Dot{E})$ is a domain.
    
    \noindent $\mathbf{Case~2:}$ The only remaining cases to be considered are when $S=\{X\}$ with $X=\{e\}$ or $T=\{Y\}$ with $Y=\{f\}$. In both these cases the relations imposed on $K(\widehat{E})$ are not of the form $ab=0$. 
  
For converse,  
  if $|E^0| > 1$, then obviously $\mathcal{A}_K(\Dot{E})$ is not a domain. Otherwise, we consider the following cases separately:\\
  
  
  \noindent $\mathbf{Case~1:}$ Assume that there are two distinct elements $X,X^{\prime} \in S$ which have only one common $Y \in D$ such that $XY \neq 0$, $X^{\prime}Y \neq 0$. Then $(XY)(X^{\prime}Y)^\ast  \, = \, \delta_{XX^{\prime}}s(X) \, = \, 0$. So we are done.\\
  \noindent $\mathbf{Case~2:}$ Assume that there are two distinct elements $Y,Y^{\prime} \in T$ which have only one common $X \in C$ such that $XY \neq 0$, $XY^{\prime} \neq 0$. Then $(XY)^\ast (XY^{\prime}) \, = \, \delta_{YY^{\prime}}r(Y) \, = \, 0$.\\
  
  This completes the proof.
 \end{proof}

\subsection{The Gelfand-Kirillov dimension}\label{section: GK dimension}\hfill\\

We first recall some basic facts on the growth of algebras from \cite{MR1721834}. Suppose $B$ is a finitely generated $K$-algebra. Choose a finite generating set of $B$ and let $V$ be the $K$-subspace of $B$ spanned by this chosen generating set. For each natural number $n$, let $V^n$ denote the $K$-subspace of $B$ spanned by all words in $V$ of length less than or equal to $n$. In particular, $V^1\,=\,V$. Then we have an ascending chain
\begin{center}
  $ K \subseteq V^1 \subseteq V^2 \subseteq \ldots \subseteq V^n \subseteq \ldots$
\end{center}
of finite dimensional $K$-subspaces of $B$ such that $B\,=\, \bigcup_{n \in \mathbb{N}_0} V^n$, where, by convention, $V^0\,=\,K$. Clearly, the sequence $\{\text{dim}_K(V^n)\}$ is a montonically increasing sequence and the asymptotic behaviour (see the definition \ref{Asymptotic_Equivalence}) of this sequence provides an invariant of the algebra $B$, called the $\textit{Gelfand-Kirillov dimension}$ of $B$, which is defined to be
\begin{equation}\label{GK-Dimension}
   \text{GKdim B} = \overline{\text{lim}}\frac{\text{log}~  \text{dim}_{\text{K}}(V^n)}{\text{log n}}.
\end{equation}

\begin{definition}\label{Asymptotic_Equivalence}
 Given two eventually monotonically increasing functions $\phi,\psi: \mathbb{N} \rightarrow \mathbb{R}^+$, we say $\phi \preceq \psi$ if there are natural numbers $a$ and $b$ such that $\phi(n) \leq a\psi(bn)$, for almost all $n \in \mathbb{N}$. We say \textit{$\phi$ is asymptotically equivalent to $\psi$}, if both $\phi \preceq \psi$ and $\psi \preceq \phi$. If $\phi$ and $\psi$ are asymptotically equivalent, we write $\phi \sim \psi$.
\end{definition}

Coming back to GK dimension of algebras, if a $K$-algebra $B$ has two distinct finite generating sets, and if $V$ and $W$ are the finite dimensional subspaces of $B$ spanned by these sets, then setting $\phi(n) = \text{dim}_{\text{K}}(V^n)$ and $\psi(n) \,=\, \text{dim}_{\text{K}}(W^n)$, one can show that $\phi \sim \psi$ \cite[Lemma 1.1]{MR1721834}. In this notation, if $\phi \preceq n^m$ for some $m \in \mathbb{N}$, then $B$ is said to have $\textit{polynomial growth}$ and in this case GKdim($B$)$\leq m$. If on the other hand, $\phi \sim a^n$ for some $a \in \mathbb{R}$ such that $a > 1$, then $B$ is said to have $\textit{exponential growth}$ and in this case GKdim($B$)$=\infty$.
\begin{definition}\label{def: Cond A prime}
A bi-separated graph $\dot{E}$ is said to satisfy \textit{Condition $(A^\prime)$} if 
\begin{enumerate}
 \item[$(A^\prime1)$:] $S=T=\emptyset$ implies either $|E^1|>0$ or $|E^0|=\infty$.
    \item[$(A^\prime2)$:]$S \neq \emptyset$ or $T \neq \emptyset$ implies at least one of the following  holds:
     \begin{enumerate}
         \item $\exists X_1,X_2\in S,~X_1\neq X_2$, $s(X_1)=s(X_2)$ and $Y\in D$ such that for $i\in\{1,2\},$ $X_i\cap Y\neq\emptyset$ and $(X_iY),~(X_iY)^\ast$ are not part of any forbidden word.
         \item $\exists Y_1,Y_2\in T,~Y_1\neq Y_2$, $r(Y_1)=r(Y_2)$ and $X\in C$ such that for $i\in\{1,2\},$ $Y_i\cap X\neq\emptyset$ and $(Y_iX),~(Y_iX)^\ast$ are not part of any forbidden word.
         \item $\exists X\in S,~Y\in D$ such that $X\cap Y\neq\emptyset$, $s(X)=r(Y)$ and $(XY)$,$(XY)^\ast$ are not part of any forbidden word.
         \item $\exists Y\in T,~X\in C$ such that $X\cap Y\neq\emptyset$, $s(X)=r(Y)$ and $(XY)$,$(XY)^\ast$ are not part of any forbidden word.
     \end{enumerate}
     \end{enumerate}
\end{definition}
\begin{proposition}
If $\dot{E}$ is a finite bi-separated graph and satisfies condition $A^\prime$ then $\mathcal{A}_K(\dot{E})$ has exponential growth.
\end{proposition}

\begin{definition}\label{quasi-cycle}\cite[Definition 20,21]{preusser2019leavitt}
 Let $\dot{E}$ be a bi-separated graph. A \textit{quasi-cycle} is a normal generalized path $p$ in $\widehat{E}$ such that $p^2$ is normal and none of the sub-words of $p^2$ of length less than $|p|$ is normal. A quasi-cycle $p$ is called \textit{self-connected} if there is a normal path $o$ in $\widehat{E}$ such that $p$ is not a prefix of $o$ and $pop$ is normal. 
\end{definition}

\begin{theorem}
 Let $\dot{E}$ be a finite bi-separated graph. Then $\mathcal{A}_K(\dot{E})$ has exponential growth if and only if there is a self-connected quasi-cycle. 
\end{theorem}

\begin{remark}
Let $\dot{E}$ be a \textit{tame} bi-separated graph and suppose that $\{\dot{E}_\lambda\mid\lambda\in\Lambda\}$ is a directed system of all finite sub bi-separated graphs of $\dot{E}$. By results of \cite[Section 3]{MR3795390} we have GKdim($\mathcal{A}_K(\dot{E})=\sup\limits_{\lambda}\text{GKdim}(\mathcal{A}_K(\dot{E}_\lambda))$.
\end{remark}

\subsection{Additional applications of Linear bases}\hfill\\
In this subsection we fix the following notations. Let $(R,I)$ be a $K$-algebra with enough idempotents. An element $a\in R$ is called \textit{homogeneous} if $a\in vRw$ for some $v,w\in I$. Let $B$ denote a $K$-basis for $R$ which consists of homogeneous elements and contains $I$. Let $l:B\rightarrow\mathbb{Z}^+$ be a map such that $l(b)=0\Leftrightarrow b\in I$. 

\begin{definition}
 An element $b\in B\cap vRw$ is called \textit{left adhesive} if $ab\in B$ for any $a\in B\cap Rv$ and \textit{right adhesive} if $bc\in B$ for any $c\in B\cap wR$. A \textit{left valuative basis element} is a left adhesive element $b\in B\cap eR$ such that $l(ab)=l(a)+l(b)$ for any $a\in B\cap Rv$. A \textit{right valuative basis element} is defined similarly. A \textit{valuative basis element} is an adhesive element $b\in B\cap vRw$ such that $l(abc)=l(a)+l(b)+l(c)$ for any $a\in B\cap Rv$ and $c\in B\cap wR$.
 \end{definition}
 
\begin{proposition}\cite[Proposition 53]{preusser2019leavitt}
 Suppose there exists a valuative basis element $b\in(B-I)\cap vRv$. Then $\dim_K(R)=\infty$, $R$ is not simple, neither left nor right Artinian and not von Newmann regular.
 \end{proposition}
 
\begin{definition}\label{cond: A}
A bi-separated graph $\dot{E}=(E,(C,S),(D,T))$ is said to satisfy \textit{Condition $(A)$} if 
 \begin{enumerate}
     \item[$(A1)$:] $S=T=\emptyset$ implies $|E^0|=\infty$ or $|E^1|>0$
     \item[$(A2)$:] $S \neq \emptyset$ or $T \neq \emptyset$ implies at least one of the following holds:
     \begin{enumerate}
         \item $\exists~X\in S,~Y\in D$ such that $X\cap Y\neq\emptyset$, $(XY)(XY)^\ast$ and $(XY)^\ast(XY)$ are not forbidden words.
         \item $\exists~Y\in T,~X\in C$ such that $X\cap Y\neq\emptyset$, $(XY)(XY)^\ast$ and $(XY)^\ast(XY)$ are not forbidden words.
     \end{enumerate}
 \end{enumerate}
\end{definition}

Let $B$ denote the set of all normal generalized paths of $\mathcal{A}_K(\dot{E})$. Let $l:B\rightarrow\mathbb{Z}^+$ denote the map which maps a path to its length. If $\dot{E}$ satisfy Condition $(A2)$ then we can choose either $X\in S,~Y\in D$ or $Y\in T,~X\in C$ such that $X\cap Y\neq\emptyset$ and $(XY)(XY)^\ast$ is not forbidden. Set $b=(XY)(XY)^\ast$, then $b$ is a valuative basis element. Hence we have the following corollary.
 
 \begin{corollary}
 Let $\dot{E}$ be a bi-separated graph that satisfies Condition $(A)$. Then $\dim_K(\mathcal{A}_K(\dot{E}))=\infty$, $\mathcal{A}_K(\dot{E})$ is not simple, neither left nor right Artinian and not von Neumann regular.
 \end{corollary}
 
 \begin{definition}
Let $b\in B\cap vRw$ and $b^\prime\in B\cap v^\prime Rw^\prime$. We say that $b$ and $b^\prime$ have \textit{no common left multiple} if there are no $a\in B\cap Rv,~a^\prime\in B\cap Rv^\prime$ such that $ab=a^\prime b^\prime$. We say that $b$ and $b^\prime$ have \textit{no common right multiple} if there are no $c\in B\cap wR,~c^\prime\in B\cap w^\prime R$ such that $bc=b^\prime c^\prime$.

An element $b\in B\cap vRw$ is called \textit{right cancellative} if $ab=cb\Rightarrow a=c$ for any $a,c\in B\cap Rv$ and \textit{left cancellative} if $ba^\prime=bc^\prime\Rightarrow a^\prime=c^\prime$ for any $a^\prime,c^\prime\in B\cap wR$.
 \end{definition}
 
\begin{proposition}\cite[Proposition 56]{preusser2019leavitt}
If there are elements $b,b^\prime\in B\cap vRv$ such that $b$ is adhesive and right cancellative, $b^\prime$ is left adhesive and $b$ and $b^\prime$ have no common left multiple, then $R$ is not left Noetherian. 

If there are elements $c,c^\prime\in B\cap vRv$ such that $c$ is adhesive and left cancellative, $c^\prime$ is right adhesive and $c$ and $c^\prime$ have no common left multiple, then $R$ is not right Noetherian. 

\end{proposition}

We have the following corollary which gives a necessary condition for $\mathcal{A}_K(\dot{E})$ to be a left or right Noetherian in terms of $\dot{E}$.

\begin{corollary}
Let $\dot{E}$ be a bi-separated graph that satisfies Condition $(A^\prime)$. Then $\mathcal{A}_K(\dot{E})$ is neither left nor right Noetherian.
\end{corollary}
\begin{proof}
We prove the statement only for conditions $(A^\prime2)(a)$ and $(A^\prime2)(c)$ leaving the other simple cases to the reader.

Suppose there exist $X_1,X_2\in S$, $X_1\neq X_2$, $s(X_1)=s(X_2)=v$ and $Y\in D$ such that for $i\in\{1,2\}$, $X_i\cap Y\neq\emptyset$ and $(X_iY)$, $(X_iY)^\ast$ are part of forbidden words. Then set $b_1:=(X_1Y)(X_1Y)^\ast$, $b_2:=(X_1Y)(X_2Y)^\ast$ and $b_3:=(X_2Y)(X_1Y)^\ast$. Then $b_1,b_2,b_3\in(B-E^0)\cap v\mathcal{A}_K(\dot{E})v$. It is easy to check that $b_1$ is adhesive and both left and right cancellative, $b_2$ is left adhesive, $b_3$ is right adhesive, $b_1,~b_2$ have no common left multiple and $b_1,~b_3$ have no common right multiple. Thus $\mathcal{A}_K(\dot{E})$ is neither left nor right Noetherian.

Now suppose that there exist $X\in S$ and $Y\in D$ such that $X\cap Y\neq\emptyset$, $s(X)=r(Y)=v$ and $(XY),~(XY)^\ast$ are not part of any forbidden word. Then both $(XY),~(XY)$ are in $(B-E^0)\cap v\mathcal{A}_K(\dot{E})v$, they are adhesive, both left and right cancellative and have neither left nor right common multiple. Therefore $\mathcal{A}_K(\dot{E})$ is neither left nor right Noetherian.
\end{proof}

\begin{center}
	\textbf{PART II}
\end{center}
In this part we specialize our attention to hypergraphs and study their Cohn-Leavitt path algebras.
\section{B-hypergraphs and their $H$-monoids}\label{section: B-hypergraphs}
We would like to recast the definition of hypergraphs in terms of bi-separated graphs (with distinguished subsets) so that we can study their Cohn-Leavitt path algebras. We note that a hyperegde $h$, for which both $I_h$ and $J_h$ are infinite, does not contribute to relations in the Leavitt path algebra. We ignore such hyperedges and all other hyperedges which do not contribute to Cohn-Leavitt path algebra relations from the definition of hypergraph. Hence we modify the definition of hypergraphs as follows.
 
\begin{definition}
 A \textit{B-hypergraph} is a pair $(\Dot{E},\Lambda)$, where $\Dot{E}=(E,(C,S),(D,T))$ is a bi-separated graph and $\Lambda$ is a nonempty indexing set whose elements are called the \textit{B-hyperedges}, such that for each $\lambda\in\Lambda$ there exists $\mathcal{X}_\lambda\subseteq C$ and $\mathcal{Y}_\lambda\subseteq D$ which further satisfy the following conditions:
 \begin{enumerate}
    \item $X\notin S$ and $Y\notin T \implies X\cap Y=\emptyset$,
    \item for any $\alpha,\beta\in\Lambda$ with $\alpha\neq\beta$, $X\in \mathcal{X_\alpha}$ and $Y\in\mathcal{Y_\beta} \implies X\cap Y=\emptyset$,
    \item for any $\lambda\in\Lambda$, $X\in \mathcal{X}_\lambda$ and $Y\in\mathcal{Y}_\lambda \implies X\cap Y\neq\emptyset$,
   \item $\Lambda=\Lambda_T^S\sqcup\Lambda_{\text{fin}}^S\sqcup \Lambda_\infty^S\sqcup\Lambda_T^\text{fin}\sqcup\Lambda_T^\infty$, and 
    \begin{eqnarray}
    S &=&\bigsqcup\limits_{\lambda\in\Lambda^S}\mathcal{X}_\lambda,\quad\text{where}\quad\Lambda^S=\Lambda_T^S\sqcup\Lambda_\text{fin}^S\sqcup\Lambda_\infty^S, \nonumber \\ 
   T &=& \bigsqcup\limits_{\lambda\in\Lambda_T}\mathcal{Y}_\lambda,\quad\text{where}\quad\Lambda_T=\Lambda_T^S\sqcup\Lambda_T^\text{fin}\sqcup\Lambda_T^\infty, \nonumber \\
  C_\text{fin} - S &=&\bigsqcup\limits_{\lambda\in\Lambda_T^\text{fin}}\mathcal{X}_\lambda, \nonumber \\ D_\text{fin} 
  - T &=& \bigsqcup\limits_{\lambda\in\Lambda_\text{fin}^S}\mathcal{Y}_\lambda, \nonumber \\
   C - C_\text{fin} &=& \bigsqcup\limits_{\lambda\in\Lambda_T^\infty}\mathcal{X}_\lambda, \nonumber \\  
   D - D_\text{fin} &=& \bigsqcup\limits_{\lambda\in\Lambda_\infty^S}\mathcal{Y}_\lambda. \nonumber 
  \end{eqnarray} 
   
\end{enumerate}
 \end{definition}

\begin{remark}\label{rem: category of B-hypergraphs}
\begin{enumerate}
    \item It is easy to check that B-hypergraphs are tame and that B-hypergraphs, along with complete morphisms form a category. This category will be denoted by \textbf{BHG}.
    
    \item Note that given a B-hypergraph $(\dot{E},\Lambda)$ with $S=C_\text{fin}$ and $T=D_\text{fin}$ we can identify $(\dot{E},\Lambda)$ with a hypergraph $\mathcal{H}$ as follows: $\mathcal{H}^0=E^0$, $\mathcal{H}^1=\Lambda$, and for each $\lambda\in\Lambda$ $s(\lambda)=(s(X))_{X\in\mathcal{X}_\lambda}$ and $r(\lambda)=(r(Y))_{Y\in\mathcal{Y}_\lambda}$. 
    
\end{enumerate}

\end{remark}

\begin{notation}
For $\lambda\in\Lambda_T$, set 
\begin{eqnarray} 
 Q_\lambda &=& \{q_Z\mid Z\subseteq\mathcal{Y}_\lambda,~0<|Z|<\infty\}~ \text{and} \nonumber \\
 Q &=& \bigsqcup\limits_{\lambda\in\Lambda_T}Q_\lambda. \nonumber
 \end{eqnarray}
For $\lambda\in\Lambda^S$, set  
\begin{eqnarray}  
 P_\lambda &=& \{p_W\mid W\subseteq\mathcal{X}_\lambda,~0<|W|<\infty\}~ \text{and} \nonumber \\
 P &=& \bigsqcup\limits_{\lambda\in\Lambda^S}P_\lambda. \nonumber
\end{eqnarray} 
\end{notation}
 
\begin{definition}\label{definition: H-monoid}
Given a B-hypergraph $(\Dot{E},\Lambda)$, its $H$-monoid $H(\dot{E},\Lambda)$ is defined as the abelian monoid generated by 
$E^0\sqcup Q\sqcup P$ modulo the following relations:
\begin{enumerate}
    \item For $\lambda\in\Lambda_T$ and $q_Z\in Q_\lambda$, $$\sum\limits_{X\in\mathcal{X}_\lambda}s(X)=\sum\limits_{Y\in Z}r(Y)+q_Z,$$
    \item For $\lambda\in\Lambda^S$ and $p_W\in P_\lambda$, $$\sum\limits_{Y\in\mathcal{Y}_\lambda}r(Y)=\sum\limits_{X\in W}s(X)+p_W,$$ 
    
    \item For $\lambda\in\Lambda_T$ and $q_{Z_1}, q_{Z_2}\in Q_\lambda$ with $Z_1\subsetneq Z_2$ 
    $$q_{Z_1}=q_{Z_2}+\sum\limits_{Y\in Z_2-Z_1}r(Y),$$
    \item For $\lambda\in\Lambda^S$ and $p_{W_1}, p_{W_2}\in P_\lambda$ with $W_1\subsetneq W_2$ 
    $$p_{W_1}=p_{W_2}+\sum\limits_{X\in W_2-W_1}s(X),$$
    \item for $\lambda\in\Lambda_T^S$,
    $$q_{\mathcal{Y}_\lambda}=0=p_{\mathcal{X}_\lambda}.$$ 
\end{enumerate}
\end{definition}

If $(\Dot{E},\Lambda)$ is B-hypergraph then $H(\dot{E},\Lambda)$ is a conical monoid. This is easy to see from the relations defining $H(\dot{E},\Lambda)$ because these relations ensure that $(x+y) \neq 0$ whenever $x \neq 0$ or $y \neq 0$, for $x,y \in H(\dot{E},\Lambda)$.

\begin{definition}\label{defn: V-monoid}
Let $R$ be a ring, and let $M_\infty(R)$ denote the set of all $\omega\times\omega$ matrices over $R$ with finitely many nonzero entries, where $\omega$ varies over $\mathbb{N}$. For idempotents $e,f\in M_\infty(R)$, the \textit{Murray-von Neumann equivalence} $\sim$ is defined as follows: $e\sim f$ if and only if there exists $x,y\in M_\infty(R)$ such that $xy=e$ and $yx=f$. 

Let $\mathcal{V}(R)$ denote the set of all equivalence classes $[e]$, for idempotents $e\in M_\infty(R)$. Define $[e]+[f]=[e\oplus f]$, where $e\oplus f$ denotes the block diagonal matrix $\begin{pmatrix}
e & 0\\ 0 & f\end{pmatrix}$. Under this operation, $\mathcal{V}(R)$ is an abelian monoid, and it is conical, that is, $a+b=0$ in $\mathcal{V}(R)$ implies $a=b=0$. Also  $\mathcal{V}(\_):$\textbf{Rings}$\rightarrow$\textbf{Mon} is a continuous functor.

Let $R$ be a unital ring and let $\mathcal{U}(R)$ be the set of all isomorphic classes of finitely generated projective left $R$-modules, endowed with direct sum as binary operation. Then $(\mathcal{U}(R),\oplus)$ is an abelian monoid. We also have $\mathcal{U}(R)\cong\mathcal{V}(R)$.
\end{definition}

\begin{theorem}\label{thm: V-monoid_H-monoid_Isomorphism}
 There is an isomorphism $\Gamma :H\rightarrow \mathcal{V}\circ \mathcal{A}_K$ of funtors $\emph{\textbf{BHG}}\rightarrow\emph{\textbf{Mon}}$.
\end{theorem}
\begin{proof}
  We first define the map $\Gamma$ as follows: For each object $\dot{E}$ in \textbf{BHG}, $$\Gamma(\dot{E},\Lambda) : H(\dot{E},\Lambda) \rightarrow \mathcal{V}\circ \mathcal{A}_K(\dot{E},\Lambda)$$ is the monoid homomorphism sending $$v \mapsto [v],$$ $$q_Z \mapsto [\text{diag}(s(X))-BB^{*}]$$ and $$p_W \mapsto [\text{diag}(r(Y))-N^{*}N],$$ where $v \in E^0$, $Z$ is any non-empty finite subset of $\mathcal{Y}_{\lambda}$, diag$(s(X))$ is the diagonal matrix of order $|\mathcal{X}_{\lambda}|$ with diagonal entries coming from the set $s(\mathcal{X}_{\lambda})$ in any order (without repetition), diag$(r(Y))$ is the diagonal matrix of order $|\mathcal{Y}_{\lambda}|$ with diagonal entries coming from the set $r(\mathcal{Y}_{\lambda})$ in any order (without repetition), $B$ is the $|\mathcal{X}_{\lambda}| \times |Z|$ matrix whose columns are precisely the ones in $Z$ and whose $i^{th}$ row consists elements of $X$ if and only if the diagonal entry in the $i^{th}$ row of diag$(s(X))$ is $s(X)$, and $N$ is the $|W| \times |\mathcal{Y}_{\lambda}|$ matrix whose rows are precisely the ones in $W$ and whose $j^{th}$ column has elements from $Y$ if and only if the diagonal entry in the $j^{th}$ column of diag$(r(Y))$ is $r(Y)$.
  
  It is not hard to see that the above map is well defined. Also the fact that every element in \textbf{BHG} is a direct sum of its finite complete subobjects and the continuity of the functors involved will suggest that it is enough to prove the results for finite subobjects. For the finite case, we use induction on $|\Lambda|$. For $\Lambda = \emptyset$, the result is trivial. So, suppose that the result holds for all finite objects $(\dot{F},\Lambda_{\Dot{F}})$ in \textbf{BHG} with $|\Lambda_{\dot{F}}| \leq (n-1)$, for some $n \geq 1$. Let $(\dot{E},\Lambda_{\dot{E}})$ be a finite object with $|\Lambda_{\dot{E}}| = n$. Fix an element $\lambda \in \Lambda_{\dot{E}}$. We can now apply induction to the object $\dot{F}$ obtained from the $\dot{E}$ by deleting all the edges in $\mathcal{X}_{\lambda}$ and leaving the remaining structure as it is, keeping $F^0 = E^0$. 
  
  First suppose that $\lambda \in \Lambda_{T}^{S}$. Then $H(\dot{E},\Lambda_{\dot{E}})$ is obtained from $H(\dot{F},\Lambda_{\dot{F}})$ by going modulo the relation $\sum\limits_{X\in\mathcal{X}_\lambda}s(X) = \sum\limits_{Y\in\mathcal{Y}_\lambda}r(Y)$. Also, the algebra $\mathcal{A}_{K}(\dot{E},\Lambda_{\dot{E}})$ is the Bergman algebra obtained from $\mathcal{A}_{K}(\dot{F},\Lambda_{\dot{F}})$ by adjoining a universal isomorphism between the finitely generated projective modules $\bigoplus\limits_{X\in\mathcal{X}_\lambda}\mathcal{A}_{K}(\dot{F},\Lambda_{\dot{F}})s(X)$ and $\bigoplus\limits_{Y\in\mathcal{Y}_\lambda}\mathcal{A}_{K}(\dot{F},\Lambda_{\dot{F}})r(Y)$. So by \cite[Theorem 5.2]{MR0357503}, $\mathcal{V}(\mathcal{A}_{K}(\dot{E},\Lambda_{\dot{E}}))$ is the quotient of $\mathcal{V}(\mathcal{A}_{K}(\dot{F},\Lambda_{\dot{F}}))$ modulo the relation $$[\text{diag}(s(X))] = [\text{diag}(r(Y))].$$ Since the map $$\Gamma(\dot{F},\Lambda_{\dot{F}}) : H(\dot{F},\Lambda_{\dot{F}}) \rightarrow \mathcal{V}\circ \mathcal{A}_K(\dot{F},\Lambda_{\dot{F}})$$ is an isomorphism by induction hypothesis, the desired result follows.
  
  Now suppose $\lambda$ does not belong to $\Lambda_{T}^{S}$. Then it is either in $\Lambda_{T}^{\infty}$ or in $\Lambda_{\infty}^{S}$. Let us first assume that $\lambda \in \Lambda_{T}^{\infty}$. In this case, $H(\dot{E},\Lambda_{\dot{E}})$ is obtained from $H(\dot{F},\Lambda_{\dot{F}})$ by adjoining a new generator $q_{\mathcal{Y}_\lambda}$ and going modulo the relation $$\sum\limits_{X\in\mathcal{X}_\lambda}s(X) = \sum\limits_{Y\in\mathcal{Y}_\lambda}r(Y) + q_{\mathcal{Y}_\lambda}.$$ On the algebra side, $\mathcal{A}_{K}(\dot{E},\Lambda_{\dot{E}})$ is obtained from $\mathcal{A}_{K}(\dot{F},\Lambda_{\dot{F}})$ in two steps by 
  \begin{enumerate}
  \item first adjoining the mutually perpendicular idempotents diag$(s(X))-BB^{*}$ and  $q_{\mathcal{Y}_\lambda}^{\prime}$, and going modulo the relation $$[\text{diag}(s(X))] = [BB^{*}] + q_{\mathcal{Y}_\lambda}^{\prime},$$ thereby, getting a new algebra $R$ and then
  \item adjoining a universal isomorphism between the left module corresponding to $[BB^{*}]$ and the left module $\bigoplus\limits_{Y \in \mathcal{Y}_{\lambda}}Rr(Y)$.
  \end{enumerate}
  So, by \cite[Theorems 5.1, 5.2]{MR0357503}, $\mathcal{V}(\mathcal{A}_{K}(\dot{E},\Lambda_{\dot{E}}))$ is obtained from $\mathcal{V}(\mathcal{A}_{K}(\dot{F},\Lambda_{\dot{F}}))$ by adjoining a new generator $q_{\mathcal{Y}_\lambda}^{\prime\prime}$ and going modulo the relation $$[\text{diag}(s(X))] = [\text{diag}(r(Y))] + q_{\mathcal{Y}_\lambda}^{\prime\prime}.$$
  This, along with the induction hypothesis, proves the theorem for the considered case.
  
  Finally suppose $\lambda \in \Lambda_{\infty}^{S}$. Again $H(\dot{E},\Lambda_{\dot{E}})$ is obtained from $H(\dot{F},\Lambda_{\dot{F}})$ by adjoining a new generator $p_{\mathcal{X}_\lambda}$ and going modulo the relation $$\sum\limits_{Y\in\mathcal{Y}_\lambda}r(Y) = \sum\limits_{X\in\mathcal{X}_\lambda}s(X) + p_{\mathcal{X}_\lambda}.$$ On the other hand, analogous to the previous case, the algebra $\mathcal{A}_{K}(\dot{E},\Lambda_{\dot{E}})$ is obtained from $\mathcal{A}_{K}(\dot{F},\Lambda_{\dot{F}})$ in two steps by 
  \begin{enumerate}
  \item first adjoining the mutually perpendicular idempotents diag$(r(Y)) - N^{*}N$ and  $p_{\mathcal{X}_\lambda}^{\prime}$, and going modulo the relation $$[\text{diag}(r(Y))] = [N^{*}N] + p_{\mathcal{X}_\lambda}^{\prime},$$ thereby, getting a new algebra $R^{\prime}$ and then
  \item adjoining a universal isomorphism between the left module corresponding to $[N^{*}N]$ and the left module $\bigoplus\limits_{X \in \mathcal{X}_{\lambda}}R^{\prime}s(X)$.
  \end{enumerate}
  So, by \cite[Theorems 5.1, 5.2]{MR0357503}, $\mathcal{V}(\mathcal{A}_{K}(\dot{E},\Lambda_{\dot{E}}))$ is obtained from $\mathcal{V}(\mathcal{A}_{K}(\dot{F},\Lambda_{\dot{F}}))$ by adjoining a new generator $p_{\mathcal{X}_\lambda}^{\prime\prime}$ and going modulo the relation $$[\text{diag}(r(Y))] = [\text{diag}(s(X))] + p_{\mathcal{X}_\lambda}^{\prime\prime},$$
  thereby completing the proof (using induction hypothesis).
  
  \end{proof}
\begin{remark}
 We note that if $M$ is any conical abelian monoid then there exists a B-hypergraph $(\dot{E},\Lambda_{\dot{E}})$ such that $M\cong H(\dot{E},\Lambda)\cong \mathcal{V(A}_K(\dot{E},\Lambda_{\dot{E}}))$. For two different proofs of this fact, we refer the reader to \cite[Proposition 4.4]{MR2980456} or  \cite[Proposition 62]{preusser2019leavitt}. 
\end{remark}

\section{Ideal lattices and Simplicity}\label{section: ideal lattice}
In this section $(\Dot{E},\Lambda)$ always denotes a B-hypergraph. Throughout this section we use the following notation: For $\lambda\in\Lambda,$ $$s(\lambda):=\bigcup\limits_{X\in\mathcal{X}_\lambda}s(X)~\text{and}~r(\lambda):=\bigcup\limits_{Y\in\mathcal{Y}_\lambda}r(X).$$

\subsection{The lattice of admissible triples in $(\dot{E},\Lambda)$}
\begin{definition}
A subset $V$ of $E^0$ is called \textit{bisaturated} if for each $\lambda\in\Lambda_T^S$, $$s(\lambda)\subseteq V \iff r(\lambda)\subseteq V.$$
The set of all bisaturated subsets of $E^0$ is denoted by BS$(\Dot{E},\Lambda)$.

Note that empty set and $E^0$ are always elements of BS$(\Dot{E},\Lambda)$. It is easy to check that BS$(\Dot{E},\Lambda)$ is closed under arbitrary intersections.

If $V$ is a subset of $E^0$, the \textit{bisaturated closure of $V$}, denoted $\overline{V}$, is the smallest bisaturated subset of $E^0$ containing $V$. Since the intersection of bisaturated subsets of $E^0$ is again bisaturated, $\overline{V}$ is well defined. 
\end{definition}

For $V\subseteq E^0$, $\overline{V}$ can be explicitly constructed as follows:
Define $V_0 = V$. If $n$ is odd positive integer, define
$$V_n = V_{n-1}\cup\{r(Y)\mid Y\in\mathcal{Y}_\lambda, \lambda\in\Lambda_T^S,~\text{and}~s(\lambda)\subseteq V_{n-1}\},$$
and if $n$ is even positive integer, define
$$V_n = V_{n-1}\cup\{s(X)\mid X\in\mathcal{X}_\lambda, \lambda\in\Lambda_T^S,~\text{and}~r(\lambda)\subseteq V_{n-1}\}.$$
Then $\overline{V}=\bigcup_{n\geq 0}V_n$.

\begin{definition}
Let $V\subseteq E^0$ be bisaturated and for any $\lambda\in\Lambda$, set
$$\mathcal{X}_{\lambda/V}=\{X\in\mathcal{X}_\lambda\mid s(X)\notin V\}~\text{and}~\mathcal{Y}_{\lambda/V}=\{Y\in\mathcal{Y}_\lambda\mid r(Y)\notin V\}.$$
Then set
$$\Lambda/V=\Lambda_T^S/V\sqcup\Lambda_\text{fin}^S/V\sqcup\Lambda_\infty^S/V\sqcup\Lambda_T^\text{fin}/V\sqcup\Lambda_T^\infty/V,$$ where
\begin{eqnarray*}
\Lambda_T^S/V &:= & \{\lambda\in\Lambda_T^S\mid0<|\mathcal{X}_{\lambda/V}|\}=\{\lambda\in\Lambda_T^S\mid0<|\mathcal{Y}_{\lambda/V}|\},\\\Lambda_\text{fin}^S/V &:=& \{\lambda\in\Lambda_\text{fin}^S,\mid0<|\mathcal{X}_{\lambda/V}|\},\\
\Lambda_\infty^S/V &:=& \{\lambda\in\Lambda_\infty^S\mid0<|\mathcal{X}_{\lambda/V}|<\infty\},\\
\Lambda_T^\text{fin}/V &:=& \{\lambda\in\Lambda_T^\text{fin}\mid0<|\mathcal{Y}_{\lambda/V}|\},\\
\Lambda_T^\infty/V &:=& \{\lambda\in\Lambda_T^\infty\mid0<|\mathcal{Y}_{\lambda/V}|<\infty\}.
\end{eqnarray*}

Let $V\subseteq E^0$ be a bisaturated set, $\Sigma\subseteq \Lambda_\text{fin}^S/V\sqcup\Lambda_\infty^S/V$ and $\Theta\subseteq\Lambda_T^\text{fin}/V\sqcup\Lambda_T^\infty/V$.
A triple $(V,\Sigma,\Theta)$ is called an \textit{admissible triple} and the set of all admissible triples in $(\dot{E},\Lambda)$ is denoted by AT$(\Dot{E},\Lambda)$.

We define a relation $\leq$ in AT$(\dot{E},\Lambda)$ as follows:
$(V_1,\Sigma_1,\Theta_1)\leq(V_2,\Sigma_2,\Theta_2)$ if $$V_1\subseteq V_2,$$
$$\Sigma_1\subseteq \Sigma_2\sqcup\Lambda^S(V),~\text{where}~\Lambda^S(V)=\Lambda^S-\Lambda/V,$$
$$\Theta_1\subseteq \Theta_2\sqcup\Lambda_T(V),~\text{where}~\Lambda_T(V)=\Lambda_T-\Lambda/V.$$
\end{definition}
We note that $(E^0,\emptyset,\emptyset)$ is the maximum and $(\emptyset,\emptyset,\emptyset)$ is the minimum in AT$(\Dot{E},\Lambda)$.

\begin{definition}
Let $V$ be a bisaturated subset of $E^0$, $\Sigma\subseteq\Lambda^S(V)\sqcup\Lambda_\text{fin}^S/V\sqcup\Lambda_\infty^S/V$, and $\Theta\subseteq\Lambda_T(V)\sqcup\Lambda_T^\text{fin}/V\sqcup\Lambda_T^\infty/V$. The \textit{$(\Sigma,\Theta)$-bisaturation of $V$} is defined as the smallest bisaturated subset $\overline{V}(\Sigma,\Theta)$ of $E^0$ containing $H$ such that
\begin{enumerate}
    \item If $\lambda\in\Sigma$ and $s(\lambda)\subseteq \overline{V}(\Sigma,\Theta)$, then $r(\lambda)\subseteq\overline{V}(\Sigma,\Theta)$ and
    \item If $\lambda\in\Theta$ and $r(\lambda)\subseteq \overline{V}(\Sigma,\Theta)$, then $s(\lambda)\subseteq\overline{V}(\Sigma,\Theta)$.
\end{enumerate}
\end{definition}

We can construct $(\Sigma,\Theta)$-saturation of $V$ as follows- Define $\overline{V}_0(\Sigma,\Theta)=V$.
If $n$ is odd positive integer, define
$$\overline{V}_n(\Sigma,\Theta) = \overline{V}_{n-1}(\Sigma,\Theta)
 \cup \{r(Y)\in E^0-\overline{V}_{n-1}(\Sigma,\Theta)\mid Y\in\mathcal{Y}_{\lambda},\lambda\in\Lambda_\text{fin}^S\cup\Sigma~\text{and}~s(\lambda)\subseteq \overline{V}_{n-1}(\Sigma,\Theta)\},$$
and if $n$ is even positive integer, define
$$\overline{V}_n(\Sigma,\Theta)=\overline{V}_{n-1}(\Sigma,\Theta)\cup\{s(X)\in E^0-\overline{V}_{n-1}(\Sigma,\Theta)\mid X\in\mathcal{X}_{\lambda},\lambda\in\Lambda_T^\text{fin}\cup\Theta~\text{and}~r(\lambda)\subseteq \overline{V}_{n-1}(\Sigma,\Theta)\}.$$
Then $\overline{V}(\Sigma,\Theta)=\bigcup_{n\geq 0}\overline{V}_n(\Sigma,\Theta)$.

\begin{proposition}
 $(\emph{AT}(\Dot{E},\Lambda),\leq)$ is a lattice, with supremum $\vee$ and infimum $\wedge$ given by
 $$(V_1,\Sigma_1,\Theta_1)\vee(V_2,\Sigma_2,\Theta_2)=(\widetilde{V},\widetilde{\Sigma},\widetilde{\Theta}),$$ where
 $$\widetilde{V}=\overline{V_1\cup V_2}(\Sigma_1\cup\Sigma_2,\Theta_1\cup\Theta_2),$$
 $$\widetilde{\Sigma}=(\Sigma_1\cup \Sigma_2)-\Lambda^S(\widetilde{V}),$$ 
 $$\widetilde{\Theta}=(\Theta_1\cup \Theta_2)-\Lambda_T(\widetilde{V}),$$ and
 
 $$(V_1,\Sigma_1,\Theta_1)\wedge(V_2,\Sigma_2,\Theta_2)=(\widehat{V},\widehat{\Sigma},\widehat{\Theta}),$$ where
 $$\widehat{V}=(V_1\cap V_2),$$
 $$\widehat{\Sigma}=(\Sigma_1\cup\Lambda^S(V))\cap(\Sigma_2\cup\Lambda^S(V))\cap(\Lambda_\emph{fin}^S/V\sqcup\Lambda_\infty^S/V),$$
 $$\widehat{\Theta}=(\Theta_1\cup\Lambda_T(V))\cap(\Theta_2\cup\Lambda_T(V))\cap(\Lambda_T^\emph{fin}/V\sqcup\Lambda_T^\infty/V).$$ 
 \end{proposition}
 \begin{proof}
 Clearly, $(\widetilde{V},\widetilde{\Sigma},\widetilde{\Theta})\in$ AT$(\Dot{E},\Lambda)$ and is greater than $(V_i,\Sigma_i,\Theta_i)$ for $i=1,2$. Now let $(V,\Sigma,\Theta)\in$ AT$(\Dot{E},\Lambda)$ such that $(V_i,\Sigma_i,\Theta_i)\leq(V,\Sigma,\Theta)$ for $i=1,2$. It is enough to prove that $\widetilde{V}\subseteq V$ for all $n\in\mathbb{Z}^+$. We do this inductively. Define $\widetilde{V}_n=\overline{(V_1\cup V_2)}_n(\Sigma_1\cup\Sigma_2,\Theta_1\cup\Theta_1)$. For $n=0$ the claim is clear by assumption. Now assume that $n\geq 1$ and that $\widetilde{V}_{n-1}\subseteq V$. Let $v\in\widetilde{V}_n$. If $v\in s(\lambda)$ or $v\in r(\lambda)$ for $\lambda\in\Lambda_T^S$, then $v\in V$ because $V$ is bisaturated. Now suppose $v\in s(\lambda)$ for $\lambda\in\Theta_1\cup\Theta_2$. By definition and the induction hypothesis, we have 
 $r(\lambda)\subseteq \widetilde{V}_m\subseteq V$, where $m$ is largest even integer less than $n$. In particular, this implies that $\lambda\notin \Theta$. Since $\lambda\in\Theta_1\cup\Theta_2\subseteq \Lambda_T(V)\cup \Theta$ we conclude that $v\in H$, which completes the induction step. The inclusion $(\Theta_1\cup\Theta_2)-\Lambda_T(\widetilde{V})\subseteq\Theta$ follows. Similar arguments shows that if $v\in r(\lambda)$ for $\lambda\in\Sigma_1\cup\Sigma_2$, then $v\in V$ and $(\Theta_1\cup\Theta_2)-\Lambda_T(\widetilde{V})\subseteq\Theta$.
 
 It is clear that $(\widehat{V},\widehat{\Sigma},\widehat{T})\in$ AT$(\Dot{E},\Lambda)$ and $(\widehat{V},\widehat{\Sigma},\widehat{T})\leq(V_i,\Sigma_i,\Theta_i)$ for $i=1,2$. If $(V,\Sigma,\Theta)\in$ AT$(\Dot{E},\Lambda)$ such that $(V,\Sigma,\Theta)\leq(V_i,\Sigma_i,\Theta_i)$ for $i=1,2$, then clearly $V\subseteq\widehat{V}$. Consider $\lambda\in\Theta-\Lambda_T(\widehat{V})$. Then there exists $v\in s(\lambda)-\widehat{V}$, so $v\notin V_j$ for some $j\in\{1,2\}$, and $\lambda\notin\Lambda_T(V_j)$. Let us fix $j=1$.  Since $(V,\Sigma,\Theta)\leq(V_1,\Sigma_1,\Theta_1)$, it follows that $\lambda\in\Theta_1$. Hence, $\mathcal{Y}_{\lambda/V_1}$ is nonempty, hence $\mathcal{Y}_{\lambda/\widehat{V}}$ is nonempty. On the other hand, $\lambda\in\Theta$ implies that $\mathcal{Y}_{\lambda/V}$ is finite, hence $\mathcal{Y}_{\lambda/\widehat{V}}$ is finite. Thus $\lambda\in\Lambda_T^\text{fin}/\widehat{V}\sqcup\Lambda_T^\infty/\widehat{V}$. We also have $\lambda\in\Theta_i\sqcup\Lambda_T(V_i)$ for $i=1,2$, because $(V,\Sigma,\Theta)\leq(V_i,\Sigma_i,\Theta_i)$ for $i=1,2$, and consequently $\lambda\in\widehat{\Theta}$. This shows $\Theta\subseteq\widehat{\Theta}\sqcup\Lambda_T(V)$. Similarly we can show that $\Sigma\subseteq\widehat{\Sigma}\sqcup\Lambda^S(V)$ proving that $(V,\Sigma,\Theta)\leq(\widehat{V},\widehat{\Sigma},\widehat{\Theta})$. This shows that $(\widehat{V},\widehat{\Sigma},\widehat{\Theta})$ is the infimum required.
 
 Hence AT$(\Dot{E}\Lambda)$ is a lattice.
 \end{proof}
 
\subsection{The lattice of order-ideals in $H(\dot{E},\Lambda)$}

\begin{definition}An \textit{order-ideal} of a monoid $M$ is a submonoid $I$ of $M$ such that $x+y\in I$ for some $x,y\in M$ implies that both $x$ and $y$ belong to $I$. 
\end{definition}
Every monoid $M$ is equipped with a pre-order $\leq$ as follows: for $x,y\in M$, $x\leq y$ if and only if there exists $z\in M$ such that $x+z=y$. Hence an equivalent definition of an order-ideal $I$ is as follows: For each $x,y\in M$, if $x\leq y$ and $y\in I$ then $x\in I$.

Let $\mathcal{L}(M)$ denote the set of all order-ideals of $M$. We note that $\mathcal{L}(M)$ is closed under arbitrary intersections. For a submonoid $J$ of $M$, let $\langle J\rangle$ consists of those elements $x\in M$ such that $x\leq y$ for some $y\in J$. Then $\langle J \rangle$ denote the order-ideal generated by $J$. Then $\mathcal{L}(M)$ can be shown to a complete lattice with respect to inclusion. For, an arbitrary family $\{I_i\}$ of order-ideals of $M$, the supremum is given by $\langle \sum I_i \rangle$.

We want to study the lattice of order-ideals of $H(\dot{E},\Lambda)$. For convenience we modify some notations in the previous section as follows.
\begin{notation}
$$\text{For}~ \lambda\in\Lambda_T,\quad\textbf{s}(\lambda):=\sum\limits_{X\in\mathcal{X}_\lambda}s(X).$$
$$\text{For}~ \lambda\in\Lambda^S,\quad\textbf{r}(\lambda):=\sum\limits_{Y\in\mathcal{Y}_\lambda}r(Y).$$
Note that the above sums are finite.
$$\text{For}~\lambda\in\Lambda_T^\text{fin},\quad q_\lambda:=q_{\mathcal{Y}_\lambda}.$$ 
$$\text{For}~\lambda\in\Lambda_\text{fin}^S,\quad p_\lambda:=p_{\mathcal{X}_\lambda}.$$

Also, $$\text{for}~\lambda\in\Lambda_T^\infty,~\text{set}~\mathcal{Z}_\lambda = \{Z\mid Z\subseteq\mathcal{Y}_\lambda, 0<|Z|<\infty\},$$ 
$$\text{for}~\lambda\in\Lambda_\infty^S,~\text{set}~\mathcal{W}_\lambda = \{W\mid W\subseteq\mathcal{X}_\lambda, 0<|W|<\infty\}.$$
$$\mathcal{Z} := \bigsqcup\limits_{\lambda\in\Lambda_T^\infty}\mathcal{Z}_\lambda~\text{and}~\mathcal{W} := \bigsqcup\limits_{\lambda\in\Lambda_\infty^S}\mathcal{W}_\lambda.$$

Finally set
$$Q^0 = \{q_\lambda\mid\lambda\in\Lambda_T^\text{fin}\}\sqcup\{q_{Z}\mid Z\in\mathcal{Z}\},$$
$$P^0 = \{p_\lambda\mid\lambda\in\Lambda_\text{fin}^S\}\sqcup\{p_{W}\mid W\in\mathcal{W}\}.$$
\end{notation}

\begin{definition}
Let $F$ be the free abelian monoid on $E^0\sqcup Q^0\sqcup P^0$. We identify $H(\dot{E},\Lambda)$ with $F/\sim$, where $\sim$ is the congruence on $F$ generated by the relations 
$$\textbf{s}(\lambda)\sim\begin{cases}
\textbf{r}(\lambda)\qquad\quad\text{if}~\lambda\in\Lambda_T^S,\\
\textbf{r}(\lambda)+q_\lambda\quad\text{if}~\lambda\in\Lambda_T^\text{fin}, \text{and}\\
\textbf{r}(\lambda)+q_Z\quad\text{if}~\lambda\in\Lambda_T^\infty~\text{and}~Z\in\mathcal{Z}_\lambda,\\
\end{cases}$$
and
$$\textbf{r}(\lambda)\sim\begin{cases}
\textbf{s}(\lambda)+p_\lambda\quad\text{if}~\lambda\in\Lambda_\text{fin}^S,\\
\textbf{s}(\lambda)+p_W\quad\text{if}~\lambda\in\Lambda_\infty^S~\text{and}~W\in\mathcal{W}_\lambda,\\
\end{cases}$$
for $Z_1,Z_2\in\mathcal{Z}$ with $Z_1\subsetneq Z_2$, and
$q_{Z_1}\sim q_{Z_2}+\textbf{r}(Z_2-Z_1),$
and for $W_1,W_2\in\mathcal{W}$ with $W_1\subsetneq W_2$, and
$p_{W_1}\sim p_{W_2}+\textbf{s}(W_2-W_1).$

\end{definition}

\begin{lemma}
If $I$ is an oder-ideal of $H(\dot{E},\Lambda)$, then the set $V=I\cap E^0$ is bisaturated.
\end{lemma}
\begin{proof}
  Let $\lambda\in\Lambda_T^S$ and $r(\lambda)\subseteq V$, then $\textbf{r}(\lambda)=\textbf{s}(\lambda)\in I$. Since $I$ is order-ideal, and $s(X)\leq\textbf{s}(\lambda)$ for each $X\in\mathcal{X}_\lambda$, we have $s(X)\in I$ for each $X\in\mathcal{X}_\lambda$, and hence $s(\lambda)\subseteq V$. Converse follows similarly.
\end{proof}

\begin{definition}\label{definition: correspondence between order-ideals and admissible triples}
Let $V$ be a bisaturated subset of $E^0$.
$$\text{For}~\lambda\in\Lambda_T^\infty/V,~\text{if}~0<|\mathcal{Y}_{\lambda/V}|<\infty, \quad 
q_{\lambda/V}:=q_{\mathcal{X}_{\lambda/V}}.$$

$$\text{For}~\lambda\in\Lambda_\infty^S/V,~\text{if}~0<|\mathcal{X}_{\lambda/V}|<\infty,\quad 
p_{\lambda/V}:=p_{\mathcal{X}_{\lambda/V}}$$
If $I$ is an order-ideal of $V(\dot{E},\Lambda)$, set $\psi(I) = (V,\Sigma,\Theta),$ where
\begin{eqnarray*}
V &:=& I\cap E^0,\\
\Sigma &:=& \{\lambda\in\Lambda_\text{fin}^S/V\mid p_\lambda\in I\}\sqcup\{\lambda\in\Lambda_\infty^S/V\mid p_{\lambda/V}\in I\},~\text{and}\\
\Theta &:=& \{\lambda\in\Lambda_T^\text{fin}/V\mid q_\lambda\in I\}\sqcup\{\lambda\in\Lambda_T^\infty/V\mid q_{\lambda/V}\in I\}.
\end{eqnarray*}

Conversely, for any $(V,\Sigma,\Theta)\in AT(\Dot{E},\Lambda)$, let $I(V,\Sigma,\Theta)$ denote the submonoid of $H(\dot{E},\Lambda)$ generated by the set $V\sqcup Q(\Theta)\sqcup P(\Sigma)$, where
\begin{eqnarray*}
Q(\Theta) &=& \{q_\lambda\mid\lambda\in\Lambda_T^\text{fin}/V\cap \Theta\}\sqcup\{q_{\lambda/V}\mid\lambda\in\Lambda_T^\infty/V\cap \Theta\},\\
P(\Sigma) &=& \{p_\lambda\mid\lambda\in\Lambda_\text{fin}^S/V\cap \Sigma\}\sqcup\{p_{\lambda/V}\mid\lambda\in\Lambda_\infty^S/V\cap \Sigma\},
\end{eqnarray*}

and $\langle I(V,\Sigma,\Theta)\rangle$ be the order-ideal generated by $I(V,\Sigma,\Theta)$. Set $\phi(V,\Sigma,\Theta) = \langle I(V,\Sigma,\Theta)\rangle$.
\end{definition}

\begin{lemma}\label{lemma: I}
If $I$ is any order-ideal of $H(\dot{E},\Lambda)$, then $I=\phi\psi(I)$.
\end{lemma}
\begin{proof}
Let $\psi(I)=(V,\Sigma,\Theta)$ and $I(V,\Sigma,\Theta)=J$ so that $\phi\psi(I)=\langle J \rangle$. It is clear that $J\subseteq I$ and therefore $\langle J\rangle\subseteq I$. For converse, consider a nonzero element $x\in I$. Then $x=\sum_iv_i+\sum_jq_{\alpha_j}+\sum_kp_{\beta_k}+\sum_lq_{Z_l}+\sum_mp_{W_m}$ for some $v_i\in E^0$, $\alpha_j\in\Lambda_T^\text{fin}$, $\beta_k\in\Lambda_\text{fin}^S$, $Z_l\in\mathcal{Z}$, and $W_m\in\mathcal{W}$. Since $I$ is an order ideal, $v_i, q_{\alpha_j}, p_{\beta_k}, q_{Z_l}, p_{W_m}\in I$, and so to prove that $x\in\langle J\rangle$, it is enough to show that $v, q_\alpha, p_\beta, q_Z, p_W$ for all $v\in E^0, \alpha\in\Lambda_T^\text{fin}, \beta\in\Lambda_\text{fin}^S, Z\in\mathcal{Z}$ and $W\in\mathcal{W}$.

\textbf{Case 1}
If $v\in E^0\cap I$, then $v\in V$ by definition of $H$, hence $v\in J$.

\textbf{Case 2}
Let $\alpha\in\Lambda_T^\text{fin}$ such that $q_\alpha\in I$. 

\textbf{Subcase 2.1}
If $r(\alpha)\subseteq V$, then $\textbf{r}(\alpha)\in I$ and so 
$\textbf{s}(\alpha)=\textbf{r}(\alpha)+q_\alpha\in I$.
Hence $s(X)\in V$ for each $X\in\mathcal{X}_\alpha$, and so $\textbf{s}(\alpha)\in J$. Since $q_\alpha\leq\textbf{s}(\alpha)$, it follows that $q_\alpha\in\langle J\rangle$.

\textbf{Subcase 2.2}
If $r(\alpha)\not\subseteq V$, then by definition $\alpha\in \Theta\cap\Lambda_T^\text{fin}/V$. Hence $q_\alpha\in J$.

\textbf{Case 3}
Let $\lambda\in\Lambda_T^\infty$ and $Z\in\mathcal{Z}_\lambda$ such that $q_Z\in I$.

\textbf{Subcase 3.1}
$\mathcal{Y}_{\lambda/V}=\emptyset$. This is equivalent to $r(\lambda)\subseteq V$ and the argument follows similar to subcase 2.1.

\textbf{Subcase 3.2}
$0<|\mathcal{Y}_{\lambda/V}|<\infty$. In this case, we have 
$$\textbf{r}(\mathcal{Y}_{\lambda/V}-Z)=\textbf{r}([\mathcal{Y}_{\lambda/V}\cup Z]-Z)\leq q_{(\mathcal{Y}_{\lambda/V}\cup Z)}+\textbf{r}([\mathcal{Y}_{\lambda/V}\cup Z]-Z)=q_Z\in I.$$
It follows that $\textbf{r}(\mathcal{Y}_{\lambda/V}-Z)\in I$ and so $r(\mathcal{Y}_{\lambda/V}-Z)\subseteq H$. Hence $\mathcal{Y}_{\lambda/V}\subseteq Z$. Since $r(Z-\mathcal{Y}_{\lambda/V})\subseteq H$, we get
$q_{\lambda/V}=\textbf{r}(Z-\mathcal{Y}_{\lambda/V})+q_Z\in I,$ 
so that $\lambda\in \Theta$ by definition, and since $q_Z\leq q_{\lambda/V}$, we get $q_Z\in\langle J\rangle$.

\textbf{Subcase 3.3}
$|\mathcal{Y}_{\lambda/V}|=\infty$. Then there exists $Y\in\mathcal{Y}_{\lambda/V}-Z$, and we have 
$$r(Y)\leq \textbf{r}(Z\sqcup\{Y\}-Z)+q_{\{Y\}}=q_Z\in I.$$
But this implies that $r(Y)\in I$ and so $r(Y)\in V$, which contradicts $Y\in\mathcal{Y}_{\lambda/H}$. Thus $q_Z\in\langle J\rangle$.

The remaining cases are proved analogously.
\end{proof}

\begin{construction}
Let $(\Dot{E},\Lambda)$ be a B-hypergraph and $(V,\Sigma,\Theta)\in AT(\Dot{E},\Lambda)$. For $A\subseteq E^1$, define
$$A_r(V)=A\cap r^{-1}(V)~\text{and}~A_s(V)=A\cap s^{-1}(V).$$

We define the quotient B-hypergraph $(\dot{\widetilde{E}},\widetilde{\Lambda})$ as follows: $\Dot{\widetilde{E}}$ is given by
$$\widetilde{E}^0=E^0-V\quad\text{and}\quad\widetilde{E}^1=E_r^1(V)\cup E_s^1(V).$$
$r_{\widetilde{E}}$ and $s_{\widetilde{E}}$ are restriction maps of $r_E$ and $s_E$ respectively.

For $v\in\widetilde{E}^0$, set 
\begin{eqnarray*}
\widetilde{C}_v &=& \{X_r(V)\mid X\in C_v~\text{and}~X_r(V)\neq\emptyset\}~\text{and}~\widetilde{C}=\bigsqcup\limits_{v\in\widetilde{E}^0}\widetilde{C}_v,\\
\widetilde{D}_v &=& \{Y_s(V)\mid Y\in D_v~\text{and}~Y_s(V)\neq\emptyset\}~\text{and}~\widetilde{D}=\bigsqcup\limits_{v\in\widetilde{E}^0}\widetilde{D}_v,\\
\widetilde{S} &=& \{X_r(V)\mid X\in S~\text{and}~ X_r(V)\neq\emptyset\}\sqcup\{X_r(V)\mid X\in\mathcal{X}_\lambda, \lambda\in \Sigma\},~ \text{and}\\
\widetilde{T} &=& \{Y_s(V)\mid Y\in S~\text{and}~ Y_s(V)\neq\emptyset\}\sqcup\{Y_s(V)\mid Y\in\mathcal{Y}_\lambda, \lambda\in \Theta\}.
\end{eqnarray*}

Let $\widetilde{\Lambda}$ be defined as follows:
\begin{eqnarray*}
\widetilde{\Lambda}_{\widetilde{T}}^{\widetilde{S}} &=& \{\widetilde{\lambda}\mid\lambda\in\Lambda_T^S/V\sqcup \Sigma\sqcup \Theta\},\\
\widetilde{\Lambda}_\text{fin}^{\widetilde{S}} &=& \{\widetilde{\lambda}\mid\lambda\in\Lambda_\text{fin}^S/V-\Sigma\},\\
\widetilde{\Lambda}_\infty^{\widetilde{S}} &=& \{\widetilde{\lambda}\mid\lambda\in\Lambda_\infty^S/V-\Sigma\},\\
\widetilde{\Lambda}_{\widetilde{T}}^\text{fin} &=& \{\widetilde{\lambda}\mid\lambda\in\Lambda_T^\text{fin}/V-\Theta\},~\text{and}\\
\widetilde{\Lambda}_{\widetilde{T}}^\infty &=& \{\widetilde{\lambda}\mid\lambda\in\Lambda_T^\infty/V-\Theta\}.
\end{eqnarray*}

\end{construction}

We note that if $\pi:M_1\rightarrow M_2$ is a monoid homomorphism and $M_2$ is conical, then ker $\pi:=\pi^{-1}(0)$ is an order-ideal of $M_1$.

\begin{theorem}\label{theorem: ker of monoid hom}
 Let $(\dot{E},\Lambda)$ be a B-hypergraph,  $(V,\Sigma,\Theta)\in \emph{AT}(\Dot{E},\Lambda)$ and $(\Dot{\widetilde{E}},\widetilde{\Lambda})$ is the corresponding quotient B-hypergraph. Suppose that $I:=\langle I(V,\Sigma,\Theta)\rangle$ is the order ideal in $M:=H(\dot{E},\Lambda)$. Then there exists a monoid homomorphism $\pi:M\rightarrow \widetilde{M}:=M(\Dot{\widetilde{E}})$ such that $I=\emph{ker}~\pi$. 
\end{theorem}
\begin{proof}
 We begin by defining $\widetilde{v},\widetilde{q}_\alpha,\widetilde{p}_\beta,\widetilde{q}_Z,\widetilde{p}_W\in\widetilde{M}$ for $v\in E^0$, $\alpha\in\Lambda_T^\text{fin}$, $\beta\in\Lambda_\text{fin}^S$, $Z\in\mathcal{Z}$, and $W\in\mathcal{W}$. For $v\in E^0$, set
 $$\widetilde{v} = \begin{cases} v\quad\text{if}~v\notin V ~\text{and}\\
 0\quad\text{if}~v\in V.\\
 \end{cases}$$

For $\alpha\in\Lambda_T^\text{fin}$, we define $\widetilde{q}_\alpha$ as follows:
\begin{enumerate}
    \item If $s(\alpha)\subseteq V$ or $r(\alpha)\subseteq V$,  $\widetilde{q}_\alpha=0$
    \item If $s(\alpha)\not\subseteq V$ and $r(\alpha)\not\subseteq V$, $\widetilde{q}_\alpha = \begin{cases} 0\quad\text{if}~\alpha\in \Theta~\text{and}\\
    q_{\widetilde{\alpha}}\quad\text{if}~\alpha\notin \Theta.\\
    \end{cases}$
\end{enumerate}

For $\beta\in\Lambda_\text{fin}^S$, we define $\widetilde{p}_\beta$ as follows:
\begin{enumerate}
    \item If $s(\beta)\subseteq V$ or $r(\beta)\subseteq V$,  $\widetilde{p}_\beta=0$.
    \item If $s(\beta)\not\subseteq V$ and $r(\beta)\not\subseteq V$,  $\widetilde{p}_\beta = \begin{cases} 0\quad\text{if}~\beta\in \Sigma~\text{and}\\
    p_{\widetilde{\beta}}\quad\text{if}~\beta\notin \Sigma.\\
    \end{cases}$
\end{enumerate}

For $\lambda\in\Lambda_T^\infty$, and $Z\in\mathcal{Z}_\lambda$, we define $\widetilde{q}_Z$ as follows:
\begin{enumerate}
    \item If $s(\lambda)\subseteq V$, $\widetilde{q}_Z = 0$.
    \item If $s(\lambda)\not\subseteq V$, set $\widetilde{Z}=\{Y_s(V)\in\mathcal{Y}_{\widetilde{\lambda}}\mid Y\in Z\}$.
    \begin{enumerate}
    \item If $\lambda\in \Theta$  $\widetilde{q}_Z = \textbf{r}(\mathcal{Y}_{\widetilde{\lambda}}-\widetilde{Z})$,
    \item If $\lambda\notin \Theta$ and $r(Z)\subseteq H$,  $\widetilde{q}_Z =
    \textbf{s}(\lambda)$
    \item If $\lambda\notin \Theta$, $r(Z)\not\subseteq V$ and $\lambda\notin\Lambda_T^\infty/V$, $\widetilde{q}_Z =  q_{\widetilde{Z}}$ and
    \item If $\lambda\notin \Theta$, $r(Z)\not\subseteq V$, and $\lambda\in\Lambda_T^\infty/V$. $\widetilde{q}_Z = 
    q_{\widetilde{\lambda}}+\textbf{r}(\mathcal{Y}_{\widetilde{\lambda}}-\widetilde{Z})$.
\end{enumerate}
\end{enumerate}

For $\lambda\in\Lambda_\infty^S$, and $W\in\mathcal{W}_\lambda$, we define $\widetilde{p}_W$ as follows:
\begin{enumerate}
    \item If $r(\lambda)\subseteq V$, $\widetilde{p}_W = 0$.
   \item If $r(\lambda)\not\subseteq V$ set $\widetilde{W}=\{X_r(V)\in\mathcal{X}_{\widetilde{\lambda}}\mid X\in W\}$.
  \begin{enumerate}
    \item If $\lambda\in \Sigma$,
    $\widetilde{p}_W = \textbf{s}(\mathcal{X}_{\widetilde{\lambda}}-\widetilde{W})$,
    \item If $r(\lambda)\not\subseteq V$, $\lambda\notin \Sigma$ and $s(W)\subseteq V$, $\widetilde{p}_W = 
    \textbf{r}(\lambda)$,
    \item If $r(\lambda)\not\subseteq V$, $\lambda\notin \Sigma$, $s(W)\not\subseteq V$ and $\lambda\notin\Lambda_\infty^S/V$, $\widetilde{p}_W = p_{\widetilde{W}}$ and
    \item If $r(\lambda)\not\subseteq V$, $\lambda\notin \Sigma$, $s(W)\not\subseteq V$, and $\lambda\in\Lambda_\infty^S/V$,  $\widetilde{p}_W =
    p_{\widetilde\lambda}+\textbf{s}(\mathcal{X}_{\widetilde{\lambda}}-\widetilde{W})$.
\end{enumerate}
\end{enumerate}
We define $\pi:M\rightarrow\widetilde{M}$ by mapping generators $v\mapsto\widetilde{v}$, for all $v\in E^0$, $q_\alpha\mapsto\widetilde{q}_\alpha$, for all $\alpha\in\Lambda_T^\text{fin}$, $p_\beta\mapsto\widetilde{p}_\beta$ for all $\beta\in\Lambda_\text{fin}^S$, $q_Z\mapsto\widetilde{q}_Z$ for all $Z\in\mathcal{Z}$, and $p_W\mapsto\widetilde{p}_W$ for all $W\in\mathcal{W}$. To show that $\pi$ defines a homomorphism we need to verify that images of the generators satisfy all the defining relations. Here we only show for $\lambda\in\Lambda_T$ and the argument follows analogously for $\lambda\in\Lambda_S$. 

Let $\lambda\in\Lambda_T$. We introduce a new notation
\begin{eqnarray*}
\widetilde{\textbf{s}}(W) &:=& \sum\limits_{X\in W}\widetilde{s(X)}~\text{for subsets}~W\subseteq \mathcal{X}_\lambda~\text{and}\\
\widetilde{\textbf{r}}(Z) &:=& \sum\limits_{Y\in Z}\widetilde{r(Y)}~\text{for subsets}~Z\subseteq \mathcal{Y}_\lambda.
\end{eqnarray*}

Suppose that $\lambda\in\Lambda_T^S$. If $s(\lambda)\subseteq V$, then $r(\lambda)\subseteq V$, and we get
$$\widetilde{\textbf{s}}(\lambda)=0=\widetilde{\textbf{r}}(\lambda).$$ 
If $s(\lambda)\not\subseteq V$ then $r(\lambda)\not\subseteq V$, and we get
$$\widetilde{\textbf{s}}(\lambda)=\textbf{s}(\mathcal{X}_{\lambda/V})=\textbf{s}(\widetilde{\lambda})=\textbf{r}(\widetilde{\lambda})=\textbf{r}(\mathcal{Y}_{\lambda/V})=\widetilde{\textbf{r}}(\lambda).$$

Suppose that $\lambda\in\Lambda_T^\text{fin}$. If $s(\lambda)\subseteq V$, then $\mathcal{Y}_{\lambda/V}=\emptyset$. If $s(\lambda)\not\subseteq V$, and $r(\lambda)\subseteq V$, then $\mathcal{X}_{\lambda/V}=\emptyset$. In both of the above cases we have
$$\widetilde{\textbf{s}}(\lambda)=0=\widetilde{\textbf{r}}(\lambda)+\widetilde{q}_\lambda.$$
So let $s(\lambda)\not\subseteq V$ and $r(\lambda)\not\subseteq V$. If $\lambda\in \Theta$ then 
$$\widetilde{\textbf{s}}(\lambda)=\textbf{s}(\mathcal{X}_{\lambda/V})=\textbf{s}(\widetilde{\lambda})=\textbf{r}(\widetilde{\lambda})=\textbf{r}({\mathcal{Y}_{\lambda/V})+0=\widetilde{\textbf{r}}(\lambda)}+\widetilde{q}_\lambda.$$
If $\lambda\notin \Theta$ then
$$\widetilde{\textbf{s}}(\lambda)=\textbf{s}(\mathcal{X}_{\lambda/V})=\textbf{s}(\widetilde{\lambda})=\textbf{r}(\widetilde{\lambda})=\textbf{r}({\mathcal{Y}_{\lambda/V})+q_{\mathcal{Y}_{\lambda/V}}=\widetilde{\textbf{r}}(\lambda)}+\widetilde{q}_\lambda.$$

Now suppose that $\lambda\in\Lambda_T^\infty$, and $Z\in\mathcal{Z}_\lambda$. If $s(\lambda)\subseteq V$, then $\mathcal{Y}_{\lambda/V}=\emptyset$, and hence $r(Z)=\emptyset$. Then we have 
$$\widetilde{\textbf{s}}(\lambda)=0=\widetilde{\textbf{r}}(Z)+\widetilde{q}_Z.$$
Hence we assume that $s(\lambda)\not\subseteq V$ for rest of the step. If $\lambda\in \Theta$ then $\widetilde{\lambda}\in\widetilde{\Lambda}_{\widetilde{T}}^{\widetilde{S}}$ and we have
$$\widetilde{\textbf{s}}(\lambda)=\textbf{s}({\widetilde{\lambda}})=\textbf{r}(\widetilde{\lambda})=\textbf{r}(\mathcal{Y}_{\widetilde\lambda})=\textbf{r}(\widetilde{Z})+\textbf{r}(\mathcal{Y}_{\widetilde{\lambda}}-\widetilde{Z})=\textbf{r}(\widetilde{Z})+\widetilde{q}_Z=\widetilde{\textbf{r}}(Z)+\widetilde{q}_Z.$$
If $\lambda\notin \Theta$ and $r(Z)\subseteq V$ then we have
$$\widetilde{\textbf{s}}(\lambda)=\textbf{s}(\widetilde{\lambda})=\widetilde{q}_Z=\widetilde{\textbf{r}}(Z)+\widetilde{q}_Z.$$
If $\lambda\notin \Theta$, $r(Z)\not\subseteq V$ and $\lambda\notin\Lambda_T^\infty/V$ then we have
$$\widetilde{\textbf{s}}(\lambda)=\textbf{s}(\widetilde{\lambda})=\textbf{r}(\widetilde{Z})+q_{\widetilde{Z}}=\widetilde{\textbf{r}}(Z)+\widetilde{q}_Z.$$
If $\lambda\notin \Theta$, $r(Z)\not\subseteq V$ and $\lambda\in\Lambda_T^\infty/V$ then we have
$$\widetilde{\textbf{s}}(\lambda)=\textbf{s}(\widetilde{\lambda})=\textbf{r}(\widetilde{\lambda})+q_{\widetilde{\lambda}}=\textbf{r}(\mathcal{Y}_{\widetilde{\lambda}}-\widetilde{Z})+\textbf{r}(\widetilde{Z})+q_{\widetilde{\lambda}}=\widetilde{\textbf{r}}(Z)+\widetilde{q}_Z.$$

Now assume that for $\lambda\in\Lambda_T^\infty$ let $Z_1,Z_2\in\mathcal{Z}_\lambda$ and $Z_1\subsetneq Z_2$. If $s(\lambda)\subseteq V$ then we have
$$\widetilde{q}_{Z_1}=0=\widetilde{\textbf{r}}(Z_2-Z_1)+\widetilde{q}_{Z_2}.$$
So we may assume that $s(\lambda)\not\subseteq V$.
If $\lambda\in \Theta$, 
$$\widetilde{q}_{Z_1}=\textbf{r}(\mathcal{Y}_{\widetilde{\lambda}})-\widetilde{Z}_1=\textbf{r}(\widetilde{Z}_2-\widetilde{Z}_1)+\textbf{r}(\mathcal{Y}_{\widetilde{\lambda}})=\widetilde{\textbf{r}}(Z_2-Z_1)+\widetilde{q}_{Z_2}.$$
Only remaining case is when $\lambda\notin \Theta$. If $\lambda\in\Lambda_T^\infty/V-\Theta$, then
$$\widetilde{q}_{Z_1}=q_{\widetilde{\lambda}}+\textbf{r}(\mathcal{Y}_{\widetilde{\lambda}}-Z_1)=q_{\widetilde{\lambda}}+\textbf{r}(\mathcal{Y}_{\widetilde{\lambda}}-Z_2)+\textbf{r}(\widetilde{Z_2}-\widetilde{Z_1})=\widetilde{q}_{Z_2}+\widetilde{\textbf{r}}(Z_2-Z_1).$$
Hence, we may assume that $\lambda\notin\Lambda_T^\infty/V$.
If $r(Z_2)\subseteq V$ then we have
$$\widetilde{q}_{Z_1}=\textbf{s}(\lambda)=\textbf{r}(\widetilde{Z}_2-\widetilde{Z}_1)+\textbf{s}(\lambda)=\widetilde{\textbf{r}}(Z_2-Z_1)+\widetilde{q}_{Z_2}.$$
If $r(Z_1)\subseteq V$ but $r(Z_2)\not\subseteq V$, we have
$$ \widetilde{q}_{Z_1}=\textbf{s}(\lambda)=\textbf{r}(\widetilde{Z_2})+q_{\widetilde{Z}}=\widetilde{\textbf{r}}(Z_2-Z_1)+\widetilde{q}_{Z_2}.$$
Finally, if $r(Z_1)\not\subseteq V$, then we have
$$\widetilde{q}_{Z_1}=q_{\widetilde{Z}_1}=\textbf{r}(\widetilde{Z}_2-\widetilde{Z}_1)+q_{\widetilde{Z}_2}=\widetilde{\textbf{r}}(Z_2-Z_1)+\widetilde{q}_{Z_2}.$$
Thus we have shown that $\pi$ is a monoid homomorphism.

Now we show that $I\subseteq \text{ker}~\pi$. Since ker$\pi$ is order-ideal, it suffices to show that $I(V,\Sigma,\Theta)\subseteq \text{ker}\pi$. For $v\in H$ we have $\pi(v)=\widetilde{v}=0$. For $\lambda\in \Theta\cap \Lambda_T^\text{fin}$, we have $\pi(q_{\lambda})=\widetilde{q}_{\lambda}=0$. If $\lambda\in \Theta\cap\Lambda_T^\infty$, then $\pi(q_{\lambda/V})=\widetilde{q}_{\lambda/V}=0$. Similarly we can verify that if $\lambda\in \Sigma\cap\Lambda_\text{fin}^S$ then $\pi(q_\lambda)=0$ and if $\lambda\in \Sigma\cap\Lambda_\infty^S$ then $\pi(q_{\lambda/V})=0$.

We claim that $\psi(\text{ker}~\pi)=(V,\Sigma,\Theta)$. For, let $\psi(\text{ker}~\pi)=(\widetilde{V},\widetilde{\Sigma},\widetilde{\Theta})$. It follows from definition that $\widetilde{V}=I\cap E^0=V$ and by the previous paragraph $\Sigma\subseteq \widetilde{\Sigma}$ and $\Theta\subseteq\widetilde{\Theta}$. Consider $\lambda\in\Lambda_T^\text{fin}/V\sqcup\Lambda_T^\infty/V$. If $\mathcal{Y}_\lambda$ is finite and $\lambda\notin \Theta$, then $\pi(q_{\lambda})=\widetilde{q}_{\lambda}\neq0$. Hence $q_\lambda\notin\text{ker}~\pi$ and so $\lambda\notin\widetilde{\Theta}$. If $\mathcal{Y}_\lambda$ is infinite and $\lambda\notin \Theta$, then $\pi(q_{\mathcal{Y}_{\lambda/V}})=\widetilde{q}_{\mathcal{Y}_{\lambda/V}}\neq0$. Thus $\lambda\notin\widetilde{\Theta}$. Hence $\Theta=\widetilde{\Theta}$. Similarly $\Sigma=\widetilde{\Sigma}$.

Finally, since $\psi(\text{ker}~\pi)=(V,\Sigma,\Theta)$ and $I=\phi\circ\psi(I)$, we have that ker $\pi=\langle I(V,\Sigma,\Theta)\rangle=I$.
\end{proof}

\begin{corollary}\label{corollary: H}
If $(V,\Sigma,\Theta)\in \emph{AT}(\Dot{E},\Lambda)$, then $(V,\Sigma,\Theta)=\psi\circ\phi(V,\Sigma,\Theta)$.
\end{corollary}

\begin{theorem}\label{theorem: AT and L lattice isomorphism}
 Let $\Dot{E}$ be a B-hypergraph. Then there are mutually inverse lattice isomorphisms
 $$\phi:\emph{AT}(\Dot{E},\Lambda)\rightarrow\mathcal{L}(H(\dot{E},\Lambda))~\text{and}~\psi:\mathcal{L}(H(\dot{E},\Lambda))\rightarrow \emph{AT}(\Dot{E},\Lambda),$$
 where $\phi(V,\Sigma,\Theta)=\langle I(V,\Sigma,\Theta)\rangle$ for $(V,\Sigma,\Theta)\in \emph{AT}(\Dot{E},\Lambda)$ and $\psi$ is defined as in definition \ref{definition: correspondence between order-ideals and admissible triples}.
 \end{theorem}
 \begin{proof}
   The maps $\psi$ and $\phi$ are well defined by definition and by Lemma \ref{lemma: I} $\phi\circ\psi$ is the identity map on $\mathcal{L}(H(\dot{E},\Lambda))$ and by Corollary \ref{corollary: H} $\psi\circ\phi$ is the identity map on AT$(\Dot{E},\Lambda)$. We only to have to show that $\psi$ and $\phi$ are order-preserving.
   
   Suppose $I_1\subseteq I_2$ are order-ideals of $H(\dot{E},\Lambda)$ and $(V_j,\Sigma_j,\Theta_j)=\psi(I_j)$ for $j=1,2$. Clearly $V_1\subseteq V_2$. We only show that $\Theta_1\subseteq \Theta_2\sqcup \Lambda_T(V_2)$. Let $\lambda\in\Theta_1$. First suppose that $\lambda\in\Lambda_T^\text{fin}/V_1$ and $q_\lambda\in I_1$. If $\lambda\in\Lambda_T^\text{fin}/V_2$, then $\lambda\in\Theta_2$. Otherwise, $r(\lambda)\subseteq V_2$ and so $\textbf{s}(\lambda)\in I_2$, which implies $s(\lambda)\in V_2$ and $\lambda\in\Lambda_T(V_2)$. Now suppose that $\lambda\in\Lambda_Y^\infty/V_2$ and $q_{\lambda/V_1}\in I_1$. If $\lambda\in\Lambda_T^\infty/V_2$, then $q_{\lambda/V_2}$ is defined and also 
   $$q_{\lambda/V_2}=\textbf{r}(\{Y\in\mathcal{Y}_\lambda\mid r(Y)\in V_2-V_1\})+q_{\lambda/V_1}\in I_2,$$
   so $\lambda\in\Theta_2$. Otherwise, $r(\lambda)\subseteq V_2$ and so $\textbf{r}(\lambda/V_1)\in I_2$, hence $\textbf{s}(\lambda)\in I_2$, again giving $\lambda\in\Lambda_T(V_2)$. $\Sigma_1\subseteq \Sigma_2\sqcup\Lambda^S(V_2)$ follows on similar lines.
   
   Finally, let $(V_1,\Sigma_1,\Theta_1)$ and $(V_2,\Sigma_2,\Theta_2)$ are elements of AT$(\dot{E},\Lambda)$ such that $(V_1,\Sigma_1,\Theta_1)\leq(V_2,\Sigma_2,\Theta_2)$. Clearly $V_1\subseteq I(V_2,\Sigma_2,\Theta_2)$. Consider $\lambda\in\Theta_1\cap\Lambda_T^\text{fin}$. If $\lambda\in\Theta_2$, then $q_\lambda\in I(V_2,\Sigma_2,\Theta_2)$ by definition of $I(V_2,\Sigma_2,\Theta_2)$. If $\lambda\in\Lambda_T(V)$, then 
   $$q_\lambda\leq q_\lambda+\textbf{r}(\lambda)=\textbf{s}(\lambda)\in I(V_2,\Sigma_2, \Theta_2)$$
   and so $q_\lambda\in\langle I(V_2,\Sigma_2,\Theta_2)\rangle$. Now consider $\lambda\in \Theta_1\cap\Lambda_T^\infty$. If $\lambda\in\Theta_2$, then $q_{\lambda/V_2}\in I(V_2,\Sigma_2,\Theta_2)$ and since $$q_{\lambda/V_1}\leq q_{\lambda/V_1}+\textbf{r}(\mathcal{Y}_{\lambda/V_2}-\mathcal{Y}_{\lambda/V_1})=q_{\lambda/V_2},$$
   it follows that $q_{\lambda/V_1}\in\langle I(V_2,\Sigma_2,\Theta_2)\rangle$. If $\lambda\in\Lambda(V)$, then 
   $$q_{\lambda/V_1}\leq q_{\lambda/V_1}+\textbf{r}(\mathcal{Y}_{\lambda/V_1})=\textbf{s}(\lambda)\in I(V_2,\Sigma_2,\Theta_2)$$
   and again $q_{\lambda/V_1}\in\langle I(V_2,\Sigma_2,\Theta_2)\rangle$. A similar arguments shows that $\Sigma_1\subseteq\Sigma_2\sqcup\Lambda^S(V)$. Therefore all the generators of $I(V_1,\Sigma_1,\Theta_1$ lie in $\phi(V_2,\Sigma_2,\Theta_2)$, and we conclude that $\phi(V_1,\Sigma_1,\Theta_1)\subseteq\phi(V_2,\Sigma_2,\Theta_2)$. Hence $\phi$ is order-preserving. 
   
 \end{proof}
 
 \subsection{The lattice of trace-ideals in $\mathcal{A}_K(\dot{E},\Lambda)$}
 
 \begin{definition}
 Let $R$ be an arbitrary ring and Idem($M_\infty(R))$ denote the set of idempotents in $M_\infty(R))$. An ideal $I$ of $R$ is called a \textit{trace-ideal} provided $I$ can be generated by the entries of the matrices in some subset of Idem($M_\infty(R))$. We denote by Tr$(R)$ the set of all trace ideals of $R$. Since Tr$(R)$ is closed under arbitrary sums and arbitrary intersections, it forms a complete lattice with respect to inclusion.
 \end{definition}
 \begin{proposition}\cite[Proposition 10.10]{MR2980456} \label{prop: LV and Tr lattice isomorphism}
  For any ring $R$ there are mutually inverse lattice isomorphisms
  $$\Phi:\mathcal{L(V}(R))\rightarrow \emph{Tr}(R)\quad\text{and}\quad\Psi:\emph{Tr}(R)\rightarrow\mathcal{L(V}(R))$$
  given by
  $$\Phi(I)=\langle\text{entries of}~e\mid e\in\emph{ Idem}(M_\infty(R))~\text{and}~[e]\in I\rangle$$
  $$\Psi(J)=\{[e]\in\mathcal{V}(R)\mid e\in \emph{Idem}(M_\infty(J))\}.$$
 \end{proposition}

 \begin{lemma}\label{lemma: idempotents}
Let $(\Dot{E},\Lambda)$ be a B-hypergraph. Then the trace ideals of $A:=\mathcal{A}_K(\Dot{E},\Lambda)$ are precisely the idempotent generated ideals and the lattice isomorphism 
$\Phi:\mathcal{L(V}(A))\rightarrow \emph{Tr}(A)$ is expressed as 
$$\Phi(I)=\langle\text{idempotents}~e\in A\mid[e]\in I\rangle.$$
\end{lemma}
\begin{proof}
  The proof goes exactly similar to \cite[Proposition 6.2]{MR2980456}, except that in the present case, the V-monoid $\mathcal{V}(\mathcal{A}_K(\dot{E}))$ is generated by 
  \begin{eqnarray}
  \{[v]~|~v \in E^0\} & \cup & \{[q_Z^\prime]~|~Z \subseteq \mathcal{Y}_\lambda,~ 0 < |Z| < \infty~\text{and}~ \lambda \notin \Lambda_S\} \nonumber \\
  & \cup & \{[p_W^\prime]~|~W \subseteq \mathcal{X}_\lambda,~0 < |W| < \infty~\text{and}~ \lambda \notin \Lambda_T\} \nonumber.
  \end{eqnarray}
  (Here, because of Theorem \ref{thm: V-monoid_H-monoid_Isomorphism}, we are using the notation $[q_Z^\prime]$ and $[p_W^\prime]$ for the generators of $\mathcal{V}(\mathcal{A}_K(\dot{E}))$. Strictly speaking, $[q_z^\prime]$ and $[p_W^\prime]$ stand for images of $q_Z$ and $p_W$ respectively under the map $\Gamma$ defined in Theorem \ref{thm: V-monoid_H-monoid_Isomorphism}).
\end{proof}
\begin{theorem}
  Let $(\Dot{E}.\Lambda)$ be a B-hypergraph and $A=\mathcal{A}_K(\Dot{E},\Lambda)$. Then there exist mutually inverse lattice isomorphisms
 $$\xi:\emph{AT}(\Dot{E},\Lambda)\rightarrow \emph{Tr}(A)~\text{and}~\zeta:\emph{Tr}(A)\rightarrow \emph{AT}(\Dot{E},\Lambda).$$
 \end{theorem}
 \begin{proof}
   Set $M:=H(\dot{E},\Lambda)$ and let $\Gamma:M\rightarrow\mathcal{V}(A)$ be the monoid isomorphism. By abuse of notation we also use $\Gamma$ to denote the induced lattice isomorphism $\mathcal{L}(M)\rightarrow\mathcal{L(V}(A))$. Due to Theorem \ref{theorem: AT and L lattice isomorphism} and Proposition \ref{prop: LV and Tr lattice isomorphism}, we have mutually inverse lattice isomorphisms
   $$\Phi\Gamma\phi:\text{AT}(\Dot{E},\Lambda)\rightarrow\text{Tr}(A)\quad\text{and}\quad\psi\Gamma^{-1}\Psi:\text{Tr}(A)\rightarrow\text{AT}(\Dot{E},\Lambda).$$
   
   More explicitly, if $J\in\emph{Tr}(A)$, then $\zeta(J)=(H,\Sigma,\Theta)$, where
   $$H=E^0\cap J,$$
   $$\Sigma=\{\lambda\in\Lambda_\text{fin}^S/H\mid p_\lambda\in J\}\sqcup\{\lambda\in\Lambda_\infty^S/H\mid p_{\lambda/H}\in J\},$$
   $$\Theta=\{\lambda\in\Lambda_T^\text{fin}/H\mid q_\lambda\in J\}\sqcup\{\lambda\in\Lambda_T^\infty/H\mid q_{\lambda/H}\in J\}.$$
   
   For converse, let $(H,\Sigma,\Theta)\in\emph{AT}(\Dot{E},\Lambda)$. First define $\xi(H,\Sigma,\Theta)=\langle H\sqcup P(\Sigma)\sqcup Q(\Theta)\rangle$, where $P(\Sigma)$ and $Q(\Theta)$ are defined as in Definition \ref{definition: correspondence between order-ideals and admissible triples}. 
   Then define $J(H,\Sigma,\Theta)$ to be the order-ideal of $\mathcal{V}(A)$ generated by the set $H^\prime\sqcup P^\prime(\Sigma)\sqcup Q^\prime(\Sigma)$, where
   $$H^\prime=\{[v]\mid v\in H\},$$
   $$P^\prime(\Sigma)=\{[p_\lambda^\prime]\mid \lambda\in\Sigma\cap\Lambda_\text{fin}^S/H\}\sqcup\{[p_{\lambda/H}^\prime]\mid\lambda\in\Sigma\cap\Lambda_\infty^S\},$$
   $$Q^\prime(\Theta)=\{[q_\lambda^\prime]\mid \lambda\in\Theta\cap\Lambda_T^\text{fin}/H\}\sqcup\{[q_{\lambda/H}^\prime]\mid^\lambda\in\Theta\cap\Lambda_T^\infty\}.$$
   
  By Lemma \ref{lemma: idempotents}, it follows that 
  $$\Phi\Gamma\phi(H,\Sigma,\Theta)=\langle\text{idempotents}~e\in A\mid [e]\in J(H,\Sigma,\Theta)\rangle.$$
  It is clear that $\xi(H,\Sigma,\Theta)\subseteq \Phi\Gamma\phi(H,\Sigma,\Theta)$.
  
  If $e$ is an idempotent in $A$ such that $[e]\in J(H,\Sigma,\Theta)$, then 
  $$[e]\leq\sum\limits_{i=1}^{n_1}[v_i]+\sum\limits_{j=1}^{n_2}[p_{\alpha_j}^\prime]+\sum\limits_{k=1}^{n_3}[p_{\beta_k/H}^\prime]+\sum\limits_{l=1}^{n_4}[q_{\gamma_l}^\prime]+\sum\limits_{m=1}^{n_5}[q_{\delta_m/H}^\prime]$$
  where $v_i\in H$, $\alpha_j\in\Sigma\cap\Lambda_\text{fin}^S$, $\beta_k\in\Sigma\cap\Lambda_\infty^S$, $\gamma_l\in\Theta\cap\Lambda_T^\text{fin}$ and $\delta_m\in\Theta\cap\Lambda_T^\infty$. Therefore $e$ is equivalent to some idempotent
  $e^\prime\leq \mathcal{D}$ where $\mathcal{D}$ is a diagonal matrix with entries $v_i,p_{\alpha_j},p_{\beta_k/H},q_{\gamma_l}$, and $q_{\delta_m/H}$.Thus it follows that $e$ lies in these $v_i,p_{\alpha_j},p_{\beta_k/H},q_{\gamma_l}$, and $q_{\delta_m/H}$. Hence $\Phi\Gamma\phi=\xi$.
 \end{proof}
\subsection{Simplicity}

A non-zero conical monoid $M$ is \textit{simple} if its only order-ideals are $\{0\}$ and $M$.

\begin{theorem}\label{theorem: trace ideal simplicity}
 Let $(\Dot{E},\Lambda)$ be a B-hypergraph. Then the following conditions are equivalent
 \begin{enumerate}
     \item The only trace ideals of $\mathcal{A}_K(\Dot{E},\Lambda)$ are $0$ and $\mathcal{A}_K(\Dot{E},\Lambda)$.
     \item $H(\dot{E},\Lambda)
     $ is a simple monoid.
     \item $S=C_\emph{fin}$, $T=D_\emph{fin}$ and the only bisaturated subsets of $E^0$ are $\emptyset$ and $E^0$.
 \end{enumerate}
\end{theorem}
\begin{proof}
From Proposition \ref{prop: LV and Tr lattice isomorphism} it follows that $(1)\Leftrightarrow(2)$.

$(2)\Rightarrow(3):$ Observe that $(\Lambda_\text{fin}^S/\emptyset)\sqcup(\Lambda_\infty^S/\emptyset)=C_\text{fin}-S$ and $(\Lambda_T^\text{fin}/\emptyset)\sqcup(\Lambda_T^\infty/\emptyset)=D_\text{fin}-T$. Similarly, $(\Lambda_\text{fin}^S/E^0)\sqcup(\Lambda_\infty^S/E^0)=\emptyset$ and $(\Lambda_T^\text{fin}/E^0)\sqcup(\Lambda_T^\infty/E^0)=\emptyset$. By Theorem \ref{theorem: AT and L lattice isomorphism}, the only members of AT$(\Dot{E},\Lambda)$ are $(\emptyset,\emptyset,\emptyset)$ and $(E^0,\emptyset,\emptyset)$.  If $\lambda\in\Lambda_\text{fin}^S$, then $(\emptyset,\{\lambda\},\emptyset)\in\text{AT}(\Dot{E},\Lambda)$. This proves that $S=C_\text{fin}$. Similarly $T=D_\text{fin}$. If $H$ is any bisaturated subset of $E^0$, then $(H,\emptyset,\emptyset)\in\text{AT}(\Dot{E},\Lambda)$ and hence the only bisaturated subsets of $E^0$ are $E^0$ and $\emptyset$.

$(3)\Rightarrow(2):$ In this case AT$(\Dot{E},\Lambda)=\{(E^0,\emptyset,\emptyset),(\emptyset,\emptyset,\emptyset)\}$. The result follows at once from Theorem \ref{theorem: AT and L lattice isomorphism}.
\end{proof}

Although the characterization of simplicity of $\mathcal{A}_K(\Dot{E},\Lambda)$ is quite difficult, from theorem \ref{theorem: trace ideal simplicity} we can easily conclude the following necessary condition.
\begin{corollary}
Let $(\dot{E},\Lambda)$ be a hypergraph and suppose $L_K(\dot{E},\Lambda)$ is simple then 
the only trace ideals of $L_K(\Dot{E},\Lambda)$ are $0$ and $L_K(\Dot{E},\Lambda)$.
\end{corollary}

\section{Representations of Leavitt path algebras of regular hypergraphs}\label{sec: reps of LPA of hypergraphs}

Given a graph $E$, the \textit{category of quiver representations} of $E$ is the category of functors from the path category $\mathcal{C}_E$ to the category of $K$-vector spaces. A morphism of quiver representations is a natural transformation between two such functors. In other words, a quiver representation $\rho$ assigns a (possibly infinite dimensional) $K$-vector space $\rho(v)$ to each $v\in E^0$ and a linear transformation $\rho(e):\rho(s(e))\rightarrow\rho(r(e))$ to each $e\in E^1$. And a morphism of quiver representations $\phi:\rho\rightarrow\rho^\prime$ is a family of linear transformations $\{\phi_v\mid\rho(v)\rightarrow\rho^\prime(v)\}_{v\in E^0}$ such that for each $e\in E^1$ the following diagram commutes:
\[ \begin{tikzcd}
\rho(s(e)) \arrow{r}{\rho(e)} \arrow[swap]{d}{\phi_{s(e)}} & \rho(r(e)) \arrow{d}{\phi_{r(e)}} \\%
\rho^\prime(s(e)) \arrow{r}{\rho^\prime(e)}& \rho^\prime(r(e)).
\end{tikzcd}
\]

This section generalizes the results of \cite{MR4009579}. Throughout this section, a hypergraph always means a regular hypergraph. We will work in the category $\mathfrak{M}_L$ of unital (right) modules over $L:=L_K(\dot{E},\Lambda)$ where $(\dot{E},\Lambda)$ is a hypergraph.  The category $\mathfrak{M}_L$ is closed under taking quotients, submodules, extensions and arbitrary sums and hence it is an abelian category with sums. Note however that it is not closed under infinite product if $E^0$ is infinite.
\begin{lemma}\label{lemma: unital module satisfies an isomorphism}
Let $M$ be a right $L$-module. Then  $\bigoplus\limits_{X\in\mathcal{X}_\lambda}Ms(X)$ is isomorphic (as vector space) to $\bigoplus\limits_{Y\in\mathcal{Y}_\lambda}Mr(Y)$  for every $\lambda\in\Lambda$.
\end{lemma}
\begin{proof}
 For each $\lambda\in\Lambda$, let $[\lambda]$ be the rectangular matrix of size $|\mathcal{Y}_\lambda|\times|\mathcal{X}_\lambda|$ whose entry in $Y^{th}$ row and $X^{th}$ column is the edge $YX$. Then $[\lambda]:\bigoplus\limits_{X\in\mathcal{X}_\lambda}Ms(X)\rightarrow\bigoplus\limits_{Y\in\mathcal{Y}_\lambda}Mr(Y)$given by $$(m_X)_{X\in\mathcal{X}_\lambda}[\lambda]=\left(\sum\limits_{X^\prime \in\mathcal{X}_\lambda}m_X\mu_{(X^\prime Y)}\right)_{Y\in\mathcal{Y}_\lambda},$$ is a well defined linear map, where $\mu_{(XY)}$ is right multiplication by the edge $XY$. We show that $[\lambda]$ is an isomorphism with the inverse $[\lambda]^\ast$, which is the adjoint transpose matrix of $[\lambda]$.
Note that $[\lambda]^\ast:\bigoplus\limits_{Y\in\mathcal{Y}_\lambda}Mr(Y)\rightarrow\bigoplus\limits_{X\in\mathcal{X}_\lambda}Ms(X)$ is given by 
$$(m_Y)_{Y\in\mathcal{Y}_\lambda}[\lambda]^\ast=\left(\sum\limits_{Y^\prime\in\mathcal{Y}_\lambda}m_Y\mu_{(Y^\prime X)^\ast}\right)_{X\in\mathcal{X}_\lambda}.$$
We check their compositions:
\begin{eqnarray*}
   (m_Y)_{Y\in\mathcal{Y}_\lambda}[\lambda]^\ast[\lambda] & = & \left(\sum\limits_{Y^\prime\in\mathcal{Y}_\lambda}m_Y\mu_{(Y^\prime X)^\ast}\right)_{X\in\mathcal{X}_\lambda}[\lambda]\nonumber \\
          & = & \left(\sum\limits_{X\in\mathcal{X}_\lambda}\left(\sum\limits_{Y^\prime\in\mathcal{Y}_\lambda}m_Y\mu_{(Y^\prime X)^\ast}\right)\mu_{(XY)}\right)_{Y\in\mathcal{Y}_\lambda}\nonumber \\
          & = & \left(\sum\limits_{X\in\mathcal{X}_\lambda}\left(\sum\limits_{Y^\prime\in\mathcal{Y}_\lambda}m_Y\mu_{(Y^\prime X)^\ast}\mu_{(XY)}\right)\right)_{Y\in\mathcal{Y}_\lambda}\nonumber \\
          & = & \left(\sum\limits_{X\in\mathcal{X}_\lambda}\sum\limits_{Y^\prime\in\mathcal{Y}_\lambda}m_Y\mu_{(Y^\prime X)^\ast (XY)}\right)_{Y\in\mathcal{Y}_\lambda}\nonumber \\
          & = & \left(\sum\limits_{Y^\prime\in\mathcal{Y}_\lambda}\sum\limits_{X\in\mathcal{X}_\lambda}m_Y\mu_{(Y^\prime X)^\ast (XY)}\right)_{Y\in\mathcal{Y}_\lambda}\nonumber \\
          & = & \left(\sum\limits_{Y^\prime\in\mathcal{Y}_\lambda}m_Y\mu_{\sum\limits_{X\in\mathcal{X}_\lambda}(Y^\prime X)^\ast (XY)}\right)_{Y\in\mathcal{Y}_\lambda}\nonumber \\
          & = & \left(\sum\limits_{Y^\prime\in\mathcal{Y}_\lambda}m_Y\mu_{\delta_{Y^\prime Y}r(Y^\prime)}\right)_{Y\in\mathcal{Y}_\lambda}\quad\text{(by}~L2)\nonumber \\
          & = & (m_Y)_{Y\in\mathcal{Y}_\lambda}.
  \end{eqnarray*}
Similarly by $L1$, we get 
$$(m_X)_{X\in\mathcal{X}_\lambda}[\lambda][\lambda]^\ast=(m_X)_{X\in\mathcal{X}_\lambda}$$
which establishes the result.
\end{proof}
\begin{remark}\label{remark: M is sum of Mv}
If $M$ is unital, then for any $m\in M$, we have $m=\sum_{k=1}^lm_kv_k$ for some vertices $v_k\in E^0$. Hence $M=\sum\limits_{v\in E^0}Mv$. When considered as paths, the vertices of $E$ form a set of orthogonal idempotents, hence the above sum is direct. Therefore we have $$M=\bigoplus_{v\in E^0}Mv.$$
\end{remark}
\begin{theorem}\label{theorem: condition H}
 The category $\mathfrak{M}_L$ is equivalent to the full subcategory of quiver representations $\rho$ of $E$ satisfying:
 \begin{equation}\tag{H}
    \text{For all}~\lambda\in\Lambda,~[\rho(\lambda)]:\bigoplus\limits_{X\in\mathcal{X}_\lambda}\rho(s(X))\rightarrow\bigoplus\limits_{Y\in\mathcal{Y}_\lambda}\rho(r(Y))~\text{is an isomorphism.} 
 \end{equation}
\end{theorem}
\begin{proof}
Let $M$ be a right $L$-module. We define a quiver representation $\rho_M$ as follows: $\rho_M(v)=Mv$ for each $v\in E^0$ and $\rho_M(e):Ms(e)\rightarrow Mr(e)$ is given by $ms(e)\rho_M(e):=ms(e)e=me=mr(e)$. By Lemma \ref{lemma: unital module satisfies an isomorphism}, (H) is satisfied. 
If $\varphi:M\rightarrow N$ is an $L$-module homomorphism then $\varphi_v$ is  the linear transformation making the following diagram commtative.
\[ \begin{tikzcd}
\rho_M(v)=Mv \arrow[hookrightarrow]{r}{} \arrow[swap]{d}{\varphi_v} & M \arrow{d}{\varphi} \\%
\rho_N(v)=Nv \arrow[hookrightarrow]{r}{}& N.
\end{tikzcd}
\]
Since right multiplication by an edge $e$ commutes with $\varphi$, this defines a homomorphism of quiver representations.

Given a quiver representation $\rho$, we define the correspoding module $M_\rho:=\bigoplus\limits_{v\in E^0}\rho(v)$. To get an $L$-module structure on $M_\rho$, we define the following projections and inclusions: For each $v\in E^0$, define $$p_v:M_\rho\twoheadrightarrow\rho(v)\quad\text{;}\quad i_v:\rho(v)\hookrightarrow M_\rho$$ and for each $\lambda\in\Lambda$, $X\in\mathcal{X}_\lambda$ and $Y\in\mathcal{Y}_\lambda$, define $$p_X:\bigoplus\limits_{X\in\mathcal{X}_\lambda}\rho(s(X))\twoheadrightarrow\rho(s(X))\quad\text{;}\quad i_X:\rho(s(X))\hookrightarrow\bigoplus\limits_{X\in\mathcal{X}_\lambda}\rho(s(X))$$
$$p_Y:\bigoplus\limits_{Y\in\mathcal{Y}_\lambda}\rho(r(Y))\twoheadrightarrow\rho(r(Y))\quad\text{;}\quad i_Y:\rho(r(Y))\hookrightarrow\bigoplus\limits_{Y\in\mathcal{Y}_\lambda}\rho(r(Y))$$
Now let $mv:=mp_vi_v$, $m(XY):=mp_{s(X)}i_X[\rho(\lambda)]p_Y i_{r(Y)}$, and $m(YX)^\ast:=mp_{r(Y)}i_Y[\rho(\lambda)]^\ast p_X i_{s(X)}$. To keep track of the last defining relations, we draw the following diagram:
\begin{center}
		\begin{tikzpicture}
	[->,>=stealth',shorten >=1pt,thick, scale=0.6]

	\path[->]	(-1,2)	edge[] 	node	{}		(1,2);
	\path[->]	(1,-2)	edge[] 	node	{}		(-1,-2);
	
	\path[->]	(-6.5,2)	edge[] 	node	{}		(-5,2);
	\path[->]	(-5,-2)	edge[] 	node	{}		(-6.5,-2);
	
	\path[->]	(5,2)	edge[] 	node	{}		(6.5,2);
	\path[->]	(6.5,-2)	edge[] 	node	{}		(5,-2);
	
	\path[->]	(-9,-2)	edge[bend left=40] 	node	{}		(-11,-0.5);
	\path[->]	(-11,0.5)	edge[bend left=40] 	node	{}		(-9,2);
	\path[->]	(9,2)	edge[bend left=40] 	node	{}		(11,0.5);
	\path[->]	(11,-0.5)	edge[bend left=40] 	node	{}		(9,-2);

	\node at ((0,2.5) {$[\lambda]$};
	\node at ((0.3,-2.5) {$[\lambda]^\ast$};
	\node at ((-3,1.8) {$\bigoplus\limits_{X\in\mathcal{X}_\lambda}\rho(s(X))$};
	\node at ((-3,-2.2) {$\bigoplus\limits_{X\in\mathcal{X}_\lambda}\rho(s(X))$};
	\node at ((3,1.8) {$\bigoplus\limits_{Y\in\mathcal{Y}_\lambda}\rho(r(Y))$};
	\node at ((3,-2.2) {$\bigoplus\limits_{Y\in\mathcal{Y}_\lambda}\rho(r(Y))$};
	\node at ((-11,0) {$\bigoplus\limits_{w\in E^0}\rho(w)$};
	\node at ((11,0) {$\bigoplus\limits_{w\in E^0}\rho(w)$};
	\node at ((-7.8,2) {$\rho(s(X))$};
	\node at ((-7.8,-2) {$\rho(s(X))$};
	\node at ((7.8,2) {$\rho(r(Y))$};
	\node at ((7.8,-2) {$\rho(r(Y))$};
	\node at ((-10.8,2) {$p_{s(X)}$};
	\node at ((-10.8,-2) {$i_{s(X)}$};
	\node at ((10.8,2) {$i_{r(Y)}$};
	\node at ((10.8,-2) {$p_{r(Y)}$};
	
	\node at ((-5.8,2.5) {$i_X$};
	\node at ((-5.8,-2.5) {$p_X$};
	\node at ((5.8,2.5) {$p_Y$};
	\node at ((5.8,-2.5) {$i_Y$};
	\end{tikzpicture}
\end{center}
Here the composition of the upper arrows correspond to right multiplication by $(XY)$ and the composition by lower arrows correspond to right multiplication by $(YX)^\ast$. Verifying that the above defining relation satisfy defining relations of $L$ is left to the reader.

Now we show that the above constructions give equivalance of categories. By Remark \ref{remark: M is sum of Mv}, we have $M_{\rho_M}=\bigoplus\limits_{v\in E^0}Mv=M$ and their $L$-module structures also match. Given a module homomorphism $\varphi: M\rightarrow N$, we have $\varphi=\bigoplus\limits_{v\in E^0}\varphi_v:\bigoplus\limits_{v\in E^0}Mv\rightarrow\bigoplus\limits_{v\in E^0}N$. 

For the composition in the other order $\rho_{M_\rho}(v)=M_\rho v=\left(\bigoplus\limits_{w\in E^0}\rho(w)\right)v=\rho(v)$ and $\rho(e)=\rho_{M_\rho}(e):M_\rho s(e)\rightarrow M_\rho r(e)$. For, let $e=(XY)$ for some $X\in\mathcal{X}_\lambda$ and $Y\in\mathcal{Y}_\lambda$, then the following diagram commutes.
\[ \begin{tikzcd}
M_\rho s(e)=\rho(s(e)) \arrow[hookrightarrow]{r}{i_{s(e)}} \arrow[swap]{d}{\rho(e)} & M_\rho=\bigoplus\limits_{w\in E^0}\rho(w) \arrow{d}{p_{s(e)}i_X[\rho(\lambda)]p_Y i_{r(e)}} \\%
M_\rho r(e)=\rho(r(e)) \arrow[hookrightarrow]{r}{i_{r(e)}}& M_\rho=\bigoplus\limits_{w\in E^0}\rho(w)
\end{tikzcd}
\]
Finally, for any homomorphism $\{\varphi_v:\rho(v)\rightarrow\sigma(v)\}_{v\in E^0}$ from $\rho$ to $\sigma$, the $v$-component of $\bigoplus\limits_{w\in E^0}\varphi_w$ is $\varphi_v:\rho_{M_\rho}(v)=\rho(v)\rightarrow\sigma_{M_\sigma}(v)=\sigma(v)$.
\end{proof}
\begin{remark}
We note that the full subcategory of graded quiver representations with respect to standard $\mathbb{Z}$-grading satisfying condition (H) is equivalent to the category of graded unital $L$-modules. The proof follows on similar lines of proof of Theorem \ref{theorem: condition H}.
\end{remark}

\begin{theorem}\label{theorem: retract}
The composition of the forgetful functor from $\mathfrak{M}_L$ to $\mathfrak{M}_E$ with $\underline{\hspace{6pt}}\otimes_{K(E)}L$ from $\mathfrak{M}_E$ to $\mathfrak{M}_L$ is naturally equivalent to the identity functor on $\mathfrak{M}_L$.
\end{theorem}
\begin{proof}
We note that both forgetful functor and $\underline{\hspace{6pt}}\otimes_{K(E)}L$ send unital modules to unital modules. 
Let the composition of forgetful functor with $\underline{\hspace{6pt}}\otimes_{K(E)}L$ be denoted by $\mathcal{F}$ and the identity functor on $\mathfrak{M}_L$ be denoted by $\mathcal{I}$. If $M$ is an $L$-module, the $L$-module homomorphism $M\otimes_{K(E)}L\rightarrow M$ given by $m\otimes a\mapsto ma$ defines an natural transformation from $\mathcal{F}$ to $\mathcal{I}$. To see that this is an isomorphism, we define its inverse $M\rightarrow M\otimes_{K(E)}L$ by $m\mapsto\sum\limits_{\substack{v\in E^0\\mv\neq 0}}m\otimes v$. Observe that this sum is finite since $M$ is unital. 

To check that the above inverse map defines an $L$-linear map, we need to check on generators. For every $w\in E^0$ and $m\in M$ we have $\sum mu\otimes v=m\otimes u=\left(\sum m\otimes v\right) u$, since $E^0$ is a set of orthogonal idempotents. For all $\lambda\in\Lambda$, $X\in\mathcal{X}_\lambda$, $Y\in\mathcal{Y}_\lambda$ and $m\in M$ we have 
\begin{eqnarray*}
   \sum m(XY)\otimes v & = & m(XY)\otimes r(Y)\quad\qquad\text{since}~ev=0~\text{iff}~r(e)\neq v \\
          & = & m(XY)\otimes \left(\sum\limits_{X^\prime\in\mathcal{X}_\lambda}(YX^\prime)^\ast(X^\prime Y)\right)\qquad\text{by}~(L2) \\
          & = & \sum\limits_{X^\prime\in\mathcal{X}_\lambda}m(XY)(YX^\prime)^\ast\otimes(X^\prime Y) \\
          & = &  ms(X)\otimes (XY)\qquad\qquad\text{by}~(L1)\\
          & = &  m\otimes (XY)\\
          & = &  \left(\sum m\otimes v\right)(XY).
  \end{eqnarray*}
Similarly $\sum m(YX)^\ast\otimes v=(\sum m\otimes v)(YX)^\ast$.

The composition $m\mapsto \sum m\otimes v\mapsto \sum mv=m$. Since elements of the form $m\otimes v$ with $m\in Mv$ generate $M\otimes L$ as an $L$-module and for such elements we have $m\otimes v\mapsto mv\mapsto mv\otimes v=m\otimes v$, the other composition is also identity.
\end{proof}

Recall that the \textbf{universal localization} $\Sigma^{-1}A$ of an algebra $A$ with respect to a set $\Sigma=\{\sigma:P_\sigma\rightarrow Q_\sigma\}$ of homomorphisms between finitely generated projective $A$-modules, is an initial object among algebra homomorphisms $f:A\rightarrow B$ such that $\sigma\otimes id_B:P_\sigma\otimes_AB\rightarrow Q_\sigma\otimes_AB$ is an isomorphism for every $\sigma\in\Sigma$.

\begin{theorem}\label{theorem: localization}
$L$ is the universal localization of $K(E)$ with respect to $$\Big\{\sigma_\lambda:\bigoplus\limits_{Y\in\mathcal{Y}_\lambda}(r(Y))K(E)\longrightarrow \bigoplus\limits_{X\in\mathcal{X}_\lambda}(s(X))K(E)\Big\}_{\lambda\in\Lambda},$$

$$(a_Y)_{Y\in\mathcal{Y}_\lambda}\xmapsto{\sigma_\lambda}\left(\sum\limits_{Y\in\mathcal{Y}_\lambda}(XY)a_Y\right)_{X\in\mathcal{X}_\lambda}.$$
\end{theorem}
\begin{proof}
Since $v\in E^0$ is an idempotent, the cyclic module $vK(E)$ is projective. For each $\lambda\in\Lambda$, $\sigma_\lambda\otimes id_L$ is an isomorphism with inverse $\sigma_\lambda^\ast$, where $(a_X)_{X\in\mathcal{X}_\lambda}\xmapsto{\sigma_\lambda^\ast}\left(\sum\limits_{X\in\mathcal{X}_\lambda}(YX)^\ast a_X\right)_{Y\in\mathcal{Y}_\lambda}$. If $f:K(E)\rightarrow B$ is an algebra homomorphism, then $f(v)^2=f(v)$ and $vK(E)\otimes_{K(E)}B\cong f(v)B$ by $a\otimes b\mapsto f(a)b$ and $b\mapsto v\otimes b$. 

Let $f:K(E)\rightarrow B$ be an algebra homomorphism such that $\sigma_\lambda\otimes id_B$ is an isomorphism for all $\lambda\in\Lambda$ then the composition $f(s(X))B\cong (s(X))K(E)\otimes_{K(E)}B\xhookrightarrow{i_{s(X)}\otimes id_B}\left(\bigoplus\limits_{X\in\mathcal{X}_\lambda}s(X)K(E)\right)\otimes_{K(E)}B\xrightarrow{\sigma_\lambda^\ast\otimes ~id_B}\left(\bigoplus\limits_{Y\in\mathcal{Y}_\lambda}r(Y)K(E)\right)\otimes_{K(E)}B\xrightarrow{p_{r(Y)}\otimes id_B}r(Y)K(E)\otimes_{K(E)}B\cong f(r(Y))B$ is uniquely and completely determined, which we call $f((YX)^\ast)$. Now $\Tilde{f}(v):=f(v)$ for all $v\in E^0$, $\Tilde{f}((XY)):=f((XY))$ for all $\lambda\in\Lambda$, $X\in\mathcal{X}_\lambda$ and $Y\in\mathcal{Y}_\lambda$ defines the unique homomorphism $\Tilde{f}:L\rightarrow B$ factoring $f$ through $K(E)\rightarrow L$. 
\end{proof}
\begin{proposition}\label{prop: dimension function can be realized}
Let $(\dot{E},\Lambda)$ be hypergraph. If $d:E^0\rightarrow\mathbb{N}\cup\{\infty\}$ satisfies $$\sum\limits_{X\in\mathcal{X}_\lambda}d(s(X))=\sum\limits_{Y\in\mathcal{Y}_\lambda}d(r(Y))~\text{for all}~\lambda\in\Lambda,$$ then there is an $L$-module $M$ with $\emph{dim}_K(Mv)=d(v)$.
\end{proposition}
\begin{proof}
Define the quiver representation $\rho$ by $\rho(v)=K^{d(v)}$ if $d(v)<\infty$ and $\rho(v)=K^{(\mathbb{N})}$ otherwise. Then by definition of $d$ we can find isomorphism $\theta_\lambda:\bigoplus\limits_{X\in\mathcal{X}_\lambda}\rho(s(X))\rightarrow\bigoplus\limits_{Y\in\mathcal{Y}-\lambda}\rho(r(Y))$ for all $\lambda\in\Lambda$. Let $\rho(XY):=i_{s(X)}\theta_\lambda p_{r(Y)}$ for all $\lambda\in\Lambda$, $X\in\mathcal{X}_\lambda$ and $Y\in\mathcal{Y}_\lambda$. Condition (H) is satisfied by construction and the corresponding $L$-module $M$ of Theorem \ref{theorem: condition H} has dim$_K(Mv)=\text{dim}_K\rho(v)=d(v)$.
\end{proof}

\begin{definition}\label{def: dimension function}
A \textbf{dimension function} of a hypergraph $(\dot{E},\Lambda)$ is a function $d:E^0\rightarrow\mathbb{N}$ satisfying $\sum\limits_{X\in\mathcal{X}_\lambda}s(X)=\sum\limits_{Y\in\mathcal{Y}_\lambda}r(Y)$ for all $\lambda\in\Lambda$.
\end{definition}
\begin{remark}\label{remark: L has f.d. rep if}
If the $L$-module $M$ is finitary, i.e, dim$(Mv)<\infty$ for all $v\in E^0$ then by Lemma \ref{lemma: unital module satisfies an isomorphism}, $d(v):=\text{dim}(Mv)$ is a dimension function. By Proposition \ref{prop: dimension function can be realized}, the converse also holds. that is, every dimension function is realizable. Moreover, since dim($M)=\sum\limits_{v\in E^0}\text{dim}(Mv)$, $d(v):=\text{dim}(Mv)$ has finite support if $M$ is finite dimensional.
\end{remark}

\subsection{Support subgraphs and the hypergraph monoid}

Let $\mathcal{H}=(\dot{E},\Lambda)$ be a hypergraph and $E^\prime$ be a full subgraph of $E$. Then there is a natural (hyper) bi-separation induced on $E^\prime$ from $\mathcal{H}$ as follows: For each $X\in C$, if $s(X)\in (E^\prime)^0$, define $X^\prime=X\bigcap (E^\prime)^1$ and similarly for each $Y\in D$, if $r(Y)\in(E^\prime)^0$, define $Y^\prime=Y\bigcap (E^\prime)^0$. Also for each $\lambda\in\Lambda$, define $\lambda^\prime$ using the following data: $\mathcal{X}_{\lambda^\prime}=\{X^\prime\mid X\in\mathcal{X}_\lambda\}$ and $\mathcal{Y}_{\lambda^\prime}=\{Y^\prime\mid Y\in\mathcal{Y}_\lambda\}$. Finally define $\Lambda^\prime=\{\lambda^\prime\mid \lambda\in\Lambda, \mathcal{X}_\lambda\neq\emptyset~\text{and}~\lambda\neq\emptyset\}$. We call $\mathcal{H}^\prime=(\dot{E^\prime},\Lambda^\prime)$ \textbf{full sub-hypergraph} of $\mathcal{H}$ (hyper-induced from $E^\prime$).

\begin{definition}
Let $\mathcal{H}=(\dot{E},\Lambda)$ be a hypergraph. A full sub-hypergraph $\mathcal{H}^\prime=(\dot{E^\prime},\Lambda^\prime)$ is called \textbf{co-bisaturated} if the following conditions are satisfied: 
For every $\lambda^\prime\in\Lambda^\prime$,
\begin{enumerate}
    \item if $s(X)\in (E^\prime)^0$, then $X\cap (E^\prime)^1\neq\emptyset$, where $X\in\mathcal{X}_{\lambda^\prime}$,
    \item if $r(Y)\in (E^\prime)^0$, then $Y\cap(E^\prime)^1\neq\emptyset$, where $Y\in\mathcal{Y}_{\lambda^\prime}$.
\end{enumerate}
\end{definition}

We note that a full sub-hypergraph $\dot{E^\prime}$ of $\dot{E}$ is a co-bisaturated if and only if $E^0-(E^\prime)^0$ is a bisaturated subset of $E^0$.

Let $\mathcal{H}=(\dot{E},\Lambda)$ be a hypergraph and $M$ be a right $L(\mathcal{H})$-module. The \textbf{support subgraph} of $M$, denoted by $E_M$, is the full subgraph of $E$ induced on $V_M:=\{v\in E^0\mid Mv\neq 0\}$. The hypergraph $\mathcal{H}_M=(\dot{E}_M,\Lambda_M)$, which is the full sub-hypergraph of $\mathcal{H}$ hyper-induced from $E_M$, is called  the \textbf{support sub-hypergraph} of $M$.
\begin{lemma}
Let $\mathcal{H}=(\dot{E},\Lambda)$ be a hypergraph and $\mathcal{H}^\prime=(\dot{E^\prime},\Lambda^\prime)$ be a full sub-hypergraph of $\mathcal{H}$. Then the following are equivalent:
\begin{enumerate}
    \item $\mathcal{H}^\prime=\mathcal{H}_M$, is the support sub-hypergraph of a unital $L(\mathcal{H})$-module $M$.
    \item $\mathcal{H}^\prime$ is co-bisaturated.
    \item The map $\theta:L(\mathcal{H})\rightarrow L(\mathcal{H}^\prime)$ defined (on generators) by
    $$\theta(x):=\begin{cases}
    x\quad\text{if}~x\in (E^\prime)^0\sqcup (E^\prime)^1\sqcup \overline{(E^\prime)^1},\\
    0\quad\text{otherwise}\end{cases}$$
    extends to an onto algebra homomorphism.
\end{enumerate}
\end{lemma}
\begin{proof}
$(1)\Rightarrow(2)$: Let $\lambda^\prime\in\Lambda^\prime$ and $X\in\mathcal{X}_{\lambda^\prime}$. Assume $s(X)\in(E_M)^0$, then $0\neq Ms(X)\hookrightarrow\bigoplus\limits_{X\in\mathcal{X}_\lambda}Ms(X)\cong\bigoplus\limits_{Y\in\mathcal{Y}_\lambda}Mr(Y)$ implies that there exists $Y\in\mathcal{Y}_\lambda$ such that $X\cap Y\neq\emptyset$ which is equivalent to $X\cap(E^\prime)^1\neq\emptyset$. Similarly, if $Y\in\mathcal{Y}_{\lambda^\prime}$ and $r(Y)\in(E_M)^0$, then $Y\cap(E^\prime)^1\neq\emptyset$.

$(2)\Rightarrow(3)$: We check that $\theta$ preserves the defining relations of $L(\mathcal{H})$. It is direct that path algebra relations are satisfied. Let $\lambda^\prime\in\Lambda^\prime$, $X_1,X_2\in\mathcal{X}_{\lambda^\prime}$ and $Y_1,Y_2\in\mathcal{Y}_{\lambda^\prime}$. If $s(X_i), r(Y_i)\in(E^\prime)^0$ then $X_i\cap(E^\prime)^1\neq\emptyset$ and $Y_1\cap(E^\prime)^1\neq\emptyset$. Hence the image of $\sum\limits_{X\in\mathcal{X}_{\lambda}}(Y_1X)^\ast(XY_2)=\delta_{Y_1Y_2}r(Y)$ is $\sum\limits_{X\in\mathcal{X}_{\lambda^\prime}}(Y_1X)^\ast(XY_2)=\delta_{Y_1Y_2}r(Y)$. Similarly, the image of $\sum\limits_{Y\in\mathcal{Y}_{\lambda}}(X_1Y)(YX_2)^\ast=\delta_{X_1X_2}s(X)$ is $\sum\limits_{Y\in\mathcal{Y}_{\lambda^\prime}}(X_1Y)(YX_2)^\ast=\delta_{X_1X_2}s(X)$.

$(3)\Rightarrow(1)$: Let $M:=L(\mathcal{H}^\prime)\cong L(\mathcal{H})/\text{Ker}\theta$. Now $v\in(E^\prime)^0$ if and only if $\theta(v)\neq 0$ and $Mv=L(\mathcal{H}^\prime)v\neq0$. Hence the vertex set of $E_M$ is $(E^\prime)^0$. It is routine to check that $\mathcal{H}^\prime$ is full sub-hypergraph and hence $\mathcal{H}_M=\mathcal{H}^\prime$. 
\end{proof}
\begin{proposition}
If $M$ is a unital $L(\mathcal{H})$-module then $M$ also has the structure of a unital $L(\mathcal{H}_M)$-module induced through the epimorphism $\theta:L(\mathcal{H})\rightarrow L(\mathcal{H}_M)$. 
Moreover, $\emph{Ker}\theta$ is generated by $E^0-V_M=\{v\in E^0\mid Mv=0\}$ and  $\emph{Ker}\theta\subseteq\emph{Ann}M$.
\end{proposition}
\begin{proof}
Let $\rho_M$ be the quiver representation of $E$ corresponding to $M$ as defined in the proof of Theorem \ref{theorem: condition H}. We claim that the restriction of $\rho_M$ to $E_M$ satisfies (H). That is, if $\rho^\prime:=\rho_M|_{E_M}$, then for all $\lambda^\prime\in\Lambda_M$,
$$[\rho^\prime]:\bigoplus\limits_{X\in\mathcal{X}_{\lambda^\prime}}\rho^\prime(s(X))\rightarrow\bigoplus\limits_{Y\in\mathcal{Y}_{\lambda^\prime}}\rho^\prime(r(Y))~\text{is an isomorphism}.$$

For, 
\begin{eqnarray*}
\bigoplus\limits_{X\in\mathcal{X}_{\lambda^\prime}}\rho^\prime(s(X)) & = & \bigoplus\limits_{X\in\mathcal{X}_{\lambda^\prime}}Ms(X)\\
& = & \bigoplus\limits_{X\in\mathcal{X}_\lambda}Ms(X)\quad\text{since}~Ms(X)=0~\text{for}~X\not\in\mathcal{X}_{\lambda^\prime}\\
& \cong & \bigoplus\limits_{Y\in\mathcal{Y}_\lambda}Mr(Y)\\
& = & \bigoplus\limits_{Y\in\mathcal{Y}_{\lambda^\prime}}Mr(Y)\quad\text{since}~Mr(Y)=0~\text{for}~Y\not\in\mathcal{Y}_{\lambda^\prime} \\
& = & \bigoplus\limits_{Y\in\mathcal{Y}_{\lambda^\prime}}\rho^\prime(r(Y)).
\end{eqnarray*}
Let $M^\prime$ be the unital $L(\mathcal{H}_M)$-module corresponding to $\rho^\prime$. Now $M^\prime$ is also an $L(\mathcal{H})$-module via $\theta:L(\mathcal{H})\rightarrow L(\mathcal{H}_M)$. As vector spaces $M^\prime=\bigoplus\limits_{v\in V_M}Mv\cong\bigoplus\limits{v\in E^0}Mv=M$. We can define an $L(\mathcal{H}_M)$-module structure on $M$ via this isomorphism. The action of the generators on $M$ and $M^\prime$ is compatible with this isomorphism, so $M\cong M^\prime$ as $L(\mathcal{H})$-modules. Thus the $L(\mathcal{H})$-module structure of $M$ is induced from the $L(\mathcal{H}_M)$-module structure via $\theta$.

For the second part, Let $I_M$ be the ideal generated by $E^0-V_M$. We show that $L(\mathcal{H}_M)\cong L(\mathcal{H})/\text{Ker}\theta$ and $L(\mathcal{H})/I_M$ are isomorphic.  Since, $E^0-V_M\subseteq \text{Ker}\theta$, we have a surjection from $L(\mathcal{H})/I_M$ to $L(\mathcal{H}_M)$. Let $\varphi:L(\mathcal{H}_M)\rightarrow L(\mathcal{H})/I$ be defined on generators by $x\mapsto x+I$ where $x\in E^0\sqcup E^1\sqcup \overline{E^1}$. It is not hard to show that $\varphi$ is a homomorphism and the inverse of the above surjection. 
\end{proof}
Recall that given a hypergraph $\mathcal{H}=(\Dot{E},\Lambda)$, its hypergraph monoid $H(\mathcal{H})$ is defined as the additive monoid generated by 
$E^0$ modulo the following relations:
$$\sum\limits_{X\in\mathcal{X}_\lambda}s(X)=\sum\limits_{Y\in \mathcal{Y}_\lambda}r(Y)\quad\text{for all}~\lambda\in\Lambda.$$

Therefore, dimension functions of $\mathcal{H}$ correspond exactly to monoid homomorphisms from $H(\mathcal{H})$ to $\mathbb{N}$. 

Since $H(\mathcal{H})$ is isomorphic to the monoid $\mathcal{V}(L(\mathcal{H}))$, the generator $v$ of $H(\mathcal{H})$ corresponds to the (right) projective $L(\mathcal{H})$-module $vL(\mathcal{H})$. The corresponding relations among the isomorphism classes of the cyclic projective modules was shown to hold in the proof of Theorem \ref{theorem: localization}. We can now reinterpret the existence of a nonzero finite dimensional represenatation in terms of the nonstable K-theory of $L(\mathcal{H})$.

\begin{theorem}
$L(\mathcal{H})$ has a nonzero finite dimensional representation if and only if $\dot{E}$ has a finite, full co-bisaturated sub-hypergraph $\mathcal{G}$ with a nonzero monoid homomorphism from $\mathcal{V}(L(\mathcal{G}))$ to $\mathbb{N}$.
\end{theorem}
\begin{proof}
By Remark \ref{remark: L has f.d. rep if}, $L(\mathcal{H})$ has a nonzero finite dimensional representation if and only if $\mathcal{H}$ has a nonzero dimension function of finite support. The support of such dimension function defines a finite, full, co-bisaturated sub-hypergraph $\mathcal{G}$ and its restriction gives a nonzero dimension function on $\mathcal{G}$ and thus a nonzero monoid homomorphism from $\mathcal{V}(L(\mathcal{G}))$ to $\mathbb{N}$ as well. Conversely, since $\mathcal{G}$ is co-bisaturated, any nonzero dimension function on $\mathcal{G}$ can be extended by $0$ to a dimension function on $\mathcal{H}$ and this gives a nonzero dimension function of finite support on $\mathcal{H}$.
\end{proof}

\section{Some remarks on Cohn-Leavitt path algebras of B-hypergraphs with Invariant Basis Number}\label{sec: lpa of hypergraphs and IBN}
Let $\mathcal{H}=(\dot{E},\Lambda)$ be a finite B-hypergraph and let $H:=H(\mathcal{H})$ be the $H$-monoid of $\mathcal{H}$. Let the Cohn-Leavitt path algebra $\mathcal{A}_K(\mathcal{H})$ be denoted simply by $L$ and its Grothendieck group by $\mathcal{K}_0(L)$. Let $\mathcal{U}(L)$ denote the submonoid of the $\mathcal{V}$-monoid $\mathcal{V}(L)$  generated by the element $[L] \in \mathcal{V}(L)$. Then $L$ has IBN property if and only if $\mathcal{U}(L)$ has infinite order. Now suppose $G(\mathcal{U}(L))$ denotes the Grothendieck group of $\mathcal{U}(L)$. Then one can show that the natural map $G(\mathcal{U}(L)) \rightarrow \mathcal{K}_0(L)$ induced by the inclusion $\mathcal{U}(L) \hookrightarrow \mathcal{V}(L)$ is an embedding (see \cite[Proposition 7]{MR3954639}). So $[L]$, treated as an element in the group $\mathcal{K}_0(L)$, has infinite order. This means that the element $[L] \otimes 1$ in $\mathcal{K}_0(L) \otimes \mathbb{Q}$ is nonzero. We know that $[L] \in \mathcal{V}(L)$ corresponds to the element $[\sum_{v \in E^0}v] \in H$ under the isomorphism of functors proved in Theorem \ref{thm: V-monoid_H-monoid_Isomorphism}. So, if $G(H)$ denotes the Grothendieck group of $H$, from the above arguments we can conclude that $L$ has IBN if and only if $\sum_{v \in E^0}v$ is nonzero in $G(H) \otimes \mathbb{Q}$, which is equivalent saying that $\sum_{v \in E^0}v$ is not in the $\mathbb{Q}-$ linear span of the elements of $R$ in $\mathbb{Q}\Omega$ (cf. \cite[Theorem 13]{MR3954639}), where  $\Omega$ is the set $E^0 \sqcup Q \sqcup P$ ($Q$ and $P$ are as in the Definition \ref{definition: H-monoid}) and
\begin{eqnarray*}
R &:=& \bigcup\limits_{\lambda \in \Lambda_S^T} \left[ \sum\limits_{X \in \mathcal{X}_{\lambda}}s(X) - \sum\limits_{Y \in \mathcal{Y}_{\lambda}}r(Y)\right]\\ &\bigsqcup& \bigcup\limits_{\lambda \in \Lambda_T^\text{fin}} \left[ \sum\limits_{X \in \mathcal{X}_{\lambda}}s(X) - \sum\limits_{Y \in \mathcal{Y}_{\lambda}}r(Y) - q_{\mathcal{Y}_{\lambda}}\right]\\
&\bigsqcup& \bigcup\limits_{\lambda \in \Lambda_S^\text{fin}} \left[\sum\limits_{Y \in \mathcal{Y}_{\lambda}}r(Y)-\sum\limits_{X \in \mathcal{X}_{\lambda}}s(X) -p_{\mathcal{X}_{\lambda}}\right]. 
\end{eqnarray*} 

\subsection{A Matrix criterion for Leavitt path algebra of a finite hypergraph having IBN}\hfill\\
In this subsection we generalize the main result of \cite[Section 3]{MR3955043}. Let $(\dot{E},\Lambda)$ be a finite hypergraph such that $|\Lambda| = h$ and $E^0 = n$. Then by theorem \ref{thm: V-monoid_H-monoid_Isomorphism}, the $\mathcal{V}$-monoid of $L_K(\dot{E})$ is generated by the set $E^0$ modulo  $h$ relations of the form
\begin{equation}\label{hypergraph_monoid_relations}
    \sum_{t=1}^{m}s(X_t) = \sum_{u=1}^{l}r(Y_u),
\end{equation}
one corresponding to each element of $\Lambda$. Let $A$ and $B$ be the coefficient matrices corresponding to the LHS and RHS respectively of the $h$ relations \eqref{hypergraph_monoid_relations}. Then it is clear that both $A$ and $B$ are $h \times n$ matrices with entries as non-negative integers. Let $T$ be a free abelian monoid on the set of all vertices. For each element $x \in T$, and for each $i$ such that $1 \leq i \leq h$, let $M_i(x)$ denote the element of $T$ which results by applying to $x$ the relation \eqref{hypergraph_monoid_relations} corresponding to the element $\lambda_i \in \Lambda$. For any sequence $\sigma$ taken from the set $\{1,2,\dots,h\}$, and any $x \in T$, let $\Delta_{\sigma}(x) \in T$ be the element obtained by applying $M_i$ operations in the order specified by $\sigma$. 

\begin{definition}
 Suppose for each pair $x,y \in T$, $[x] = [y]$ in the $\mathcal{V}$-monoid if and only if there are two sequences $\sigma$ and $\sigma^{\prime}$ taken from the set $\{1,2,\dots,h\}$ such that $\Delta_{\sigma}(x) = \Delta_{\sigma^{\prime}}(y)$ in $T$. Then we say that the \textit{confluence condition} holds in $T$.
\end{definition}

\begin{theorem}\label{IBN_for_hypergraphs}
 Let $(\dot{E},\Lambda)$ be a finite hypergraph such that $|\Lambda| = h$ and $E^0 = n$. Suppose $A$ and $B$ are the coefficient matrices corresponding to LHS and RHS respectively of the $h$ relations of the $\mathcal{V}$-monoid of $L_{K}(\dot{E})$. Also suppose that the confluence condition holds in $T$, the free abelian monoid on $E^0$. Then $L_{K}(\dot{E})$ has Invariant Basis Number if and only if $\text{rank}(B^t-A^t) < \text{rank}([B^t-A^t ~ c])$, where $c$ is the column matrix of order $n \times 1$ with all its entries equal to $1$.
 \begin{proof}
  Suppose that $\text{rank}(B^t-A^t) < \text{rank}(B^t-A^t~c )$. We prove that if $m$ and $p$ are positive integers such that
  \begin{equation}\label{IBN_equation}
      m [\sum_{i=1}^nv_i] = p[\sum_{i=1}^nv_i]
  \end{equation}
  in the $\mathcal{V}$-monoid,then $m=p$. So let us assume that the equation \eqref{IBN_equation} holds for some positive integers $m$ and $p$. Since the confluence condition holds in $T$, there are two sequences $\sigma$ and $\sigma^{\prime}$ taken from $\{1,2,\dots,h\}$ such that $\Delta_{\sigma}(m\sum_{i=1}^nv_i) = \Delta_{\sigma^{\prime}}(p\sum_{i=1}^nv_i) = \gamma (\text{say})$ in $T$. Now suppose $M_j$ is invoked $k_j$ times in $\Delta_{\sigma}$ and $k_{j}^{\prime}$ times in $\Delta_{\sigma^{\prime}}$. Then we have 
  \begin{eqnarray*}
   \gamma & = & \Delta_{\sigma}(m\sum_{i=1}^nv_i) \nonumber \\
          & = & [m+k_1(b_{11}-a_{11})+\dots+k_h{(b_{h1}-a_{h1})}]v_1 \nonumber \\
          & + & [m+k_1(b_{12}-a_{12})+\dots+k_h{(b_{h2}-a_{h2})}]v_2 \nonumber \\
          & + & \dots \dots \dots \nonumber \\
          & + & [m+k_1(b_{1z}-a_{1z})+\dots+k_h{(b_{hz}-a_{hz})}]v_z \nonumber \\
          & + & [m + k_1(b_{1(z+1)})+\dots+k_h(b_{h(z+1)})]v_{z+1} \nonumber \\
          & + & \dots \dots \dots \nonumber \\
          & + & [m + k_1(b_{1n})+\dots+k_h(b_{hn})]v_n.
  \end{eqnarray*}
  Also
  \begin{eqnarray*}
   \gamma & = & \Delta_{\sigma^{\prime}}(p\sum_{i=1}^nv_i) \nonumber \\
          & = & [p+k_1^{\prime}(b_{11}-a_{11})+\dots+k_h^{\prime}{(b_{h1}-a_{h1})}]v_1 \nonumber \\
          & + & [p+k_1^{\prime}(b_{12}-a_{12})+\dots+k_h^{\prime}{(b_{h2}-a_{h2})}]v_2 \nonumber \\
          & + & \dots \dots \dots \nonumber \\
          & + & [p+k_1^{\prime}(b_{1z}-a_{1z})+\dots+k_h^{\prime}{(b_{hz}-a_{hz})}]v_z \nonumber \\
          & + & [p + k_1^{\prime}(b_{1(z+1)})+\dots+k_h^{\prime}(b_{h(z+1)})]v_{z+1} \nonumber \\
          & + & \dots \dots \dots \nonumber \\
          & + & [p + k_1^{\prime}(b_{1n})+\dots+k_h^{\prime}(b_{hn})]v_n.
  \end{eqnarray*}
  Let $m_i=(k_i^{\prime}-k_i)$ for $i=1,\dots,h$. From the above two equations, we have the following system of equations-
  \begin{eqnarray*}
   (m-p) & = & m_1(b_{11}-a_{11})+\dots+m_h(b_{h1}-a_{h1}) \nonumber \\
   (m-p) & = & m_1(b_{12}-a_{12})+\dots+m_h(b_{h2}-a_{h2})\nonumber \\
   & \vdots & \nonumber \\
   (m-p) & = & m_1(b_{1z}-a_{1z})+\dots+m_h(b_{hz}-a_{hz})\nonumber \\
   & \vdots & \nonumber \\
   (m-p) & = & m_1(b_{1n}-a_{1n})+\dots+m_h(b_{hn}-a_{hn}).
  \end{eqnarray*}
  So $(m_1,\dots,m_h) \in \mathbb{Z}^h$ is a solution of the linear system $(B^t-A^t)x = (m-p)c$, where $x = (x_1,\dots,x_h)^t$ and $c$ is the column matrix mentioned in the statement of the theorem. This means $\text{rank}(B^t-A^t) = \text{rank}(B^t-A^t~~(m-p)c)$. We know that if $m \neq p$, then $\text{rank}(B^t-A^t~~(m-p)c) = \text{rank}(B^t-A^t~~c)$. This would mean that $\text{rank}(B^t-A^t) = \text{rank}(B^t-A^t~~c)$ whenever $m \neq p$, contrary to our initial assumption. This proves the first part.
  
  \noindent Conversely, assume that $\text{rank}(B^t-A^t) = \text{rank}(B^t-A^t~~c):=r$. We will prove that there exists a pair of distinct positive integers $m$ and $p$ such that $m[\sum_{i=1}^nv_i] = p[\sum_{i=1}^nv_i]$ in the $\mathcal{V}$-monoid of $\mathcal{A}_K(\dot{E})$. 
  
  \noindent The fact that $\text{rank}(B^t-A^t~~c) = r$ means that after finite number of elementary row operations, $(B^t-A^t~~c)$ can be brought to the form 
 \arraycolsep=3pt\medmuskip=1mu
\setcounter{MaxMatrixCols}{12}
\[\begin{pmatrix*}
0 & \dots & d_{1j_1} & \dots & d_{1j_2-1} & 0 & d_{1j_2+1} & \dots & d_{1j_r-1} & 0 & \dots & c_1 \\
0 & \dots & 0 & \dots & 0 & d_{2j_2} & d_{2j_2+1} & \dots & d_{2j_r-1} & 0 & \dots & c_2\\
\vdots & \vdots & \vdots & \vdots & \vdots &  \vdots & \dots & \vdots & \vdots & \vdots & \dots & \vdots \\
0 & \dots & \dots & \dots & 0 & \dots & \dots & 0 & 0 & d_{rj_r} & \dots & c_r \\
0 & \dots & \dots & \dots & 0 & \dots & \dots & 0 & 0 & 0 & \dots & 0\\
\vdots & \vdots & \vdots & \vdots & \vdots &  \vdots & \dots & \vdots & \vdots & \vdots & \dots & \vdots \\
0 & \dots & \dots & \dots & 0 & \dots & \dots & 0 & 0 & 0 & \dots & 0\\
\end{pmatrix*}\]
where the entries are integers, $d_{1j_1}d_{2j_2}\dots d_{rj_r} \neq 0$ and $\sum_{i=1}^rc_i^2 \neq 0$.
So it is clear that one particular solution for the linear system $(B^t-A^t)x = c$ is the column vector $(\frac{c_1}{d_{1j_1}},\frac{c_2}{d_{2j_2}},\dots,\frac{c_r}{d_{rj_r}},0,\dots,0)^t$.

\noindent Now let  
\begin{gather*}
  m_j :=
 \begin{cases} \frac{c_i|d_{1j_1}d_{2j_2}\dots d_{rj_r}|}{d_{ij_i}} & \text{if}~ j=j_i~(1 \leq i \leq r)\\
 0 & \text{otherwise},
\end{cases}    
\end{gather*}
$$p:= \text{max}\{|m_j| \mid j = 1,\dots,h\},\qquad
m:= |d_{1j_1}d_{2j_2}\dots d_{rj_r}| + p~ \text{and}$$
\begin{gather*}
(k_j^{\prime},k_j) := \begin{cases} (0,0) & \text{if}~m_j =0\\
(m_j,0) & \text{if}~m_j > 0\\
(0,-m_j) & \text{if}~m_j < 0.
\end{cases}
\end{gather*}

\noindent From the above definitions, it is clear that $(m-p) > 0$. So the $h$-tuple $(m_1,m_2,\dots,m_h)$ is a solution for the linear system $(B^t-A^t)x = (m-p)c$. This, from the first part of the proof, is equivalent to showing that $m[\sum_{i=1}^nv_i] = p[\sum_{i=1}^nv_i]$. This means that $\mathcal{A}_K(\dot{E})$ does not have Invariant Basis Number, thereby completing the proof.
 \end{proof}
\end{theorem}

\section*{Acknowledgement}
The authors would like to thank Ramesh Sreekantan, Roozbeh Hazrat and Pere Ara  for their valuable suggestions and comments.

The first named author gratefully acknowledges the Department of Atomic Energy (DAE), India for financial support through PhD scholarship.

The second named author was a visitor of Indian Statistical Institute Bangalore Center during the preparation of this manuscript. He would like to thank the institute for providing conducive environment for research. He thanks B. Venkatesh for his kind support.
He also gratefully acknowledges the Department of Atomic Energy (DAE), India for financial support through Postdoctoral Fellowship.

\end{document}